\documentclass[10pt, a4paper]{article}%
\usepackage{amsmath,amssymb,eucal}
\usepackage{amsfonts}
\usepackage{mathrsfs}
\usepackage{slashed}
\usepackage{graphicx}
\usepackage{ifpdf}
\usepackage{url}
\numberwithin{equation}{section} 
\def\beq{\begin{equation}}
\def\eeq{\end{equation}}

\def\bC{ {{\mathbb{C}}}}
\def\bR{ {{\mathbb{R}}}}

\def\Tr{ {{\rm{Tr}}} }

\newcommand{\pk}[1]{p_{\kappa}}
\newcommand{\ba}{\mathbf{a}}

\newcommand{\ol}{\overline{l}}
\newcommand{\ul}{\underline{l}}
\newcommand{\oL}{\overline{L}}
\newcommand{\uL}{\underline{L}}
\newcommand{\on}{\overline{n}}
\newcommand{\un}{\underline{n}}

\newcommand{\bk}{\breve{\kappa}}

\newtheorem{defn}{{\bf Definition}}[section]

\newtheorem{cor}[defn]{{\bf Corollary}}
\newtheorem{lem}[defn]{{\bf Lemma}}
\newtheorem{prop}[defn]{{\bf Proposition}}
\newtheorem{rem}[defn]{{\bf Remark}}
\newtheorem{example}[defn]{Example}

\newtheorem{notation}[defn]{Notation}
\newenvironment{proof}[1][Proof]{\textbf{#1.} }{\hfill \rule{0.5em}{0.5em}}
\begin{document}

\title{Einstein-Hilbert Path Integrals in $\bR^4$}
\author{Adrian P. C. Lim \\
Email: ppcube@gmail.com
}

\date{}

\maketitle

\begin{abstract}
A hyperlink is a finite set of non-intersecting simple closed curves in $\mathbb{R} \times \mathbb{R}^3$. The dynamical variables in General Relativity are the vierbein $e$ and a $\mathfrak{su}(2)\times\mathfrak{su}(2)$-valued connection $\omega$. Together with Minkowski metric, $e$ will define a metric $g$ on the manifold.

The Einstein-Hilbert action $S(e,\omega)$ is defined using $e$ and $\omega$. We will define a path integral $I$ by integrating a functional $H(e,\omega)$ against a holonomy operator of a hyperlink $L$, and the exponential of the Einstein-Hilbert action, over the space of vierbeins $e$ and $\mathfrak{su}(2)\times\mathfrak{su}(2)$-valued connections $\omega$.

Three different types of functional will be considered for $H$, namely area of a surface, volume of a region and the curvature of a surface $S$. Using our earlier work done on Chern-Simons path integrals in $\mathbb{R}^3$, we will derive and write these infinite dimensional path integrals $I$ as the limit of a sequence of Chern-Simons integrals.
\end{abstract}

\hspace{.35cm}{\small {\bf MSC} 2010: } 83C45, 81S40, 81T45, 57R56 \\
\indent \hspace{.35cm}{\small {\bf Keywords}: Area, Volume, Curvature, Path integral, Einstein-Hilbert, Quantum gravity}




\section{Einstein-Hilbert Action}\label{s.pre}

Consider a 3-manifold $M$, hence a 4-manifold $\bR \times M$, and a principal bundle $V$ over $\bR \times M$, with structure group $G$. Let $\mathfrak{g}$ be the Lie Algebra of $G$. The vector space of all smooth $\mathfrak{g}$-valued one forms on the manifold $\bR \times M$ will be denoted by $\mathcal{A}_{\bR \times M, \mathfrak{g}}$. Denote the group of all smooth $G$-valued mappings on $\bR \times M$ by $\mathcal{G}$, called the gauge group. The gauge group induces a gauge transformation on $\mathcal{A}_{\bR \times M, \mathfrak{g}}$,  $\mathcal{A}_{\bR \times M, \mathfrak{g}} \times \mathcal{G} \rightarrow \mathcal{A}_{\bR \times M, \mathfrak{g}}$, given by \beq A \cdot \Omega \equiv A^{\Omega} := \Omega^{-1}d\Omega + \Omega^{-1}A\Omega \nonumber \eeq for $A \in \mathcal{A}_{\bR \times M, \mathfrak{g}}$, $\Omega \in \mathcal{G}$. The orbit of an element $A \in \mathcal{A}_{\bR \times M, \mathfrak{g}}$ under this operation will be denoted by $[A]$ and the set of all orbits by $\mathcal{A}/\mathcal{G}$.

The 4-manifold we will consider in this article is $\bR \times \bR^3 \equiv \bR^4$, with tangent bundle $T\bR^4$. The tangent-space indices are denoted by $a, b, c, d$ and `Lorentz' indices as $\mu, \gamma, \alpha, \beta$, both taking values from $\{0, 1, 2, 3\}$.

Let $V \rightarrow \bR \times \bR^3$ be a 4-dimensional vector bundle, with structure group $SO(3,1)$. This implies that $V$ is endowed with a metric, $\eta^{ab}$, of signature $(-, +, +, +)$, and a volume form $\epsilon_{abcd}$.

\begin{notation}
Fix the standard coordinates on $\bR^4\equiv \bR \times \bR^3 $, with time coordinate $x_0$ and spatial coordinates $(x_1, x_2, x_3)$.

Let $\Lambda^q(\bR^4)$ denote the q-th exterior power of $\bR^4$ and we choose the canonical basis \newline
$\{dx_0, dx_1, dx_2, dx_3\}$ for $\Lambda^1(\bR^4)$. Using the standard coordinates on $\bR^4$, let $\Lambda^1(\bR^3)$ denote the subspace in $\Lambda^1(\bR^4)$ spanned by $\{dx_1, dx_2, dx_3\}$. Finally, a basis for $\Lambda^2(\bR^4)$ is given by \beq \{dx_0 \wedge dx_1, dx_0 \wedge dx_2, dx_0 \wedge dx_3, dx_2\wedge dx_3, dx_3 \wedge dx_1, dx_1 \wedge dx_2\}. \nonumber \eeq

We adopt Einstein's summation convention, i.e. we sum over repeated superscripts and subscripts.
\end{notation}

Suppose $V$ has the same topological type as $T\bR^4$, so that isomorphisms between $V$ and $T\bR^4$ exist. Hence we may assume that $V$ is a trivial bundle over $\bR^4$. Without loss of generality, we will assume the Minkowski metric $\eta^{ab}$ is given by \beq \eta = -dx_0 \otimes dx_0 + \sum_{i=1}^3 dx_i \otimes dx_i. \nonumber \eeq And $\epsilon^{\mu \gamma \alpha \beta} \equiv \epsilon_{\mu \gamma \alpha \beta}$ is equal to 1 if the number of transpositions required to permute $(0123)$ to $(\mu\gamma\alpha\beta)$ is even; otherwise it takes the value -1.

However, there is no natural choice of an isomorphism. A vierbein $e$ is a choice of isomorphism between $T\bR^4$ and $V$. It may be regarded as a $V$-valued one form, obeying a certain condition of invertibility. A spin connection $\omega_{\mu \gamma}^a$ on $V$, is anti-symmetric in its indices $\mu$, $\gamma$. It takes values in $\Lambda^2 (V)$, whereby $\Lambda^k (V)$ denotes the $k$-th antisymmetric tensor power or exterior power of $V$. The isomorphism $e$ and the connection $w$ can be regarded as the dynamical variables of General Relativity.

The curvature tensor is defined as
\beq R_{\mu\gamma}^{ab} = \partial_a \omega_{\mu \gamma}^{b} - \partial_b \omega_{\mu \gamma}^{a} + [\omega^a, \omega^b]_{\mu\gamma},\ \partial_a \equiv \partial/\partial x_a, \nonumber \eeq
or as $R = d\omega + \omega \wedge \omega$. It can be regarded as a two form with values in $\Lambda^2 (V)$.

Using the above notations, the Einstein-Hilbert action is written as \beq S_{EH}(e, \omega) = \frac{1}{8}\int_{\bR^4}\epsilon^{\mu \gamma \alpha \beta}\epsilon_{abcd}\ e_\mu^{a}e_\gamma^b R_{\alpha\beta}^{cd}. \label{e.eh.2} \eeq The expression $e \wedge e \wedge R$ is a four form on $\bR^4$ taking values in $V \otimes V \otimes \Lambda^2(V)$ which maps to $\Lambda^4(V)$. But $V$ with the structure group $SO(3,1)$ has a natural volume form, so a section of $\Lambda^4(V)$ may be canonically regarded as a function. Thus Equation (\ref{e.eh.2}) is an invariantly defined integral.
By varying Equation (\ref{e.eh.2}) with respect to $e$, we will obtain the Einstein equations in vacuum. See \cite{Witten:1988hc}.

The metric $\eta^{ab}$ on $V$, together with the isomorphism $e$ between $T\bR^4$ and $V$, gives a (non-degenerate) metric $g^{ab} = e_\mu^a e_\gamma^b \eta^{\mu\gamma}$ on $T\bR^4$. By varying Equation (\ref{e.eh.2}) with respect to the connection $\omega$, we will obtain an equation that identifies $\omega$ as the Levi-Civita connection associated with the metric $g^{ab}$.

\section{Notations}

Throughout this article, $\sqrt{-1}$ will be denoted by $i$.

\begin{notation}\label{n.s.1}(Subspaces in $\bR^4$)\\
In this article, we will write $\bR^4 \equiv \bR \times \bR^3$, whereby $\bR$ will be referred to as the time-axis and $\bR^3$ is the spatial 3-dimensional Euclidean space. In future, when we write $\bR^3$, we refer to the spatial subspace in $\bR^4$. Let $\pi_0: \bR^4 \rightarrow \bR^3$ denote this projection.

Let $\{e_i\}_{i=1}^3$ be the standard basis in $\bR^3$. And $\Sigma_i$ is the plane in $\bR^3$, containing the origin, whose normal is given by $e_i$. So, $\Sigma_1$ is the $x_2-x_3$ plane, $\Sigma_2$ is the $x_3-x_1$ plane and finally $\Sigma_3$ is the $x_1-x_2$ plane.

Note that $\bR \times \Sigma_i \cong \bR^3$ is a 3-dimensional subspace in $\bR^4$. Here, we replace one of the axis in the spatial 3-dimensional Euclidean space with the time-axis. Let $\pi_i: \bR^4 \rightarrow \bR \times \Sigma_i$ denote this projection.
\end{notation}

\begin{notation}(Indices)\label{n.n.5}\\
In this article, the symbols are indexed by several indices. To make it easier for the reader to follow, we will reserve certain symbols for specific indices.

In the rest of the article, indices labeled $i, j, k$, $\bar{i}, \bar{j}, \bar{k}$  will only take values from 1 to 3. These indices will keep track of the spatial coordinate $x_i$.

Indices such as $a,b,c, d$ and greek indices such as $\mu,\gamma, \alpha, \beta$ will take values from 0 to 3. We will use the greek indices to index the basis in $\Lambda^2(V)$.

We will let $I = [0,1]$ be the unit interval, and \beq I^2 \equiv I \times I,\ \ I^3 = I \times I \times I,\ \ I^4 = I \times I \times I \times I. \nonumber \eeq We will let $s, \bar{s}, t, \bar{t}$ denote real numbers in $I$ and $\hat{s} = (s,\bar{s})$, $\hat{t} = (t,\bar{t})$. And $d\hat{s} \equiv ds d\bar{s}$, $d\hat{t} \equiv dt d\bar{t}$. Typically, $s, \bar{s}, t, \bar{t}$ will be reserved as the variable for some parametrization, i.e. $\vec{\rho}: s \in I \mapsto \bR^4$.

\end{notation}

\begin{notation}\label{n.s.2}(Symmetric group $S_3$)\\
Let $S_3$ denote the symmetric group on the set $\{1,2,3\}$. In this group, there is a cyclic subgroup, $C_3$ given by the set $C_3 = \{(1,2,3), (2,3,1), (3,1,2)\}$.
And $\Upsilon$ denote the set $\{(2,3), (3,1), (1,2)\}$.

Let $\tau: \{1,2,3\} \rightarrow \Upsilon$, by \beq \tau: 1 \mapsto (2,3),\ \ 2 \mapsto (3,1),\ \ 3 \mapsto (1,2). \nonumber \eeq

Let $\epsilon^{ijk} \equiv \epsilon_{ijk}$ be defined on the set $\{1,2,3\}$, by \beq \epsilon^{123} = \epsilon^{231} = \epsilon^{312} = 1,\ \ \epsilon^{213} = \epsilon^{321} = \epsilon^{132} = -1, \nonumber \eeq if $i,j,k$ are all distinct; 0 otherwise.
\end{notation}

\begin{notation}\label{n.v.1}(Vectors in $\bR^4$)\\
More often than not, given a symbol $p$, we will use $\vec{p} \equiv (p_0, p_1, p_2, p_3)$ to denote a 4-vector, $p \equiv (p_1, p_2, p_3)$ to denote a 3-vector and $\hat{p}$ to denote a 2-vector.

Suppose we have a vector $\sigma \equiv (\sigma_1, \sigma_2, \sigma_3) \in \bR^3$. We will write $\vec{\sigma} := (0, \sigma) \equiv (0, \sigma_1, \sigma_2, \sigma_3)$.

Write $x = (x_1, x_2, x_3)$. For $i= 1, 2, 3$, we will write \beq \hat{x}_i =
\left\{
  \begin{array}{ll}
    (x_2, x_3), & \hbox{$i=1$;} \\
    (x_1, x_3), & \hbox{$i=2$;} \\
    (x_1, x_2), & \hbox{$i=3$.}
  \end{array}
\right. \nonumber \eeq
\end{notation}

\begin{notation}(Representation of $\mathfrak{su}(2) \times \mathfrak{su}(2)$ )\label{n.su.1}\\
Let $V$ be a vector space of dimension 4. In the rest of this article, we take our principal bundle over $\bR^4$ to be trivial, i.e. $\bR^4 \times V \rightarrow \bR^4$ will be our trivial bundle in consideration. Fix a basis $\{E^\gamma\}$ in $V$. Write $E^{\gamma \mu}= E^\gamma \wedge E^\mu \in \Lambda^2(V)$, thus $\{E^{\gamma \mu}\}_{0\leq \gamma<\mu\leq 3}$ is a basis for $\Lambda^2(V)$.


Let $\mathfrak{su}(2)$ be the Lie Algebra of $SU(2)$. We can map $\Lambda^2(V)$ to the Lie Algebra $\mathfrak{su}(2) \times \mathfrak{su}(2)$ via a linear map. Let $\{\breve{e}_1, \breve{e}_2, \breve{e}_3\}$ be any basis for the first copy of $\mathfrak{su}(2)$ and
$\{\hat{e}_1, \hat{e}_2, \hat{e}_3\}$ be any basis for the second copy of $\mathfrak{su}(2)$, satisfying the conditions
\begin{align*}
[\breve{e}_1, \breve{e}_2] =& \breve{e}_3,\ \ [\breve{e}_2, \breve{e}_3] = \breve{e}_1,\ \ [\breve{e}_3, \breve{e}_1] = \breve{e}_2, \\
[\hat{e}_1, \hat{e}_2] =& \hat{e}_3,\ \ [\hat{e}_2, \hat{e}_3] = \hat{e}_1,\ \ [\hat{e}_3, \hat{e}_1] = \hat{e}_2.
\end{align*}

Let
\beq E^{01} \mapsto (\breve{e}_1,0) \equiv \hat{E}^{01},\ E^{02} \mapsto (\breve{e}_2,0) \equiv \hat{E}^{02},\ E^{03} \mapsto (\breve{e}_3,0) \equiv \hat{E}^{03} \nonumber \eeq and \beq E^{23} \mapsto (0,\hat{e}_1) \equiv \hat{E}^{23},\ E^{31} \mapsto (0,\hat{e}_2) \equiv \hat{E}^{31},\ E^{12} \mapsto (0,\hat{e}_3) \equiv \hat{E}^{12}. \nonumber \eeq Do note that $\hat{E}^{\alpha\beta} = -\hat{E}^{\beta\alpha} \in \mathfrak{su}(2) \times \mathfrak{su}(2)$. Refer to Notation \ref{n.s.2}. Now $\hat{E}^{\tau(1)} = \hat{E}^{23}$, $\hat{E}^{\tau(2)} = \hat{E}^{31}$ and $\hat{E}^{\tau(3)} = \hat{E}^{12}$.

This isomorphism
that sends $E^{\alpha\beta} \mapsto \hat{E}^{\alpha\beta}$ will be fixed throughout this article. Using the above basis, define \beq \mathcal{E}^+ = \sum_{i=1}^3\breve{e}_i\ ,\ \mathcal{E}^- = \sum_{i=1}^3\hat{e}_i, \nonumber \eeq and a $4 \times 4$ complex matrix
\beq \mathcal{E} =
\left(
  \begin{array}{cc}
    -\mathcal{E}^+ &\ 0 \\
    0 &\ \mathcal{E}^- \\
  \end{array}
\right). \nonumber \eeq

Write \beq \mathcal{F}^+ = \sum_{j=1}^3 \hat{E}^{0j} \equiv (\mathcal{E}^+, 0), \ {\rm and}\ \mathcal{F}^- = \sum_{j=1}^3 \hat{E}^{\tau(j)}\equiv (0, \mathcal{E}^-). \nonumber \eeq

For $A, B, C, D \in \mathfrak{su}(2)$, we define the Lie bracket on $\mathfrak{su}(2) \times \mathfrak{su}(2)$ as \beq [(A,B), (C,D)] = ([A,C], [B, D]) \in \mathfrak{su}(2) \times \mathfrak{su}(2). \nonumber \eeq

Let $\rho^\pm: \mathfrak{su}(2) \rightarrow {\rm End}(V^\pm)$ be an irreducible finite dimensional representation, indexed by half-integer and integer values $j_{\rho^\pm} \geq 0$. The representation $\rho: \mathfrak{su}(2) \times \mathfrak{su}(2) \rightarrow {\rm End}(V^+) \times {\rm End}(V^-)$ will be given by $\rho = (\rho^+, \rho^-)$, with \beq \rho: \alpha_i\hat{E}^{0i} + \beta_j \hat{E}^{\tau(j)} \mapsto \left(\sum_{i=1}^3\alpha_i \rho^+(\breve{e}_i) , \sum_{j=1}^3\beta_j \rho^-(\hat{e}_j) \right). \nonumber \eeq By abuse of notation, we will now write $\rho^+ \equiv (\rho^+, 0)$ and $\rho^- \equiv (0, \rho^-)$ in future and thus $\rho^+(\hat{E}^{0i}) \equiv \rho^+(\breve{e}_i)$, $\rho^-(\hat{E}^{\tau(j)}) \equiv \rho^-(\hat{e}_j)$. Also write \beq \rho^+(\mathcal{F}^+) = \sum_{j=1}^3 \rho^+(\hat{E}^{0j}), \ {\rm and}\ \rho^-(\mathcal{F}^-) = \sum_{j=1}^3 \rho^-(\hat{E}^{\tau(j)}). \nonumber \eeq

Note that the dimension of $V^\pm$ is given by $2j_{\rho^\pm} + 1$. Then it is known that the Casimir operator is
\begin{align*}
\sum_{i=1}^3 \rho^+(\hat{E}^{0i})\rho^+(\hat{E}^{0i}) \equiv \sum_{i=1}^3 \rho^+(\breve{e}_i) \rho^+(\breve{e}_i) = -\xi_{\rho^+} I_{\rho^+},\\
\sum_{i=1}^3 \rho^-(\hat{E}^{\tau(i)})\rho^-(\hat{E}^{\tau(i)}) \equiv \sum_{i=1}^3 \rho^-(\hat{e}_i) \rho^-(\hat{e}_i)= -\xi_{\rho^-} I_{\rho^-},
\end{align*}
$I_{\rho^\pm}$ is the $2j_{\rho^\pm} + 1$ identity operator for $V^\pm$ and $\xi_{\rho^\pm} := j_{\rho^\pm}(j_{\rho^\pm}+1)$.

Without loss of generality, we assume that $\rho^\pm(\hat{E})$ is skew-Hermitian for any $\hat{E} \in \mathfrak{su}(2)$, so by choosing a suitable basis in $V^+$, we will always assume that $\rho^+(i\mathcal{E}^+)$ is diagonal, with the real eigenvalues given by the set
\beq \left\{
       \begin{array}{ll}
         \{\pm \breve{\lambda}_1, \pm \breve{\lambda}_2, \cdots, \pm \breve{\lambda}_{(2j_{\rho^+}+1)/2}\}, & \hbox{$2j_{\rho^+} + 1$ is even;} \\
         \{\pm \breve{\lambda}_1, \pm \breve{\lambda}_2, \cdots, \pm \breve{\lambda}_{j_{\rho^+}}, 0\}, & \hbox{$2j_{\rho^+} + 1$ is odd.}
       \end{array}
     \right. \nonumber \eeq Similarly, by choosing another suitable basis in $V^-$, we will always assume that $\rho^-(i\mathcal{E}^-)$ is diagonal, with the real eigenvalues given by the set
\beq \left\{
       \begin{array}{ll}
         \{\pm \hat{\lambda}_1, \pm \hat{\lambda}_2, \cdots, \pm \hat{\lambda}_{(2j_{\rho^-}+1)/2}\}, & \hbox{$2j_{\rho^-} + 1$ is even;} \\
         \{\pm \hat{\lambda}_1, \pm \hat{\lambda}_2, \cdots, \pm \hat{\lambda}_{j_{\rho^-}}, 0\}, & \hbox{$2j_{\rho^-} + 1$ is odd.}
       \end{array}
     \right. \nonumber \eeq

Then, we have \beq \Tr\ \rho^+(e^{i\mathcal{E}^{+}}) =
\left\{
  \begin{array}{ll}
    \sum_{v=1}^{(2j_{\rho^+}+1)/2}2\cosh(\breve{\lambda}_v), & \hbox{$2j_{\rho^+} + 1$ is even;} \\
    1 + \sum_{v=1}^{j_{\rho^+}}2\cosh(\breve{\lambda}_v), & \hbox{$2j_{\rho^+} + 1$ is odd,}
  \end{array}
\right. \nonumber \eeq and
\beq \Tr\ \rho^-(e^{i\mathcal{E}^{-}}) =
\left\{
  \begin{array}{ll}
    \sum_{v=1}^{(2j_{\rho^-}+1)/2}2\cosh(\hat{\lambda}_v), & \hbox{$2j_{\rho^-} + 1$ is even;} \\
    1 + \sum_{v=1}^{j_{\rho^-}}2\cosh(\hat{\lambda}_v), & \hbox{$2j_{\rho^-} + 1$ is odd.}
  \end{array}
\right. \nonumber \eeq

In either case, we have $\Tr\ \rho^{\pm}(e^{i\mathcal{E}^{\pm}}) \geq 1$ and hence $\log\ \Tr\ \rho^{\pm}(e^{i\mathcal{E}^{\pm}}) \geq 0$, for any irreducible representation. Finally, note that $\Tr\ \rho^{\pm}(e^{i\mathcal{E}^{\pm}})$ is well-defined, even though $\mathcal{E}^\pm$ is not in general.
\end{notation}

\begin{rem}
By choosing the group $SU(2) \times SU(2)$, we actually define a spin structure on $\bR^4$.
\end{rem}

\begin{notation}\label{n.h.1}(On hyperlinks in $\bR^4$)\\
For a finite set of non-intersecting simple closed curves in $\bR^3$ or in $\bR \times \Sigma_i$, we will refer to it as a link. If it has only one component, then this link will be referred to as a knot. A simple closed curve in $\bR^4$ will be referred to as a loop. A finite set of non-intersecting loops in $\bR^4$ will be referred to as a hyperlink in this article. We say a link or hyperlink is oriented if we assign an orientation to its components.

Let $L$ be a hyperlink. We say $L$ is a time-like hyperlink, if given any 2 distinct points $p\equiv (x_0, x_1, x_2, x_3), q\equiv (y_0, y_1, y_2, y_3) \in L$, $p \neq q$, we have
\begin{itemize}
  \item $\sum_{i=1}^3(x_i - y_i)^2 > 0$;
  \item if there exists $i, j$, $i \neq j$ such that $x_i = y_i$ and $x_j = y_j$, then $x_0 - y_0 \neq 0$.
\end{itemize}
Throughout this article, all our hyperlinks in consideration will be time-like. We refer the reader to \cite{EH-Lim06} as to why the term time-like was used.

We will have 2 different hyperlinks, $\oL = \{\ol^u:\ u=1, \ldots, \on\}$ and $\uL = \{\ul^v:\ v=1, \ldots, \un\}$. The former will be called a matter hyperlink; the latter will be referred to as a geometric hyperlink. The symbols $u, \bar{u}, v, \bar{v}$ will be indices, taking values in $\mathbb{N}$. They will keep track of the loops in our hyperlinks $\oL$ and $\uL$. The symbols $\on$ and $\un$ will always refer to the number of components in $\oL$ and $\uL$ respectively.

Given a hyperlink $\oL$ and a hyperlink $\uL$, we also assume that together (by using ambient isotopy if necessary), they form another hyperlink with $\on + \un$ components. Denote this new hyperlink by $\chi(\oL, \uL) \equiv \chi(\{\ol^u\}_{u=1}^{\on}, \{\ul^v\}_{v=1}^{\un})$.

Color the matter hyperlink $\oL$, which means we choose a representation $\rho_u: \mathfrak{su}(2) \times \mathfrak{su}(2) \rightarrow {\rm End}(V_u^+) \times {\rm End}(V_u^-)$ for each component $\ol^u$, $u=1, \ldots, \on$, in the hyperlink $\oL$. Note that we do not color $\uL$, i.e. we do not choose a representation for $\uL$. Finally, we will also refer $\chi(\oL, \uL)$ as a colored hyperlink.

\end{notation}

\begin{notation}(Parametrization of curves)\label{n.n.3}\\
Let $\vec{y}^u \equiv (y_0^u, y_1^u, y_2^u, y_3^u) : [0,1] \rightarrow \bR \times \bR^3$ be a parametrization of a loop $\ol^u$, $u=1, \ldots, \on$. We will write $y^u(s) = (y_1^u(s), y_2^u(s), y_3^u(s))$ and $\vec{y}^u(s) \equiv \vec{y}_s^u$. We will also write $\vec{y}^u = (y^u_0, y^u)$. Similar notation for $\vec{\varrho}^{v}$, $v = 1, \ldots, \un$, which is a parametrization of a loop $\ul^v$. When the loop is oriented, we will choose a parametrization which is consistent with the assigned orientation.

Let $\rho: S \rightarrow \bR^3$ for some set $S$. Typically, $S = I, I^2$ or $I^3$. We can write $\rho(s) \equiv (\rho_1(s), \rho_2(s), \rho_3(s))$. Refer to Notation \ref{n.v.1}. We will write
\beq \hat{\rho}_i(s) =
 \left\{
  \begin{array}{ll}
    (\rho_2(s), \rho_3(s)), & \hbox{$i=1$;} \\
    (\rho_1(s), \rho_3(s)), & \hbox{$i=2$;} \\
    (\rho_1(s), \rho_2(s)), & \hbox{$i=3$.}
  \end{array}
\right. \nonumber \eeq

\end{notation}

\begin{notation}(Parametrization of surfaces)\label{n.s.3}
Choose an orientable, closed and bounded surface $S \subset \bR^4$, with or without boundary. If it has a boundary $\partial S$, then $\partial S$ is assumed to be a time-like hyperlink. Do note that we allow $S$ to be disconnected, with finite number of components. Parametrize it using \beq \vec{\sigma}: (t, \bar{t}) \in I^2 \mapsto \left(\sigma_0(t,\bar{t}), \sigma_1(t,\bar{t}), \sigma_2(t,\bar{t}), \sigma_3(t,\bar{t}) \right)\in  \bR^4 \nonumber \eeq and let \beq J_{\alpha\beta} = \frac{\partial \sigma_\alpha}{\partial t} \frac{\partial \sigma_\beta}{\partial \bar{t}} - \frac{\partial \sigma_\alpha}{\partial \bar{t}} \frac{\partial \sigma_\beta}{\partial t}. \nonumber \eeq

When we project $S$ inside $\bR^3$ as $\pi_0(S)$, we can parametrize it using $\sigma \equiv (\sigma_1, \sigma_2, \sigma_3)$. Let
\begin{align*}
K_\sigma(\hat{t}) \equiv& K_\sigma(t,\bar{t}) := (J_{01}(t,\bar{t}), J_{02}(t,\bar{t}), J_{03}(t,\bar{t})),\\
J_\sigma(\hat{t}) \equiv& J_\sigma(t,\bar{t}) := \frac{\partial}{\partial t}\sigma(t, \bar{t}) \times  \frac{\partial}{\partial\bar{t}}\sigma(t, \bar{t}) \equiv (J_{23}(\hat{t}), J_{31}(\hat{t}), J_{12}(\hat{t})).
\end{align*}
\end{notation}

\section{Einstein-Hilbert path integral}

Any spin connection $\omega$ described in Section \ref{s.pre} can be written as $\omega \equiv A^a_{\alpha\beta} \otimes dx_a\otimes \hat{E}^{\alpha\beta}$, whereby $A^a_{\alpha\beta}: \bR^4 \rightarrow \bR$ is smooth and we identify $\Lambda^2(V)$ with $\mathfrak{su}(2) \times \mathfrak{su}(2)$. See Notation \ref{n.su.1}. Considering that this space of smooth spin connections is too big for our purpose, we need to `trim' down this space.

If we are considering $\mathcal{A}_{\bR^4, \mathfrak{g}}$, the standard approach would be to consider $\mathcal{A}_{\bR^4, \mathfrak{g}}$, modulo gauge transformations. Let $\{F_\alpha\}$ be any basis in $\mathfrak{g}$. Under axial gauge fixing, every $A \in \mathcal{A}_{\bR^4, \mathfrak{g}}$ can be gauge transformed into $A^i_{\alpha} \otimes dx_i \otimes F^{\alpha} $, $A^i_{\alpha}: \bR^4 \rightarrow \bR$ smooth, subject to the conditions \beq A^1_{\alpha}(0, x^1, 0, 0) = 0,\ A^2_{\alpha}(0, x^1, x^2, 0)=0,\ A^3_{\alpha}(0, x^1, x^2, x^3) = 0. \label{e.r.1} \eeq

Now, 3+1 gravity is not a gauge theory, in the sense that if we interpret $e$ and $\omega$ as gauge fields, then the Einstein-Hilbert action should be invariant under gauge transformation. But there is no such action in gauge theory. Furthermore spin connection is $\Lambda^2(V)$-valued one form, not exactly a gauge field.  However, we still can apply axial gauge fixing and by making the identification $\Lambda^2(V) \cong \mathfrak{su}(2) \times \mathfrak{su}(2)$, we now  consider \beq \omega = A^i_{\alpha\beta} \otimes dx_i\otimes \hat{E}^{\alpha\beta} \in \overline{\mathcal{S}}_\kappa(\bR^4) \otimes \Lambda^1(\bR^3)\otimes \mathfrak{su}(2) \times \mathfrak{su}(2)  =: L_\omega. \label{e.o.1} \eeq Observe that $A^{i}_{\alpha\beta} = -A^{i}_{\beta\alpha}\in \overline{\mathcal{S}}_\kappa(\bR^4)$ and $\overline{\mathcal{S}}_\kappa(\bR^4)$ is the Schwartz space discussed in Section \ref{s.ss}.

\begin{rem}
Note that the restrictions given by Equation (\ref{e.r.1}) will not be imposed on $\omega$.
\end{rem}

Recall $e$ is $V$-valued one form. Even though $e$ is not a gauge, we will still apply axial gauge fixing argument as above. As a consequence, we will have to drop the invertibility condition, and consider all non-invertible transformations $e: T\bR^4 \rightarrow V$, written as \beq e = B^i_\gamma \otimes dx_i \otimes E^\gamma \in \overline{\mathcal{S}}_\kappa(\bR^4) \otimes \Lambda^1(\bR^3) \otimes V  =: L_e. \label{e.o.2} \eeq

\begin{rem}
\begin{enumerate}
  \item Note that $e(\partial/\partial x_0) = 0$, so after applying axial gauge fixing, we consider non-invertible transformations $e$.
  \item The restrictions given by Equation (\ref{e.r.1}) will not be imposed on $e$.
  \item The reader may object to apply axial gauge fixing to $e$; after all $e$ is not a gauge and in General Relativity, $e$ defines a metric, which is non-degenerate in classical General Relativity. However, as discussed in \cite{Witten:1988hc}, we must consider $e$ to be non-invertible to make sense of or develop 2+1 quantum gravity. Likewise here, to develop a 3+1 quantum gravity, we have to consider non-invertible $e$.
\end{enumerate}
\end{rem}

At this point, it is good to discuss the significance of $w$ and $e$. By identifying $\Lambda^2(V)$ with $\mathfrak{su}(2) \times \mathfrak{su}(2)$, we interpret $\omega$ as a connection with values in $\mathfrak{su}(2) \times \mathfrak{su}(2)$. In Notation \ref{n.su.1}, the first copy of $\mathfrak{su}(2)$ is generated by $\{\hat{E}^{0i}\}_{i=1}^3$, which corresponds to boost in the $x_i$ direction in the Lorentz group; the second copy of $\mathfrak{su}(2)$ is generated by $\{\hat{E}^{\tau(i)}\}_{i=1}^3$, which corresponds to rotation about the $x_i$-axis in the Lorentz group. When we give a representation $\rho^\pm$ to a colored loop $\ol$, which we interpret as representing a particle, we are effectively assigning values to the translational and angular momentum of this particle.

The vierbein $e$ can be interpreted as translating $V$, a 4-dimensional vector space. By choosing an orthonormal basis $\{f_\alpha\}_{\alpha=0}^3$ using the Minkowski metric, we may interpret $f_\alpha$ as generator for translation in the $f_\alpha$ direction, which corresponds to translation in the Poincare group. Note that the Lorentz group is a Lie subgroup of the Poincare Lie group. This means that we may think of $\{\omega, e\}$ as a connection with values in the Poincare Lie Algebra.

\begin{defn}\label{d.t.5}(Time ordering operator)\\
For any permutation $\sigma \in S_r$, \beq \mathcal{T}(A(s_{\sigma(1)})\cdots A(s_{\sigma(r)})) = A(s_1)\cdots A(s_r),\ s_1 > s_2 > \ldots > s_r. \nonumber \eeq Suppose now our matrices $A^u(s)$ are indexed by the curves $u$ and time $s$. Extend the definition of the time ordering operator, first ordering in decreasing values of $u$, followed by the time $s$.
\end{defn}

Consider two oriented hyperlinks, $\oL = \{\ol^u \}_{u=1}^{\on}$, $\uL = \{\ul^v \}_{v=1}^{\underline{n}}$ in $\bR \times\bR^3$. Color each component of $\oL$ with representation $\rho_u$. The hyperlinks $\oL$ and $\uL$ are entangled together to form an oriented colored hyperlink, denoted by $\chi(\oL, \uL)$.  Let $q \in \bR$ be known as a charge.

Define
\begin{align*}
V(\{\ul^v\}_{v=1}^{\underline{n}})(e) :=& \exp\left[ \sum_{v=1}^{\un} \int_{\ul^v} \sum_{\gamma=0}^3 B^i_\gamma \otimes dx_i\right], \\
W(q; \{\ol^u, \rho_u\}_{u=1}^{\on})(\omega) :=& \prod_{u=1}^{\on}\Tr_{\rho_u}  \mathcal{T} \exp\left[ q\int_{\ol^u} A^i_{\alpha\beta} \otimes dx_i\otimes \hat{E}^{\alpha\beta}  \right].
\end{align*}
Here, $\mathcal{T}$ is the time-ordering operator defined in Definition \ref{d.t.5}.

\begin{rem}
The term $\mathcal{T} \exp\left[ q\int_{\ol^u} \omega \right]$ needs some further explanation. If we regard $\Lambda^2(V) \cong \mathfrak{su}(2)\times \mathfrak{su}(2)$, then the former is known as a holonomy operator along a loop $\ol^u$, for a spin connection $\omega$.
\end{rem}

%

Our aim in this article is to give a plausible definition for an Einstein-Hilbert path integral, of the form \beq \frac{1}{Z}\int_{\omega \in L_\omega,\ e \in L_e}H(e, \omega)V(\{\ul^v\}_{v=1}^{\un})(e) W(q; \{\ol^u, \rho_u\}_{u=1}^{\on})(\omega)\  e^{i S_{EH}(e, \omega)}\ De D\omega, \label{e.eha.1} \eeq whereby $De$ and $D\omega$ are Lebesgue measures on $L_e$ and $L_\omega$ respectively and \beq Z = \int_{\omega \in L_\omega,\ e \in L_e}e^{i S_{EH}(e, \omega)}\ De D\omega. \nonumber \eeq

Here, $H$ is some continuous function, possibly taking values in $\bR$ or in some Lie Algebra. In this article, we will consider 3 possible functions, namely
\begin{itemize}
  \item the area of some surface $S \subset \bR^3$,
  \item the volume of a region $R \subset \bR^3$,
  \item and finally the curvature, integrated over some surface $S \subset \bR^4$.
\end{itemize}
These 3 functions will be dealt with separately, in Sections \ref{s.ao}, \ref{s.vo} and \ref{s.co} respectively.

To define these three path integrals, we will show that one can write the path integral in the form of a Chern-Simons path integral, which was studied in \cite{CS-Lim01} and \cite{CS-Lim03}. See Section \ref{s.eha}. Using similar arguments in our previous work, we will write down a set of Chern-Simons rules, given by Definition \ref{d.cs.r.1}, to define the path integral in all three cases. Our main result in this article is to derive these definitions, which are respectively given by Definitions \ref{d.api}, \ref{d.vpi} and \ref{d.cpi} respectively. In a sequel to this article, we will compute explicitly these path integrals, using topological invariants. See \cite{EH-Lim03}, \cite{EH-Lim04} and \cite{EH-Lim05}.

The idea of using hyperlinks in $\bR \times \bR^3$ in the quantization of gravity is not new. The idea of using loops to describe quantum gravity appeared in \cite{PhysRevLett.61.1155}. In that article, the authors wrote down a path or functional integral using a suitable (infinite dimensional) measure. Unfortunately, they did not state the choice of this measure or even give a plausible definition for such an ill-defined integral.

\section{Schwartz space}\label{s.ss}

\begin{notation}\label{n.s.4}
For a vector space $V$, $V^{\otimes^n} \equiv V \otimes \cdots \otimes V$ will mean the $n$-th tensor product of $V$. The notation $V^{\times^n} \equiv V \times \cdots \times V$ means the $n$-th direct product of $V$. If $V$ is an inner product space, then $V^{\otimes^n}$ inherits the tensor inner product.
\end{notation}

\begin{notation}\label{n.n.1}
In this article, $\vec{y} \equiv (y_0, y)$, whereby $y \equiv (y_1, y_2, y_3) \in \bR^3$. If $x \in \bR^n$, we will write $(p_\kappa^x)^2$ to denote the $n$-dimensional Gaussian function, center at $x$, variance $1/\kappa^2$. For example, \beq p_\kappa^x(\cdot) = \frac{\kappa^2}{2\pi}e^{-\kappa^2|\cdot - x|^2/4},\ x \in \bR^4. \nonumber \eeq We will also write $(q_\kappa^x)^2$ to denote the 1-dimensional Gaussian function, i.e. \beq q_\kappa^x(\cdot) = \frac{\sqrt{\kappa}}{(2\pi)^{1/4}}e^{-\kappa^2 (\cdot - x)^2/4}. \nonumber \eeq
\end{notation}

\begin{notation}\label{n.bk.1}
Later on, we will approximate the Dirac-delta function with $p_\kappa$. The bigger the $\kappa$, the better is the approximation. In the end, we will let $\kappa$ go to infinity.

Let $\bk := \kappa/2\sqrt{4\pi}$. This is an important factor, which we need to use throughout this article. As we will see later, we need the correct powers of $\bk$, to ensure that our path integrals will converge to something meaningful.
\end{notation}

Consider the inner product space $\mathcal{S}_\kappa(\mathbb{R})$ which is contained inside the Schwartz space. The space $\mathcal{S}_\kappa(\mathbb{R})$ consists of functions of the form $f  \sqrt{\phi_\kappa}$, whereby $\phi_\kappa$ is the Gaussian function $ \phi_{\kappa}(x) = \kappa e^{-\kappa^2|x|^2/2}/(2\pi)^{1/2}$ and $f$ is a polynomial.

Let $f, g$ be polynomials in $\bR$. The inner product $\langle \cdot, \cdot \rangle$ is given by \beq \langle f \sqrt{\phi_\kappa}, g \sqrt{\phi_\kappa} \rangle = \int_{\bR} f\cdot g \cdot \phi_\kappa\ d\lambda, \nonumber \eeq $\lambda$ is Lebesgue measure on $\mathbb{R}$. Let $\overline{\mathcal{S}}_\kappa(\bR)$ be the smallest Hilbert space containing $\mathcal{S}_\kappa(\mathbb{R})$, using this inner product. Let $\overline{\mathcal{S}}_\kappa(\bR^4)$ be the smallest Hilbert space containing $\mathcal{S}_\kappa(\bR)^{\otimes^4}$.

Suppose $f \in \mathcal{S}_\kappa(\bR^4)$ and $g \notin \overline{\mathcal{S}}_\kappa(\bR^4)$, but $g$ is bounded and continuous. If further $f$ is $L^1$ integrable, by abuse of notation, we will write \beq \langle f, g \rangle := \int_{\bR^4} f\cdot g\ d\lambda, \nonumber \eeq integrating using Lebesgue measure on $\bR^4$.

Suppose $V$ is some vector space and consider the tensor product $C^\infty(\bR^4) \otimes V$. Let $\sum_u \alpha_u \otimes \beta_u \in C^\infty(\bR^4) \otimes V$. We will abuse notation and write for $\gamma \in C^\infty(\bR^4)$, \beq \left\langle \gamma,\sum_u \alpha_u \otimes \beta_u \right\rangle := \sum_u \int_{\bR^4 }\gamma\cdot \alpha_u\ d\lambda\ \otimes \beta_u, \nonumber \eeq
provided the integral $\int_{\bR^4} \gamma \cdot \alpha_u\ d\lambda$ converges.

In summary, we wish to highlight to the reader, that in the rest of this article, when we write $\langle \cdot, \cdot \rangle$, it means integrate using Lebesgue measure, for a given product of $C^\infty$ functions. If $f, g \in C^\infty(\bR^n)$, then \beq \langle f, g \rangle \equiv \int_{\bR^n} f\cdot g\ d\lambda, \nonumber \eeq whereby we integrate the product using Lebesgue measure over $\bR^n$.

\section{Important Linear operators}\label{s.lo}

See Notation \ref{n.s.4}. We will often write $V^{\times^3}$ to mean the vector space consisting of 3-vectors, whose components take values in $V$. For example, $\bR^3$ will be the usual vector space consisting of real 3-vectors. The vector space $V^{\times^9}$ will refer to the space containing vectors with 9 components, whose components take values in $V$.

Given a 3-vector $u = (u_1, u_2, u_3)$, $u_i \in C^\infty(\bR^4)$, we may identify $C^\infty(\bR^4)^{\times^3}$ with a subspace inside $(C^\infty(\bR^4) \otimes \Lambda^1(\bR^3))^{\times^3}$, by \beq u \in C^\infty(\bR^4)^{\times^3} \longleftrightarrow (u_1 \otimes dx_1, u_2 \otimes dx_2, u_3 \otimes dx_3) \in (C^\infty(\bR^4) \otimes \Lambda^1(\bR^3))^{\times^3}. \nonumber \eeq

We will write a 9-vector in $V^{\times^9}$ as $u = (u_1, u_2, u_3)$, whereby each $u_i \in V^{\times^3}$. In the case when $V = C^\infty(\bR^4)$, we may identify $C^\infty(\bR^4)^{\times^9}$ with a subspace inside $(C^\infty(\bR^4) \otimes \Lambda^1(\bR^3))^{\times^9}$, by \beq u  \longleftrightarrow (u_1 \otimes dx_1, u_2 \otimes dx_2, u_3 \otimes dx_3) \in (C^\infty(\bR^4) \otimes \Lambda^1(\bR^3))^{\times^9}. \nonumber \eeq So, each $u_i \otimes dx_i \in (C^\infty(\bR^4) \otimes \Lambda^1(\bR^3))^{\times^3}$.

The Hodge star operator $\ast$ is a linear isomorphism between $\Lambda^1(\bR^3)$ and $\Lambda^2(\bR^3)$ using the volume form $dx_1\wedge dx_2 \wedge dx_3$, i.e.
\begin{align*}
\ast(dx_1) = dx_2 \wedge dx_3,\ \ \ast(dx_2) = dx_3 \wedge dx_1,\ \ \ast(dx_3) = dx_1\wedge dx_2.
\end{align*}
We define another linear isomorphism $\ddag$ between $\Lambda^1(\bR^3)$ and $\Lambda^2(\bR^3)$, by
\begin{align*}
\ddag(dx_i \wedge dx_j + dx_k \wedge dx_i) =& dx_i,
\end{align*}
for $(i,j,k) \in C_3$.

Certain operators arise during the analysis of the Chern-Simons path integrals in \cite{CS-Lim01}, \cite{CS-Lim02} and \cite{CS-Lim03}. The following linear operators act on dense subsets in $\overline{\mathcal{S}}_\kappa(\bR)$ and $\overline{\mathcal{S}}_\kappa(\bR^4)$.

\begin{defn}(Integral operators)\label{d.lo.1}
\begin{enumerate}
\item
For $x = (x_0,x_1, x_2, x_3)$, write \beq x(s_a) :=
\left\{
  \begin{array}{ll}
    (s_0,x_1, x_2, x_3), & \hbox{$a=0$;} \\
    (x_0,s_1, x_2, x_3), & \hbox{$a=1$;} \\
    (x_0,x_1, s_2, x_3), & \hbox{$a=2$;} \\
    (x_0,x_1, x_2, s_3), & \hbox{$a=3$.}
  \end{array}
\right. \nonumber \eeq
\item Let $\partial_a \equiv \partial/\partial x_a$ be a differential operator. There is an operator $\partial_a^{-1}$ acting on a dense subset in $\overline{\mathcal{S}}_\kappa(\bR^4)$, \beq (\partial_a^{-1}f)(x) := \frac{1}{2}\int_{-\infty}^{x_a} f(x(s_a))\ ds_a - \frac{1}{2}\int_{x_a}^{\infty} f(x(s_a))\ ds_a,\ f \in \overline{\mathcal{S}}_\kappa(\bR^4). \label{e.d.1} \eeq Here, $x_a \in \bR$. Notice that $\partial_a\partial_a^{-1}f \equiv f$ and $\partial_a^{-1}f$ is well-defined provided $f$ is in $L^1$, but it is not inside $\overline{\mathcal{S}}_\kappa(\bR^4)$.
\item
Let $\lambda \in \overline{\mathcal{S}}_\kappa(\bR^4)$. We define an operator \beq m_\kappa^i(\lambda) := \frac{1}{\kappa}\partial_i + \lambda,\ i=1, 2, 3, \nonumber \eeq $\lambda$ acts on $f \in \overline{\mathcal{S}}_\kappa(\bR^4)$ by multiplication. Its inverse, $m_\kappa^{i}(\lambda)^{-1}$ acts on a dense subset in $\overline{\mathcal{S}}_\kappa(\bR^4)$ by
\begin{align}
(m_\kappa^{i}(\lambda)^{-1} h)(x) := \frac{\kappa}{2}\left[ \int_{-\infty}^{x_i} - \int_{x_i}^\infty  \right] e^{(s_i-x_i)\kappa\lambda(x(s_i))} h(x(s_i))\ ds_i   ,\ h \in \overline{\mathcal{S}}_\kappa(\bR^4) . \nonumber
\end{align}
\end{enumerate}
\end{defn}

\begin{rem}\label{r.lo.1}
When $\lambda \equiv 0$, then $m_\kappa^{i}(0)^{-1}  = \kappa\partial_i^{-1}$.
\end{rem}

\begin{notation}\label{n.k.1}
Refer to Notations \ref{n.v.1} and \ref{n.n.1}. For each $i = 1, 2, 3$, write
\beq \left\langle p_\kappa^{\vec{x}}, p_\kappa^{\vec{y}} \right\rangle_i =
\left\langle p_\kappa^{\hat{x}_{i}}, p_\kappa^{\hat{y}_{i}} \right\rangle \left\langle q_\kappa^{x_{i}}, \kappa\partial_0^{-1}q_\kappa^{y_{i}} \right\rangle\left\langle \partial_0^{-1}q_\kappa^{x_{0}}, q_\kappa^{y_{0}}\right\rangle. \nonumber \eeq

Here, \beq \partial_0^{-1}q_\kappa^{x_{0}}(t) \equiv \frac{1}{2}\int_{-\infty}^t q_\kappa^{x_{0}}(\tau)\ d\tau -
\frac{1}{2}\int_{t}^\infty q_\kappa^{x_{0}}(\tau)\ d\tau. \nonumber \eeq

Note that $\left\langle \partial_0^{-1}q_\kappa^{x_{0}}, q_\kappa^{y_{0}}\right\rangle$  means we integrate $\partial_0^{-1}q_\kappa^{x_{0}} \cdot q_\kappa^{y_{0}}$ over $\bR$, using Lebesgue measure. It is well-defined because $q_\kappa^{x_0}$ is in $L^1$. Refer to Section \ref{s.ss}.
\end{notation}

To define our Einstein-Hilbert path integrals later on, we need to introduce the following differential operators.

\begin{defn}(Differential operators)\label{d.lo.2}
\begin{enumerate}
\item\label{d.lo.2a} We will consider real 3-vectors, each component taking values in the real line. However, we may consider a 3-vector whose components are in $C^\infty(\bR^4) \otimes \Lambda^1(\bR^3)$. Given a differential operator $\partial_j $, it acts on $u^i \otimes dx_i \in C^\infty(\bR^4) \otimes \Lambda^1(\bR^3)$, $j = 1, 2, 3$ by \beq \partial_j( u^i \otimes dx_i) = \sum_{i \neq j} \frac{\partial u^i}{\partial x_j} \otimes dx_j \wedge dx_i \in C^\infty(\bR^4) \otimes \Lambda^2(\bR^3). \nonumber \eeq Here, there is no implied sum over the $j$.
\item\label{d.lo.2b} Consider a 3-vector $u = (u_1, u_2, u_3) \in (C^\infty(\bR^4) \otimes \Lambda^1(\bR^3))^{\times^3}$. Given $v = (v_1, v_2, v_3) \in (C^\infty(\bR^4))^{\times^3}$, we can define the cross product $v \times u$ by
    \beq v \times u = \left(v_2 u_3 - v_3u_2,\ v_3u_1 - v_1 u_3,\ v_1 u_2 - v_2 u _1 \right)
    . \nonumber \eeq
\item Let $u = (u_1, u_2, u_3) \in (C^\infty(\bR^4) \otimes \Lambda^1(\bR^3))^{\times^3}$. By abuse of notation, we define \beq \nabla \times (u_1, u_2, u_3) := \left(\partial_2 u_3 - \partial_3 u_2,\ \partial_3u_1 - \partial_1 u_3,\ \partial_1 u_2 - \partial_2 u_1 \right).
     \nonumber \eeq Note that each component lies in $C^\infty(\bR^4) \otimes \Lambda^2(\bR^3)$.
\end{enumerate}
\end{defn}

\section{Chern-Simons Integrals}

Suppose we have a (real) Hilbert space $H$, with inner product $\langle \cdot, \cdot \rangle$. Write $H_{\bC} \equiv  \bC \otimes_{\bR} H$ to be the complexification of $H$. The Chern-Simons integral is typically an infinite dimensional integral over $H^{\times^2} \equiv H \times H$.

In \cite{CS-Lim01} and \cite{CS-Lim03}, we defined the Chern-Simons path integrals over in $\bR^3$ and $S^2 \times S^1$ respectively. The definition of such ill-defined path integrals is via constructing an Abstract Wiener space using the Segal Bargmann Transform, followed by defining a path integral of the form \beq \frac{1}{Z}\int_{(w_+ ,w_-) \in H^{\times^2}} e^{\langle w_+, \alpha_+ \rangle}e^{\langle w_-, \alpha_- \rangle}e^{i\langle w_+, w_- \rangle}\ Dw_+ Dw_-, \nonumber \eeq
$\alpha_+, \alpha_- \in H $, with \beq Z = \int_{(w_+ ,w_-) \in H^{\times^2}} e^{i\langle w_+, w_- \rangle}\ Dw_+ Dw_-, \nonumber \eeq over the Abstract Wiener space. Note that $Dw_\pm$ is some Lebesgue measure over $H$, which does not exist. See \cite{MR0461643}.

Using Fourier transform and analytic continuation (See Proposition 3.3 in \cite{CS-Lim01}.), we define the above
integral as $e^{i\langle \alpha_+, \alpha_- \rangle}$, $i = \sqrt{-1}$. From this definition, we can easily extend to more general type of path integrals.

\begin{defn}
An integral written on $H^{\times^2}$ is said to be a Chern-Simons integral if it is of the form \beq \frac{1}{Z}\int_{(w_+ ,w_-) \in H^{\times^2}} F(w_+, w_-)e^{i\langle w_+, w_- \rangle}\ Dw_+ Dw_-, \nonumber \eeq for some continuous function $F$ and \beq Z = \int_{(w_+,w_-) \in H^{\times^2}}e^{i\langle w_+, w_- \rangle}\ Dw_+ Dw_- \nonumber \eeq is some normalization constant.
\end{defn}

Let $T_{\bar{u}}: \bC \rightarrow {\rm End}(H_\bC)$ be a linear operator. And let $\beta_u, \hat{\beta}_{\bar{u}}, \alpha_\pm \in H$ be fixed, for $u=1,\ldots, n$ and $\bar{u}=1, \ldots, d$. Typically, the Chern-Simons integral we want to compute is of the form \beq \frac{1}{Z}\int_{(w_+,w_-) \in H^{\times^2}} F(\langle w_-, \beta_1 \rangle, \ldots, \langle w_-, \beta_n \rangle) e^{\langle w_-, \sum_{\bar{u}=1}^dT_{\bar{u}}(\langle w_-, \hat{\beta}_{\bar{u}} \rangle )\alpha_- \rangle} e^{\langle w_+, \alpha_+ \rangle} e^{i\langle w_+, w_- \rangle}\ Dw_+ Dw_-. \label{e.cs.1} \eeq Here, $F$ is some continuous function on $\bR^n$, which admits an analytic continuation on $\bC^n$.

\begin{rem}
Expression \ref{e.cs.1} is not the most general form whereby we can make a sensible definition for it, but for those Chern-Simons integrals that we are going to consider, this will suffice.
\end{rem}

The following definition is taken from Definition 8 in \cite{CS-Lim03}, which can be thought of as a generalization of Proposition 3.3 in \cite{CS-Lim01}.

\begin{defn}\label{d.cs.1}
We define the integral in Expression \ref{e.cs.1} as \beq F(i\langle \alpha_+, \beta_1 \rangle, \ldots, i\langle \alpha_+, \beta_n \rangle) e^{i\langle \alpha_+,  \sum_{\bar{u}=1}^dT_{\bar{u}}(i\langle \alpha_+, \hat{\beta}_{\bar{u}} \rangle )\alpha_- \rangle}. \nonumber \eeq
\end{defn}

\begin{rem}
From $e^{\langle w_+, \alpha_+ \rangle}$, we replace $w_-$ with $i\alpha_+$ in the expression \beq F(\langle w_-, \beta_1 \rangle, \ldots, \langle w_-, \beta_n \rangle) e^{\langle w_-, \sum_{\bar{u}=1}^dT_{\bar{u}}(\langle w_-, \hat{\beta}_{\bar{u}} \rangle )\alpha_- \rangle}. \nonumber \eeq This will give us the expression in Definition \ref{d.cs.1}.
\end{rem}

Unfortunately, the Chern-Simons path integrals that one is interested in is not exactly in the above form. (See \cite{CS-Lim01}and \cite{CS-Lim03}.) Typically, we need to consider \beq \frac{1}{Z}\int_{(u_+,u_-) \in H^{\times^2}} F(u_+, u_-)e^{i\langle u_+, Tu_- \rangle}\ Du_+ Du_-, \label{e.cs.2} \eeq with \beq Z = \int_{(u_+,u_-) \in H^{\times^2}}e^{i\langle u_+, Tu_- \rangle}\ Du_+ Du_-. \nonumber \eeq Here, $T$ is some unbounded operator acting on $H$.

%

To define the path integral in Expression \ref{e.cs.2}, the conventional approach would be to write it in the form
\beq \frac{1}{Z}\int_{(u_+,u_-) \in H^{\times^2}} F(u_+, T^{-1}Tu_-)e^{i\langle u_+, Tu_- \rangle}\ \det[T]^{-1}Du_+ D[Tu_-], \nonumber \eeq with \beq Z = \int_{(u_+,u_-) \in H^{\times^2}}e^{i\langle u_+, Tu_- \rangle}\ \det[T]^{-1}Du_+ D[Tu_-]. \nonumber \eeq Now $\det[T]^{-1}$ is typically some undetermined constant. It will be factored out and canceled with another copy in $Z$. Finally, we replace $Tu_- \mapsto u_-$ and we interpret Expression \ref{e.cs.2} as \beq \frac{1}{\tilde{Z}}\int_{(u_+,u_-) \in H^{\times^2}} F(u_+, T^{-1}u_-)e^{i\langle u_+, u_- \rangle}\ Du_+ Du_-, \nonumber \eeq with \beq \tilde{Z} = \int_{(u_+,u_-) \in H^{\times^2}}e^{i\langle u_+, u_- \rangle}\ Du_+ Du_-. \nonumber \eeq

Because it is not possible to define the determinant of $T$, the above heuristic change of variables argument is not mathematically rigorous. However, it does serve as a starting point on how to define the Chern-Simons integrals that we are really interested in and give us a plausible definition for such an integral.

\subsection{Chern-Simons Path Integral in $\bR^4$}\label{ss.cs}

We now need to describe our infinite dimensional Hilbert space, defined over a 4-manifold $\tilde{M} = \bR \times \bR^3$. 

The infinite dimensional Hilbert space we will consider is $H = L^2(\tilde{M})\otimes W $, $W$ is some finite dimensional inner product space. The inner product on $H$ is the tensor inner product from both $L^2(\tilde{M})$ and $W$. We will denote the inner products from $L^2(\tilde{M})$, $W$ and their tensor inner product by the same symbol $\langle \cdot, \cdot \rangle$.

Let $\{e_i^u\}$ be some orthonormal basis in $W$, indexed by $i$ and $u$. The index $u$ can take values from 1 to some whole number, but $i$ will take values 1, 2 and 3. Hence, the dimension of $W$ must be a multiple of 3.

Any $w \in H$ can be written in the form $w \equiv w_u^i\otimes e_i^u$. Suppose $L: W \rightarrow {\rm End}(W)$. We also assume that for $x \in W$, $L(x)$ is a skew-symmetric operator, i.e. $\langle L(x)z, y \rangle = -\langle z, L(x)y \rangle$. For $w_u^i, w_{\bar{u}}^j \in L^2(\tilde{M})$, define a linear operator $L(w_u^i \otimes e_i^u): L^2(\tilde{M})\otimes W  \rightarrow L^2(\tilde{M}) \otimes W$ by
\beq [L(w_u^i \otimes e_i^u )(w_{\bar{u}}^j \otimes e_j^{\bar{u}})](m) := L(w_u^i(m)e_i^u)(w_{\bar{u}}^j(m) e_j^{\bar{u}}) \in W. \nonumber \eeq Note that here, we evaluate at the point $m \in \tilde{M}$, so $w_u^i(m) e_i^u \in W$, similarly for $w_{\bar{u}}^i(m) e_i^{\bar{u}}$. Hence the operator $L(w_u^i \otimes e_i^u )$ acts pointwise for each $m \in \tilde{M}$. Typically, $L(w_u^i \otimes e_i^u)$ is some multiplication operator.

Finally, let $D: L^2(\tilde{M}) \otimes W \rightarrow L^2(\tilde{M}) \otimes W$ be some skew-symmetric differential operator, acting on $L^2(\tilde{M}) \otimes W$. So $D + L(w_-)$ is skew-symmetric, i.e. \beq \langle [D + L(w_-)]w_+, w_- \rangle = -\langle w_+, [D + L(w_-)]w_- \rangle,\ w_\pm \in L^2(\tilde{M}) \otimes W. \nonumber \eeq

At the point $\vec{x} \in \tilde{M}$, \beq [D + L(w_-)]w_+(\vec{x}) = (Dw_+)(\vec{x}) + L(w_-(\vec{x}))w_+(\vec{x}) = [D + L(w_-(\vec{x}))](w_+)(\vec{x}). \nonumber \eeq Hence, \beq [D + L(w_-)]^{-1}w_+(\vec{x}) \equiv [D + L(w_-(\vec{x}))]^{-1}w_+(\vec{x}). \label{e.ti.1} \eeq

\begin{defn}\label{d.om.1}
Let $H = L^2(\tilde{M}) \otimes W $, whereby $W$ is a finite dimensional inner product space and  $\{x_0,x_1, x_2, x_3\}$ be coordinates for $\bR \times \bR^3$.

Fix a basis
$\{ e_i^u\ :\ i=1, 2, 3,\ u=1, \ldots, d\}$ for $W$ and thus any $w_\pm \in H$ can be written in the form $w_\pm = w_{\pm, u}^i \otimes e_i^u $ in the latter. We denote all inner products with $\langle \cdot, \cdot \rangle$.
Introduce a Lie Algebra, $V$ and suppose $\{E^u\}_{u=1}^d \subset V$ is a basis.

Define
\begin{align*}
\Omega_1 :=&
\left\{ \int_L w_{\pm, u}^i \otimes dx_i \otimes E^u\ |\ L\ {\rm is\ a}\ 1-{\rm dim\ submanifold\ in}\ \bR \times \bR^3  \right\}, \\
\Omega_2:=& \left\{ \int_S g^{ab}(w_-)\otimes dx_a \wedge dx_b\ |\ S\ {\rm is\ a}\ 2-{\rm dim\ submanifold\ in}\ \bR \times \bR^3 \right\}, \\
\Omega_3:=& \left\{ \int_R g^{abc}(w_-)\otimes dx_a \wedge dx_b \wedge dx_c\ |\ R\ {\rm is\ a}\ 3-{\rm dim\ submanifold\ in}\ \bR \times \bR^3 \right\}.
\end{align*}
In the above, $g^{ab}$ and $g^{\alpha \beta \mu}$ are continuous functions on $W$.
\end{defn}

\begin{rem}\label{r.om.1}
\begin{enumerate}
  \item Note that the functions $g^{ab}$ need not be real-valued. To keep things in general, one can assume that it take values in some finite dimensional vector space.
  \item For a 1-submanifold $M^1$ to be considered in $\Omega_1$, we typically consider a hyperlink. However, we may in certain occasions consider an open curve. Note that an integral in $\Omega_1$ is actually $V$-valued.
  \item For a 2-submanifold $M^2$ to be considered in $\Omega_2$, henceforth referred to as a surface, we do allow the surface to have a boundary. Typically, we consider closed and bounded surfaces, with or without boundary. We can consider a finite union of such surfaces, with empty intersection.
  \item For a 3-submanifold $M^3$ to be considered in $\Omega_3$, it should be closed and bounded, henceforth referred to as a compact region. We allow it to be disconnected, with finitely many components.
\end{enumerate}
\end{rem}

\begin{example}\label{ex.l.1}
Assume that $\oL \equiv \{\ol^v\}_{v=1}^{\on}$ is a hyperlink. Now, to compute a typical integral from $\Omega_1$, for each loop $\ol^v$, we will choose a parametrization $\vec{y}^v \equiv (y_0^v, y_1^v, y_2^v, y_3^v): I \rightarrow \tilde{M}$ such that $\vec{y}^v(I)$ describes the loop $\ol^v \subset \tilde{M}$. Similarly, for each loop $\ul^{\bar{v}}$ in $\uL \equiv \{\ul^{\bar{v}}\}_{\bar{v}=1}^{\un}$, let $\vec{\varrho}^{\bar{v}} \equiv (\varrho_0^{\bar{v}}, \varrho_1^{\bar{v}}, \varrho_2^{\bar{v}}, \varrho_3^{\bar{v}}): I \rightarrow \tilde{M}$ such that $\vec{\varrho}^{\bar{v}}(I)$ describes the loop $\ul^{\bar{v}} \subset \tilde{M}$.

Explicitly, the integral in $\Omega_1$ is computed as
\begin{align}
\int_{\oL} w_{+,u}^i \otimes dx_i \otimes E^u  \equiv& \sum_{v=1}^{\on}\int_I w_{+,u}^i(\vec{y}^v(s))  y_i^{v,\prime}(s)\otimes E^u_{s}\ ds
=: \int_{\oL}\overline{G}^1\left\{w_{+,u}^iy_i^{v,\prime}\otimes E^u\right\}. \label{e.o.3a}
\end{align}

Here, $E^u_{s} \equiv E^u$ and the subscript $s$ is to keep track of the ordering of the matrices $E^u$. This is necessary because when we take tensor products of these terms, we need to time order them according to the time ordering operator, given in Definition \ref{d.t.5}.

We will also need the following integral,
\begin{align}
\int_{\uL} \sum_{u=1}^dw_{-,u}^i \otimes dx_i \equiv& \sum_{\bar{v}=1}^{\un}\sum_{u=1}^d\int_I w_{-,u}^i(\vec{\varrho}^{\bar{v}}(s))  \varrho_i^{\bar{v},\prime}(s)\ ds
=: \int_{\uL}\underline{G}^1\left\{w_{-,u}^i\varrho_i^{\bar{v},\prime}\right\}.
\label{e.o.3}
\end{align}
Note that this integral is real-valued.
\end{example}

\begin{example}(Surface Integral)\label{ex.l.2}\\
Assume that $M^2 \equiv S$ is a connected surface and $\{g^{ab}\}$ is a set of continuous real-valued functions on $W$. Now, to compute a typical integral from $\Omega_2$, we will choose a parametrization $\vec{\sigma} \equiv (\sigma_0, \sigma_1, \sigma_2, \sigma_3): I^2 \rightarrow \tilde{M} = \bR \times \bR^3$ such that $\vec{\sigma}(I^2)$ describes the surface $S \subset \tilde{M}$. See Notation \ref{n.s.3}.

Recall $w_- = w_{-, u}^i \otimes e_i^u$ and $w_-(\vec{x}) \equiv w_{-, u}^i(\vec{x})e_i^u$. Explicitly, the integral becomes $(\hat{s} = (s, \bar{s}))$
\begin{align}
\int_S &g^{ab}(w_-) dx_a \wedge dx_b
\equiv \int_{I^2} g^{0i}(w_-(\vec{\sigma}(\hat{s})) )\ J_{0i}(\hat{s})\ d\hat{s}
+ \int_{I^2} g^{\tau(j)}(w_-(\vec{\sigma}(\hat{s})))\ J_{\tau(j)}(\hat{s})\ d\hat{s} \nonumber \\
=:& \int_{M^2}G^2\{w_{-,u}^i(\vec{x})\} \nonumber
\end{align}
\end{example}

\begin{example}(Volume Integral)\label{ex.l.3}\\
Assume that $M^3 \equiv R$ is a compact region and $\{g^{abc}\}$ is a set of real-valued continuous functions on $W$. For simplicity, we assume that $M^3 \subset \{0\} \times \bR^3$. Now, to compute a typical integral from $\Omega_3$, we will choose a parametrization $\vec{\rho} \equiv (0, \rho_1, \rho_2, \rho_3) \equiv (0, \rho): I^3 \rightarrow \tilde{M}$ such that $\vec{\rho}(I^3)$ describes the compact region $R \subset \bR \times \bR^3$. Now, $w_- = w_{-, u}^i \otimes e_i^u$ and $w_-(\vec{x}) \equiv w_{-, u}^i(\vec{x}) e_i^u$.

Explicitly, the integral becomes $(\tau \equiv (\tau_1, \tau_2, \tau_3))$
\begin{align}
\int_R & g^{abc}(w_-) dx_a \wedge dx_b \wedge dx_c
\equiv \int_{I^3} g^{abc}( w_-(\vec{\rho}(\tau)))\ \left|\frac{\partial \rho}{\partial \tau_1} \cdot \frac{\partial \rho}{\partial \tau_2} \times \frac{\partial \rho}{\partial \tau_3}\right|(\tau)\ d\tau_1d\tau_2d\tau_3
\nonumber \\
=:& \int_{M^3}G^3\{w_{-,u}^i(\vec{x})\}. \nonumber
\end{align}
\end{example}

Recall the time ordering operator $\mathcal{T}$ defined in Definition \ref{d.t.5}. Note that
\begin{align*}
\mathcal{T}\exp\left[ \int_{\oL}\overline{G}^1\left\{w_{+,u}^iy_i^{v,\prime}\otimes E^u\right\} \right]
=&  \sum_{n=0}^\infty \frac{1}{n!}\mathcal{T} \left[ \int_{\oL}\overline{G}^1\left\{w_{+,u}^iy_i^{v,\prime}\otimes E^u\right\} \right]^{\otimes^n} \in \bigoplus_{n=0}^\infty V_{\bC}^{\otimes^n}, \\
\exp\left[ \int_{\uL}\underline{G}^1\left\{w_{-,u}^i \varrho_i^{\bar{v},\prime}\right\} \right] =&  \sum_{n=0}^\infty \frac{1}{n!}\left[ \int_{\uL}\underline{G}^1\left\{w_{-,u}^i\varrho_i^{\bar{v},\prime}\right\} \right]^{n} \in \bR.
\end{align*}
Here, $w^{\otimes^n}$ means take the tensor product of $n$-copies of $w$ and $V_\bC \equiv V \otimes_{\bR}\bC$.

\begin{rem}
We refer the reader to \cite{CS-Lim02} on how does one compute \beq \mathcal{T} \left[ \exp\left( \sum_{v=1}^{\on}\int_I w_{+,u}^i(\vec{y}^v(s))  y_i^{v,\prime}(s)\otimes E^u_{s}\ ds \right)\right]. \nonumber \eeq However, as shown in our calculations later, it is not necessary to keep track of the subscript $s$ and we can effectively drop the time ordering operator for simplicity. Thus, the reader can safely ignore the subscript $s$.
\end{rem}

Recall $w_\pm = w_{\pm, u}^i \otimes e^u_i$ and let \beq \exp\left[i\langle w_{+,u}^i \otimes e_i^u, [D+L(w_-)] w_{-,\bar{u}}^j \otimes e_j^{\bar{u}} \rangle\right]\ Dw_+ Dw_- \equiv D\Lambda. \nonumber \eeq Let $Z = \int_{H \times H} D\Lambda$.
As it turns out, certain functional integrals can be written similarly to a Chern-Simons integral.

Suppose we want to make sense of the following path integral \beq \frac{1}{Z}\int_{(w_+, w_- )\in H^{\times^2}} \int_{M^p}G^p\{w_{-,\bar{u}}^i(\vec{x})\} \otimes \left[ e^{\int_{\oL} w_{+,u}^i  \otimes dx_i \otimes E^u }e^{ \int_{\uL}\sum_{\bar{u}} w_{-,\bar{u}}^j \otimes dx_j } \right]D\Lambda, \label{e.cs.3} \eeq for $p=0, 2, 3$. See Examples \ref{ex.l.1}, \ref{ex.l.2} and \ref{ex.l.3}. When $p = 0$, we define $\int_{M^p}G^p(w_-(\vec{x})) := 1$.

\begin{rem}
\begin{enumerate}
  \item The above path integral is not the most general form which we can give a definition.
  \item If we apply the time ordering operator, then the path integral will take values in \beq  \bigoplus_{n=0}^\infty V_{\bC}^{\otimes^n} . \nonumber \eeq 
\end{enumerate}

\end{rem}



Let $\delta^{\vec{x}}$ be the Dirac-delta function, i.e. $w_{\pm,u}^i(\vec{x}) \equiv \langle w_{\pm,u}^i, \delta^{\vec{x}} \rangle$. Of course, $\delta^{\vec{x}}$ does not lie in the Hilbert space $L^2(\tilde{M})$. So we have to approximate it using $L^1$ functions on $\tilde{M}$, i.e. $p_\kappa^{\vec{x}} \rightarrow \delta^{\vec{x}}$ in the distribution sense, as $\kappa \rightarrow \infty$.

\begin{example}
Refer to Example \ref{ex.l.1}.
We can write Equations (\ref{e.o.3a}) and (\ref{e.o.3}) as
\begin{align}
\sum_{v=1}^{\on}\int_I&\ ds \langle w_{+,u}^i, \delta^{\vec{y}^v(s)} \rangle\ y_i^{v,\prime}(s)\ \otimes E^u_{s} =: \int_{\oL}\overline{G}^1\{\langle w_{+,u}^i, \delta^{\vec{y}^v} \rangle y_i^{v,\prime}\otimes E^u\}, \label{e.pi.2}\\
\sum_{\bar{v}=1}^{\un}\int_I&\ ds  \sum_{\bar{u}=1}^d\langle w_{-,\bar{u}}^j, \delta^{\vec{\varrho}^{\bar{v}}(s)} \rangle\ \varrho_j^{\bar{v},\prime}(s)=: \int_{\uL}\underline{G}^1\{\langle w_{-,\bar{u}}^j, \delta^{\vec{\varrho}^{\bar{v}}} \rangle \varrho_j^{\bar{v},\prime}\}\equiv  \int_{\uL}\underline{G}^1\left\{ \left\langle w_{-}, \delta^{\vec{\varrho}^{\bar{v}}}\otimes \sum_{j=1}^3 \varrho_j^{\bar{v},\prime}e_j^{\bar{u}} \right\rangle \right\} \nonumber
\end{align}
respectively. Replace $\delta$ with $p_\kappa$ in Equation (\ref{e.pi.2}), then we have
\begin{align}
\sum_{v=1}^{\on}&\int_I\ ds \langle w_{+,u}^i, p_\kappa^{\vec{y}^v(s)} \rangle\ y_i^{v,\prime}(s)\ \otimes E^u_s =: \left\langle w_{+},
\sum_{v=1}^{\on}\sum_{i=1}^3 \sum_{u=1}^d \int_I\ ds\ p_\kappa^{\vec{y}^v(s)} \ y_i^{v,\prime}(s)\ \otimes e_i^u \otimes E^u_{s} \right\rangle. \label{e.pi.1}
\end{align}
Write \beq \pi_+\{\vec{y}^v\} := \sum_{v=1}^{\on}\sum_{i=1}^3 \sum_{u=1}^d \int_I\ ds\ p_\kappa^{\vec{y}^v(s)} \ y_i^{v,\prime}(s)\ \otimes e_i^u \otimes E^u_{s} . \label{e.pi.3} \eeq
\end{example}

\begin{rem}
Note that $\pi_+\{\vec{y}^v\}$ is in $L^2(\tilde{M}) \otimes W \otimes V$. And \beq \left\langle w_{+,u}^i \otimes e_i^u, \int_I\ ds\ p_\kappa^{\vec{y}^v(s)} \ y_j^{v,\prime}(s)\ \otimes e_j^{\bar{u}}\otimes E^{\bar{u}}_{s}  \right\rangle = \int_I\ ds \left\langle w_{+, \bar{u}}^j, p_\kappa^{\vec{y}^v(s)}\right\rangle\ y_j^{v,\prime}(s)\ \otimes E^{\bar{u}}_{s}, \nonumber \eeq which takes values in $V$. Note that on both the LHS and RHS, there is no sum over the index $j$ and $\bar{u}$.
\end{rem}

To write down the definition for the path integral in Expression \ref{e.cs.3}, we have to go through the long and tedious process of constructing an Abstract Wiener space using the Segal Bargmann Transform, applying Definition \ref{d.cs.1} and furthermore, approximating the Dirac-delta function using a Gaussian function. This was the approach used in \cite{CS-Lim02} and \cite{CS-Lim03} to define the Chern-Simons path integrals, given respectively by Equation (2.6) in \cite{CS-Lim02} and Equation (24) in \cite{CS-Lim03}. The same approach can be adapted to give a definition for Expression \ref{e.cs.3}.

In anticipation of future work, we propose to write down a set of rules, so that in future, one can write down a definition for the path integral in Expression \ref{e.cs.3}, without going through the full derivation using the Abstract Wiener space approach. These rules are derived from Equation (2.6) in \cite{CS-Lim02} and Equation (24) in \cite{CS-Lim03}.

\begin{defn}\label{d.cs.r.1}
The following are the Chern-Simons rules for evaluating an integral given by Expression \ref{e.cs.3}.

\begin{enumerate}
  \item\label{d.cs.r.1a} Replace $w_{\pm,u}^i(\vec{y})$ with $\langle w_{\pm,u}^i, \delta^{\vec{y}} \rangle$.
  \item\label{d.cs.r.1b} Write $T_{w_-} = D + L(w_-)$ and note that $T_{w_-}^{-1}\equiv [D + L(w_-)]^{-1}$ is skew-symmetric, i.e. \beq \left\langle T_{w_-}^{-1}w_+, w_- \right\rangle = -\left\langle w_+, T_{w_-}^{-1}w_- \right\rangle,\ w_\pm \in L^2(\tilde{M}) \otimes W.  \nonumber \eeq
  See Equation (\ref{e.ti.1}) for the definition of $T_{w_-}^{-1}$.

  Now \beq w_{-,\bar{u}}^i(\vec{x})= \langle w_{-,v}^j(\vec{x}) e_j^v,e_i^{\bar{u}}\rangle
=\left\langle [T_{w_-}^{-1} T_{w_-}w_{-,v}^j \otimes e_j^v](\vec{x}),e_i^{\bar{u}}\right\rangle = \left\langle T_{w_-(\vec{x})}^{-1} [ T_{w_-}w_{-}](\vec{x}),e_i^{\bar{u}}\right\rangle .
\nonumber \eeq Here, $T_{w_-(\vec{x})}^{-1} = [D + L(w_-(\vec{x}))]^{-1}$.

So we can write \beq
\langle T_{w_-(\vec{x})}^{-1} [ T_{w_-}w_{-}](\vec{x}),e_i^{\bar{u}}\rangle = \langle T_{w_-(\vec{x})}^{-1} T_{w_-}w_{-}, \delta^{\vec{x}}\otimes e_i^{\bar{u}}\rangle = -\langle  T_{w_-}w_{-}, T_{w_-(\vec{x})}^{-1} \delta^{\vec{x}}\otimes e_i^{\bar{u}}\rangle
. \nonumber \eeq
    Hence the integral can be written in the form
  \begin{align*}
  \frac{1}{Z}&\int_{(w_+, w_-)\in H^{\times^2}}Dw_+ Dw_-e^{i\langle w_+,  T_{w_-}w_- \rangle}  \int_{M^p}G^p\left\{ \left\langle T_{w_-(\vec{x})}^{-1}T_{w_-}w_{-}, \delta^{\vec{x}}\otimes e_i^{\bar{u}} \right\rangle \right\} \\
  &\otimes
   \exp\left(\int_{\oL}\overline{G}^1\left\{\langle w_{+,u}^i, \delta^{\vec{y}^v} \rangle y_i^{v,\prime}\otimes E^u\right\} \right)\\
  &\hspace{4cm} \exp\left(\int_{\uL}\underline{G}^1\left\{\left\langle T_{w_-(\vec{\varrho})}^{-1}T_{w_-}w_{-}, \delta^{\vec{\varrho}^{\bar{v}}}\otimes \sum_{i=1}^3\varrho_i^{\bar{v},\prime} e_i^{\bar{u}} \right\rangle \right\} \right),
  \end{align*}
  which is equal to
  \begin{align*}
  \frac{1}{Z}&\int_{(w_+, w_- )\in H^{\times^2}}Dw_+ Dw_-e^{i\langle w_+,  T_{w_-}w_- \rangle} \int_{M^p}G^p\left\{-\left\langle T_{w_-}w_-, T_{w_-(\vec{x})}^{-1}\delta^{\vec{x}} \otimes e_i^{\bar{u}} \right\rangle \right\} \\
  &\otimes  \exp\Bigg(\int_{\oL}\overline{G}^1\left\{\left\langle w_{+,u}^i, \delta^{\vec{y}^v} \right\rangle y_i^{v,\prime}\otimes E^u \right\}\Bigg) \\
& \hspace{4cm} \exp\Bigg( \int_{\uL}\underline{G}^1\left\{-\left\langle T_{w_-}w_-, T_{w_-(\vec{\varrho})}^{-1}\delta^{\vec{\varrho}^{\bar{v}}}\otimes \sum_{i=1}^3\varrho_i^{\bar{v},\prime}e_i^{\bar{u}} \right\rangle \right\}\Bigg) .
  \end{align*}
  Replace $T_{w_-}w_-$ by $w_-$ and write $w_-(\vec{x}) = \langle w_{-,u}^j, \delta^{\vec{x}} \rangle e_j^{u} \in W$.
  Hence \beq T_{w_-(\vec{x})}^{-1} = T_{\langle w_{-,u}^j, \delta^{\vec{x}} \rangle e_j^{u}}^{-1} =:T\left[\langle w_{-,u}^j, \delta^{\vec{x}} \rangle e_j^{u} \right]^{-1}. \nonumber \eeq

  So we will now focus on the following path integral of the form
  \begin{align}
  \frac{1}{Z}&\int_{(w_+, w_- )\in H^{\times^2}}Dw_+ Dw_-e^{i\langle w_+,  w_- \rangle}  \int_{M^p}G^p\left\{-\left\langle w_-, T\left[\langle w_{-,u}^j, \delta^{\vec{x}} \rangle e_j^{u}\right]^{-1}\delta^{\vec{x}}\otimes e_i^{\bar{u}} \right\rangle \right\} \nonumber \\
  &\otimes \exp\Bigg( \int_{\oL}\overline{G}^1\left\{\left\langle w_{+,u}^i, \delta^{\vec{y}^v} \right\rangle y_i^{v,\prime}\otimes E^u\right\} \Bigg)\nonumber \\
& \hspace{3cm} \exp\Bigg( \int_{\uL}\underline{G}^1\left\{-\left\langle w_-, T\left[\left\langle w_{-,u}^j, \delta^{\vec{\varrho}^{\bar{v}}} \right\rangle e_j^{u}\right]^{-1}\delta^{\vec{\varrho}^{\bar{v}}}\otimes \sum_{i=1}^3\varrho_i^{\bar{v},\prime}e_i^{\bar{u}} \right\rangle \right\} \Bigg) , \label{e.w.4}
  \end{align}
  which resembles that of a Chern-Simons integral
  and \beq Z = \int_{(w_+, w_- )\in H^{\times^2}}Dw_+ Dw_-e^{i\langle w_+,  w_- \rangle}. \nonumber \eeq

  \begin{rem}
  Instead of giving a plausible meaning to Expression \ref{e.cs.3}, we will now define Expression \ref{e.w.4} by treating it as a Chern-Simons integral.
  \end{rem}

  \item\label{d.cs.r.1c} Because $\delta^x$ is not inside $L^2(\tilde{M})$, we replace it with a Gaussian-like function supported in $\tilde{M}$, $p_\kappa^{\vec{x}}$, such that $p_\kappa^{\vec{x}} \rightarrow \delta^{\vec{x}}$ in distribution. The above integral now becomes
  \begin{align*}
  \frac{1}{Z}&\int_{(w_+, w_- )\in H^{\times^2}}Dw_+ Dw_-e^{i\langle w_+,  w_- \rangle}  \int_{M^p}G^p\left\{-\left\langle w_-, T\left[\langle w_{-,u}^j, p_\kappa^{\vec{x}} \rangle e_j^{u}\right]^{-1}p_\kappa^{\vec{x}}\otimes e_i^{\bar{u}} \right\rangle \right\} \\
  &\otimes \Bigg[ \exp\Bigg( \int_{\oL}\overline{G}^1\left\{\left\langle w_{+,u}^i, p_\kappa^{\vec{y}^v} \right\rangle y_i^{v,\prime}\otimes E^u\right\} \Bigg)\\
& \hspace{3cm} \exp\Bigg( \int_{\uL}\underline{G}^1\left\{-\left\langle w_-, T\left[\left\langle w_{-,u}^j, p_\kappa^{\vec{\varrho}^{\bar{v}}} \right\rangle e_j^{u}\right]^{-1}p_\kappa^{\vec{\varrho}^{\bar{v}}}\otimes \sum_{i=1}^3\varrho_i^{\bar{v},\prime}e_i^{\bar{u}} \right\rangle \right\} \Bigg) \Bigg].
  \end{align*}
  Note that the operator $T\left[\left\langle w_{-,u}^j, p_\kappa^{\vec{\varrho}^{\bar{v}}} \right\rangle e_j^{u}\right]^{-1}$ acts on $p_\kappa^{\vec{\varrho}^{\bar{v}}} \otimes \sum_{i=1}^3 \varrho_i^{\bar{v},\prime}e_i^{\bar{u}} \in L^2(\tilde{M}) \otimes W$.

  \item\label{d.cs.r.1d} We need to put in factors of $\bk = \kappa/2\sqrt{4\pi}$. For each integral from $\Omega_p$, we scale it by factors of $\bk$ as follows:
   \begin{align*}
   \int_{M^p}G^p \mapsto \bk^p \int_{M^p}G^p,\ \ \int_{\uL}\underline{G}^1 \mapsto \bk \int_{\uL}\underline{G}^1
   .
   \end{align*}
   However, for $\int_{\oL}\overline{G}^1$, we scale it by a factor of $\kappa\bk$, i.e. \beq \int_{\oL}\overline{G}^1 \mapsto \kappa\bk \int_{\oL}\overline{G}^1. \nonumber \eeq

  \item\label{d.cs.r.1e}
Refer to Equation (\ref{e.pi.3}) for the definition of $\pi_+\{\vec{y}^v\}$. From Equations (\ref{e.pi.2}) and (\ref{e.pi.1}), \beq \kappa\bk\int_{\oL}\overline{G}^1\left\{\left\langle w_{+,u}^i, p_\kappa^{\vec{y}^v} \right\rangle y_i^{v,\prime}\otimes E^u\right\} = \left\langle w_+, \kappa\bk\pi_+\{\vec{y}^v\} \right\rangle. \nonumber \eeq
Finally apply Definition \ref{d.cs.1}, i.e. replace $w_{-,\bar{u}}^i$ and $w_-$ respectively with
\begin{align*}
w_{-,u}^j &\longmapsto \sqrt{-1}\kappa\bk\pi_{+, u}^j\{\vec{y}^v\} := \sqrt{-1}\kappa\bk\sum_{v=1}^{\on}\int_I\ p_\kappa^{\vec{y}^v(s)}\ y_j^{v,\prime}(s)\ ds ,\\
w_- &\longmapsto \sqrt{-1}\kappa\bk\pi_+\{\vec{y}^v\} = \sqrt{-1}\kappa\bk\sum_{v=1}^{\on}\sum_{i=1}^3 \sum_{u=1}^d \int_I\ ds \ p_\kappa^{\vec{y}^v(s)} \ y_i^{v,\prime}(s)\ \otimes e_i^u\otimes E^u_{s} ,
\end{align*}
and the path integral is defined as
\begin{align*}
  \bk^p &\int_{M^p}G^p\left\{-\left\langle i\kappa\bk\pi_+\{\vec{y}^v\}, T\left[\left\langle i\kappa\bk\pi_{+, u}^j\{\vec{y}^v\}, p_\kappa^{\vec{x}} \right\rangle e_j^{u}\right]^{-1}p_\kappa^{\vec{x}}\otimes e_i^{\bar{u}} \right\rangle \right\} \\
&\otimes \exp\left(\bk
  \int_{\uL}
  \underline{G}^1\left\{ -\left\langle i\kappa\bk\pi_+\{\vec{y}^v\}, T\left[\left\langle i\kappa\bk\pi_{+, u}^j\{\vec{y}^v\}, p_\kappa^{\vec{\varrho}^{\bar{v}}} \right\rangle e_j^{u}\right]^{-1}p_\kappa^{\vec{\varrho}^{\bar{v}}} \otimes \sum_{i=1}^3\varrho_i^{\bar{v},\prime}e_i^{\bar{u}} \right\rangle \right\} \right).
\end{align*}
Denote \beq
\nu_{\bar{u}}^i(1; \bar{v}, \kappa) =\left\langle \sqrt{-1}\kappa\bk\pi_{+,\bar{u}}^i\{\vec{y}^v\}, p_\kappa^{\vec{\varrho}^{\bar{v}}} \right\rangle :=
\sqrt{-1}\kappa\bk\sum_{v=1}^{\on}\int_I\ \langle p_\kappa^{\vec{y}^v(s)}, p_\kappa^{\vec{\varrho}^{\bar{v}}} \rangle\ y_i^{v,\prime}(s)\ ds \nonumber \eeq and
\beq
\nu_{\bar{u}}^i(p;\kappa) = \left\langle \sqrt{-1}\kappa\bk\pi_{+,\bar{u}}^i\{\vec{y}^v\}, p_\kappa^{\vec{x}} \right\rangle :=
\sqrt{-1}\kappa\bk\sum_{v=1}^{\on}\int_I\ \langle p_\kappa^{\vec{y}^v(s)}, p_\kappa^{\vec{x}} \rangle\ y_i^{v,\prime}(s)\ ds. \nonumber \eeq

Note that $\nu_{\bar{u}}^i(1;\bar{v}, \kappa) \in \bC$ and $\nu_{\bar{u}}^i(1;\bar{v}, \kappa)e_i^{\bar{u}}, \nu_u^j(p,\kappa)e_j^u \in \bC \otimes W$. Refer to Examples \ref{ex.l.2} and \ref{ex.l.3}. Then we can rewrite our expression for the path integral as
\begin{align}
  \bk^p &\int_{M^p}G^p\left\{-i\left\langle \kappa\bk\pi_+\{\vec{y}^v\}, T\left[\nu_{u}^j(p;\kappa) e_j^{u}\right]^{-1}p_\kappa^{\vec{x}}\otimes e_i^{\bar{u}} \right\rangle \right\} \nonumber\\
&\otimes \exp\left(-i\bk
  \int_{\uL}
  \underline{G}^1\left\{ \left\langle \kappa\bk\pi_+\{\vec{y}^v\}, T\left[\nu_{u}^j(1;\bar{v}, \kappa) e_j^{u}\right]^{-1}p_\kappa^{\vec{\varrho}^{\bar{v}}} \otimes \sum_{i=1}^3\varrho_i^{\bar{v},\prime}e_i^{\bar{u}} \right\rangle \right\} \right), \label{e.cs.4}
\end{align}
with \beq \pi_+\{\vec{y}^v\} := \sum_{v=1}^{\on}\sum_{i=1}^3 \sum_{u=1}^d \int_I\ ds\ p_\kappa^{\vec{y}^v(s)} \ y_i^{v,\prime}(s)\ \otimes e_i^u\otimes E^u_{s} . \nonumber \eeq
\end{enumerate}
\end{defn}

\begin{rem}
\begin{enumerate}
  \item In Item \ref{d.cs.r.1d}, we scaled by factors $\bk$ for each integral over a submanifold. The reason for this scaling originated from our work in \cite{CS-Lim01} and \cite{CS-Lim03}, whereby we scaled each line integral by a factor $\kappa/2$, instead of $\bk$. But for the line integral involving $\oL$, we need an extra $\kappa$. The reason is because we need this extra factor to obtain non-trivial limits when we take $\kappa$ going to infinity. Such an extra factor is also used in \cite{CS-Lim01} and in \cite{CS-Lim03}.
  \item Note that $T\left[\nu_{u}^j(1;\bar{v}, \kappa) e_j^{u}\right]^{-1}$ can be defined, depending on how we define $D$ and $L$. And it acts on $p_\kappa^{\vec{\varrho}^{\bar{v}}} \otimes e_i^{\bar{u}} \in L^2(\tilde{M}) \otimes W$. Typically, \beq T\left[\nu_{u}^j(1;\bar{v}, \kappa) e_j^{u}\right]^{-1}p_\kappa^{\vec{\varrho}^{\bar{v}}} \otimes \sum_{\bar{u}=1}^d\sum_{i=1}^3\varrho_i^{\bar{v},\prime} e_i^{\bar{u}} \equiv f_{\bar{u}}^i(\kappa;\vec{\varrho}^{\bar{v}}) \otimes e_i^{\bar{u}}, \nonumber \eeq whereby $f_{\bar{u}}^i(\kappa;\vec{\varrho}^{\bar{v}})$ is a $C^\infty$ bounded function on $\tilde{M}$.

      Thus,
      \begin{align*}
      &\left\langle \kappa\bk\pi_+\{\vec{y}^v\}, T\left[\nu_{u}^i(1;\bar{v}, \kappa) e_i^{u}\right]^{-1}p_\kappa^{\vec{\varrho}^{\bar{v}}} \otimes \sum_{\bar{u}=1}^d\sum_{j=1}^3\varrho_j^{\bar{v},\prime} e_j^{\bar{u}} \right\rangle \\
      &\hspace{2cm}= \kappa\bk\sum_{v=1}^{\on}\int_I\ ds\ \left\langle p_\kappa^{\vec{y}^v(s)}, f_{u}^i(\kappa;\vec{\varrho}^{\bar{v}}) \right\rangle\ y_i^{v,\prime}(s)\ \otimes E^u_{s}
      \end{align*}
      and \beq \left\langle p_\kappa^{\vec{y}^v(s)}, f_{u}^i(\kappa;\vec{\varrho}^{\bar{v}}) \right\rangle \equiv \int_{\bR^4}p_\kappa^{\vec{y}^v(s)} \cdot f_{u}^i(\kappa;\vec{\varrho}^{\bar{v}})\ d\lambda \nonumber \eeq is dependent on $s$. (See Section \ref{s.ss}.)


  \item The term \beq \int_{\uL}
  \underline{G}^1\left\{ \left\langle \kappa\bk\pi_+\{\vec{y}^v\}, T\left[\nu_{u}^j(1;\bar{v}, \kappa) e_j^{u}\right]^{-1}p_\kappa^{\vec{\varrho}^{\bar{v}}} \otimes \sum_{i=1}^3\varrho_i^{\bar{v},\prime}e_i^{\bar{u}} \right\rangle \right\} \nonumber \eeq is expressed explicitly as \beq \sum_{\bar{v}=1}^{\un}\sum_{v=1}^{\on} \int_{I^2}\ dsd\bar{s} \ \kappa\bk\left\langle p_\kappa^{\vec{y}^v(s)}, f_{u}^i(\kappa;\vec{\varrho}^{\bar{v}}(\bar{s}) \right\rangle\ y_i^{v,\prime}(s)\ \otimes E^u_{s} . \nonumber \eeq

  \item The term \begin{align}
\int_{M^p}G^p &\left\{-i\left\langle \kappa\bk\pi_+\{\vec{y}^v\}, T\left[\nu_{u}^j(p;\kappa) e_j^{u}\right]^{-1}(p_\kappa^{\vec{x}}\otimes e_j^{u}) \right\rangle \right\} \nonumber\\
&\equiv \int_{M^p}G^p\left\{i\left\langle \kappa\bk T\left[\nu_{u}^j(p;\kappa) e_j^{u}\right]^{-1}\pi_+\{\vec{y}^v\}, p_\kappa^{\vec{x}}\otimes e_i^{\bar{u}} \right\rangle \right\}. \label{e.g.1}
\end{align}
deserves some explanation.

Now, \beq T\left[\nu_{\bar{u}}^i(p;\kappa) e_i^{\bar{u}}\right]^{-1}\pi_+\{\vec{y}^v\} \equiv \sum_{v=1}^{\on}\sum_{j=1}^3 \sum_{u=1}^d \int_I\ ds\ T\left[\nu_{\bar{u}}^i(p;\kappa) e_i^{\bar{u}}\right]^{-1}\left[ p_\kappa^{\vec{y}^v(s)} \otimes y_j^{v,\prime}(s)e_j^u \right]  \otimes E^u_{s}  \nonumber \eeq lies inside $W \otimes V$.
This linear operator $T\left[\nu_{u}^j(p;\kappa) e_j^{u}\right]^{-1}$ acts on $p_\kappa^{\vec{y}^v(s)} \otimes y_i^{v,\prime}(s)e_i^u$.

Recall that $V$ is some Lie Algebra. Hence the RHS of Equation (\ref{e.g.1}) will involve $g(E)$, whereby $g$ is some continuous function and $E$ will be a matrix. There will be some issues on how to define $g(E)$, as it is not true that $g(E)$ can be make sense of. In certain cases, we will show that $g(E)$ can be rigorously defined. So, we will leave the matter as it is and address it later for specific examples to come.

  \item Later, we will show that the terms $\kappa\nu_{\bar{u}}^i(p;\kappa) \rightarrow 0$ and $\kappa\nu_{\bar{u}}^i(1;\bar{v}, \kappa) \rightarrow 0$ as $\kappa$ goes to 0. To compute the limit of Expression \ref{e.cs.4} as $\kappa$ goes to infinity, it suffices to compute the limit of
\begin{align*}
  \bk^p \int_{M^p}G^p&\left\{i\left\langle \kappa\bk D^{-1}\pi_+\{\vec{y}^v\}, p_\kappa^{\vec{x}}\otimes e_i^{\bar{u}} \right\rangle \right\} \\
&\otimes \exp\left[-i\bk
  \int_{\uL}
  \underline{G}^1\left\{ \left\langle \kappa\bk\pi_+\{\vec{y}^v\}, D^{-1}\left(p_\kappa^{\vec{\varrho}^{\bar{v}}} \otimes \sum_{i=1}^3\varrho_i^{\bar{v},\prime}e_i^{\bar{u}} \right) \right\rangle \right\} \right].
\end{align*}

\end{enumerate}

\end{rem}

\section{Einstein-Hilbert Path Integral}\label{s.eha}

We can now begin to define the path integral in Expression \ref{e.eha.1}. First, we will need to rewrite it to make it look like the Chern-Simons path integral, as defined in Expression \ref{e.cs.3}.

\begin{notation}\label{n.n.6}
Recall we defined $A^i_{\alpha\beta}$ and $B^i_\alpha$ in Equations (\ref{e.o.1}) and (\ref{e.o.2}) respectively. Define the following 3-vectors,
\begin{align*}
B^i =& (B_1^i, B_2^i, B_3^i),\ \ B_0 = (B_0^1, B_0^2, B_0^3),\ \ B_i = (B_i^1, B_i^2, B_i^3), \\
A^i_0 =& (A^i_{01}, A^i_{02}, A^i_{03}),\ \ A^i = (A_{23}^i, A_{31}^i, A_{12}^i), \\
A_{23} =& (A_{23}^1, A_{23}^2, A_{23}^3),\ \ A_{31} = (A_{31}^1, A_{31}^2, A_{31}^3),\ \ A_{12} = (A_{12}^1, A_{12}^2, A_{12}^3).
\end{align*}
For 3-vectors $x$, $y$, $x \cdot y$ will denote the usual dot product and $x \times y$ will denote the cross product.

Write $B = (B^1, B^2, B^3)$ and $A_0 = (A_0^1, A_0^2, A_0^3)$. Let $\tilde{B}^i = B_i \times B_0$ and hence write $\tilde{B} = (\tilde{B}^1, \tilde{B}^2, \tilde{B}^3)$, $A = (A_{23}, A_{31}, A_{12})$. These vectors thus defined are 9-vectors. By abuse of notation, we write for a 9-vector $F = (F^1, F^2, F^3)$, each $F^i$ is a 3-vector, \beq B \vec{\times} F = \frac{1}{2}\left(
B^2 \times F^3- B^3 \times F^2, B^3 \times F^1- B^1 \times F^3, B^1 \times F^2- B^2 \times F^1 \right). \nonumber \eeq

In particular,
\begin{align*}
B \vec{\times} B \equiv& \frac{1}{2}\left(
B^2 \times B^3- B^3 \times B^2, B^3 \times B^1- B^1 \times B^3, B^1 \times B^2- B^2 \times B^1 \right)  \\
=& \left( B^2 \times B^3, B^3 \times B^1, B^1 \times B^2 \right) =\left( [B \vec{\times} B]_1, [B \vec{\times} B]_2, [B \vec{\times} B]_3\right),
\end{align*}
which is a 9-vector. Finally, we write \beq \partial_0 A_0 \cdot B \vec{\times}B := \partial_0 A_0^i \cdot [B \vec{\times} B]_i. \nonumber \eeq We will also use the dot to denote the usual scalar product for 9-vectors.
\end{notation}

With Notation \ref{n.n.6}, we have the following proposition.

\begin{prop}
Using Equations (\ref{e.o.1}) and (\ref{e.o.2}), the Einstein-Hilbert action given in Equation (\ref{e.eh.2}) can be written as
\begin{align*}
\int_{\bR^4} \partial_0 A_0  \cdot B \vec{\times} B
- \int_{\bR^4} \partial_0 A \cdot \tilde{B} ,
\end{align*}
whereby it is understood we integrate over Lebesgue measure on $\bR^4$.
\end{prop}

\begin{proof}
Refer to Notation \ref{n.su.1}. Suppose that $\omega$ and $e$ are given by Equations (\ref{e.o.1}) and (\ref{e.o.2}). We have
\begin{align*}
e \wedge e =& B^i_\gamma \otimes E^\gamma \otimes dx_i \wedge B^j_\mu \otimes E^\mu \otimes dx_j \\
=& \frac{1}{2}B^i_\gamma B^j_\mu \otimes ( E^\gamma \otimes E^\mu - E^\mu \otimes E^\gamma)\otimes  dx_i \wedge dx_j \\
=& B^1_\gamma B^2_\mu \otimes E^{\gamma \mu}\otimes dx_1  \wedge dx_2 + B^2_\gamma B^3_\mu \otimes E^{\gamma \mu} \otimes dx_2 \wedge dx_3 +B^3_\gamma B^1_\mu \otimes E^{\gamma \mu} \otimes dx_3 \wedge dx_1.
\end{align*}

And
\begin{align*}
d\omega + \omega \wedge \omega =& \partial_0 A^i_{\alpha\beta} \otimes E^{\alpha\beta}\otimes  dx_0 \wedge dx_i + \frac{1}{2}A^i_{\alpha\beta} A^j_{\gamma\mu}\otimes [E^{\alpha\beta}, E^{\gamma\mu}]\otimes  dx_i \wedge dx_j \\
&+ \frac{\partial}{\partial x_i} A^j_{\alpha\beta}\otimes E^{\alpha\beta}\otimes  dx_i \wedge dx_j,\ \ \partial_0 \equiv \frac{\partial}{\partial x_0}.
\end{align*}

Thus, the Einstein-Hilbert action becomes
\begin{align*}
\frac{1}{8}&\int_{\bR^4} \epsilon^{abcd} [e \wedge e]_{ab} \wedge [R]_{cd} \\
=& \frac{1}{8}\int_{\bR^4}\epsilon^{abcd}B^1_\gamma B^2_\mu[E^{\gamma \mu}]_{ab} \cdot \partial_0 A^3_{\alpha\beta}[E^{\alpha\beta}]_{cd} dx_1\wedge dx_2 \wedge dx_0 \wedge dx_3\\
+& \frac{1}{8}\int_{\bR^4}\epsilon^{abcd}B^2_\gamma B^3_\mu[E^{\gamma \mu}]_{ab} \cdot \partial_0 A^1_{\alpha\beta}[E^{\alpha\beta}]_{cd} dx_2\wedge dx_3 \wedge dx_0 \wedge dx_1\\
+&\frac{1}{8}\int_{\bR^4}\epsilon^{abcd}B^3_\gamma B^1_\mu[E^{\gamma \mu}]_{ab} \cdot \partial_0 A^2_{\alpha\beta}[E^{\alpha\beta}]_{cd} dx_3\wedge dx_1 \wedge dx_0 \wedge dx_2.
\end{align*}

The integral can be written as
\begin{align*}
& \frac{1}{2}\int_{\bR^4} B^1 \times B^2 \cdot \partial_0 A^3_{0} +  B^2 \times B^3 \cdot \partial_0 A^1_{0} + B^3 \times B^1 \cdot \partial_0 A^2_{0} \\
&-\frac{1}{2}\int_{\bR^4}  B^2 \times B^1 \cdot \partial_0 A^3_{0} +  B^3 \times B^2 \cdot \partial_0 A^1_{0} + B^1 \times B^3 \cdot \partial_0 A^2_{0} \\
&- \int_{\bR^4}  B^1 B_0^2 \cdot \partial_0 A^3 +  B^2 B_0^3 \cdot \partial_0 A^1 +  B^3  B_0^1\cdot\partial_0 A^2\\
&+\int_{\bR^4}  B^2  B_0^1 \cdot \partial_0 A^3 +  B^3  B_0^2\cdot\partial_0 A^1 +  B^1  B_0^3\cdot\partial_0 A^2,
\end{align*}
which can be further simplified into
\begin{align*}
\frac{1}{2}&\int_{\bR^4}  B^1 \times B^2 \cdot \partial_0 A^3_{0} +  B^2 \times B^3 \cdot \partial_0 A^1_{0} + B^3 \times B^1 \cdot \partial_0 A^2_{0} \\
&-\frac{1}{2}\int_{\bR^4}  B^2 \times B^1 \cdot \partial_0 A^3_{0} +  B^3 \times B^2 \cdot \partial_0 A^1_{0} + B^1 \times B^3 \cdot \partial_0 A^2_{0}\\
&- \int_{\bR^4} \partial_0 A_{23} \cdot [B_1  \times B_0] + \partial_0 A_{31} \cdot [B_2 \times B_0 ] + \partial_0 A_{12}\cdot [B_3 \times B_0 ].
\end{align*}

Using Notation \ref{n.n.6}, the Einstein-Hilbert action can be written as
\begin{align*}
&\int_{\bR^4} \partial_0 A_0  \cdot B \vec{\times} B
- \int_{\bR^4} \partial_0 A \cdot \tilde{B} .
\end{align*}
\end{proof}

If we treat $\tilde{B}$ as an independent variable, then we are lead to define a path integral with product measure
\beq \exp\left[ i\int_{\bR^4} \partial_0 A_0\cdot B \vec{\times} B   - \partial_0 A \cdot \tilde{B} \right]  D[B]D[A_0]D[\tilde{B}]D[A], \nonumber \eeq with all $A_0$, $A$, $B$ and $\tilde{B}$ being independent variables. Write \beq D\Lambda = D[B]D[A_0]D[\tilde{B}]D[A]. \nonumber \eeq

Notice that there are 2 measures, namely \beq \exp\left[ i\int_{\bR^4} \partial_0 A_0\cdot B \vec{\times} B\right]D[B]D[A_0]\ {\rm and}\ \exp\left[ -i\int_{\bR^4} \partial_0 A \cdot \tilde{B}\right]D[\tilde{B}]D[A]. \nonumber \eeq

To reconcile with the model in Subsection \ref{ss.cs}, we see that $W = \bR^9$. The Lie Algebra we have in mind will be $V = \mathfrak{su}(2) \times \mathfrak{su}(2)$. The Hilbert space we need to consider will hence be \beq H = \overline{\mathcal{S}}_\kappa(\bR^4) \otimes \bR^9 \otimes \mathfrak{su}(2) \times \mathfrak{su}(2). \nonumber \eeq

\begin{notation}
Let $\bar{E} = (1, 1, 1)$ be a 3-vector.
\end{notation}

\begin{notation}\label{n.x.1}
Refer to Notations \ref{n.n.5} and \ref{n.su.1}. Consider two oriented hyperlinks, $\oL = \{\ol^u \}_{u=1}^{\on}$, $\uL = \{\ul^v \}_{v=1}^{\un}$ in $\bR^4$, which together form an oriented hyperlink $\chi(\oL, \uL)$. Color the former $\oL$ with a representation for each component. Let $q \in \bR$ be known as a charge. Define
\begin{align*}
V(\{\ul^v\}_{v=1}^{\un})(\{B^i_\gamma\}) :=& \exp\left[ \sum_{v=1}^{\un} \int_{\ul^v} \bar{E} \cdot B^i \otimes dx_i + B_0^i \otimes dx_i\right], \\
W(q; \{\ol^u, \rho_u\}_{u=1}^{\on})(\{A^k_{\alpha\beta}\}) :=& \prod_{u=1}^{n}\Tr_{\rho_u}  \mathcal{T} \exp\left[ q\int_{\ol^u} A^i_{0j} \otimes dx_i\otimes \hat{E}^{0j}  + A_{\tau(j)}^k \otimes dx_k \otimes \hat{E}^{\tau(j)}\right].
\end{align*}
\end{notation}

\begin{rem}
\begin{enumerate}
  \item Here, $\mathcal{T}$ is the time-ordering operator defined in Definition \ref{d.t.5}.
  \item The functional $W$ is to compute the holonomy of a connection, along a hyperlink $\oL$.
  \item The functionals $V$ and $W$ will appear in Sections \ref{s.ao}, \ref{s.vo} and \ref{s.co}.
\end{enumerate}

\end{rem}

\section{Area Path Integral}\label{s.ao}

Refer to Notation \ref{n.s.3}. Fix a closed and bounded orientable surface, with or without boundary, denoted by $S$, inside $\bR^3$. We allow $S$ to be disconnected, with finite number of components. Parametrize it using \beq \sigma: \hat{t} \equiv (t, \bar{t}) \in I^2 \mapsto \sigma(\hat{t}) \in \bR^3 \nonumber \eeq and let $J_\sigma$ denote the Jacobian of $\sigma$. We will write $J_\sigma = (J_{23}, J_{31}, J_{12})$.

We can always write $S$ as a finite disjoint union of surfaces because we allow ambient isotopy of $S$, so without any loss of generality, we assume that the surface lies in the $x_2-x_3$ plane. Hence, $S \subset \bR^3$ with normal given by $e_1= (1, 0, 0)$. Write $\vec{\sigma} = (0, \sigma)$. We will also assume that the projected link $\pi_0(\oL)$ intersects $S$ at finitely many points.

Using the dynamical variables $\{B_\gamma^i\}$ and the Minkowski metric $\eta^{ab}$, we see that the metric $g^{ab} \equiv B^a_\mu\eta^{\mu\gamma}B^b_\gamma$ and the
area is given by \beq {\rm Area\ of}\ S(\{B_\mu^i\}) \equiv A_S(\{B_\mu^i\}) := \int_S \sqrt{g^{22}g^{33} - (g^{23})^2} \ dA. \nonumber \eeq Explicitly, \beq A_S(\{B_\mu^i\}) = \int_{I^2} \sqrt{g^{22}g^{33} - (g^{23})^2}(\vec{\sigma}(\hat{t}))\ J_{23}(\hat{t})\ d\hat{t}. \nonumber \eeq

Refer to Notation \ref{n.x.1}. We want to define the following area path integral,
\begin{align}
\frac{1}{Z_{EH}}\int V( \{\ul^v\}_{v=1}^{\un})(\{B^i_\mu\})&W(q; \{\ol^u, \rho_u\}_{u=1}^{\on})(\{A^k_{\alpha\beta}\})\nonumber\\
& \times  A_S(\{B^i_\mu\}) \exp\left[i\int_{\bR^4} \partial_0 A_0\cdot B \vec{\times} B   - \partial_0 A \cdot \tilde{B}\right] D\Lambda, \label{e.a.3}
\end{align}
with
\begin{align*}
Z_{EH} = \int \exp\left[i\int_{\bR^4} \partial_0 A_0\cdot B \vec{\times} B   - \partial_0 A \cdot \tilde{B}\right] D\Lambda.
\end{align*}

\begin{rem}\label{r.ax.1}
\begin{enumerate}
  \item When $S$ is the empty set, we define $A_\emptyset \equiv 1$, so we write Expression \ref{e.a.3} as $Z(q; \chi(\oL, \uL))$, which in future be termed as the Wilson Loop observable of the colored hyperlink $\chi(\oL, \uL))$.
  \item We will write Expression \ref{e.a.3} as $\hat{A}_S[Z(q; \chi(\oL, \uL))]$.
\end{enumerate}
\end{rem}

We will now make use of Definition \ref{d.cs.r.1} to make sense of the path integral in Expression \ref{e.a.3}. We will go through the steps in Definition \ref{d.cs.r.1} in detail.

Let $e_1=(1,0,0)$ be a 3-vector. Firstly, note that $g^{ij} = B^i \cdot B^j - B_0^i \cdot B_0^j$ and all other entries are 0. Secondly, it is straightforward to see that \beq g^{22}g^{33} - (g^{23})^2 = |B^2 \times B^3 |^2 - |B_1 \times B_0 \cdot e_1|^2 - |B_2 \times B_0 \cdot e_1|^2 - |B_3 \times B_0 \cdot e_1|^2. \nonumber \eeq

Refer to Notation \ref{n.n.6}. Observe that $B^2 \times B^3 - B^3 \times B^2 = 2 B^2 \times B^3$. Thus, we can also write  \beq g^{22}g^{33} - (g^{23})^2 = |[B \vec{\times} B]_1 |^2 - |B_1 \times B_0 \cdot e_1|^2- |B_2 \times B_0 \cdot e_1|^2 - |B_3 \times B_0 \cdot e_1|^2. \label{e.a.7} \eeq

\begin{notation}\label{n.c.1}
Write $m(B) := B \vec{\times}$, $m(B^i) := B^i \times$ and $m(B_i) := B_i \times$ as operators, so we have \beq m(B^i)B^j\equiv B^i \times B^j,\ m(B_i)B_0 \equiv B_i \times B_0\ {\rm and}\ m(B)B \equiv B \vec{\times} B. \nonumber \eeq

Write $\bar{B} = m(B)^{-1}B$ and $\bar{B}^j$ refers to the $j$-th spatial component, which is a 3-vector. So, $B = m(B)^{-1} (B \vec{\times} B)$. Also recall $\tilde{B}^i = B_i \times B_0 = m(B_i)B_0$, so we can write $B_0 = m(B_i)^{-1}\tilde{B}^i$. There is no sum over the index $i$.

For a 3-vector $v \equiv (v_1, v_2, v_3)$, we write $m(B)^{-1}v$ to mean $m(B)^{-1}$ acting on the 9-vector $(v_1\bar{E}, v_2\bar{E}, v_3\bar{E})$, $\bar{E} = (1,1,1)$. And $[m(B)^{-1}v]^j$ will also refer to the $j$-th spatial component, which is a 3-vector, with $[m(B)^{-1}v]^j_i$ referring to its $i$-th component. We will also write $m(B)^{-1}(v \otimes \bar{E}) \equiv m(B)^{-1}v$.
\end{notation}

\begin{defn}
Let $\delta^{\vec{x}}$ be the Dirac-delta function, i.e. for any function $f$, $\langle f, \delta^{\vec{x}} \rangle = f(\vec{x})$. Also refer to Notation \ref{n.bk.1}. Define
\begin{align}
\tilde{V}(\{B^j_\mu\})
&:= \exp\Bigg\{ \sum_{j=1}^3\left\langle B^j, -\sum_{v=1}^{\underline{n}} \int_0^1 \left\{ \left[ m(\langle B, \delta^{\vec{\varrho}^v_{\bar{s}}} \rangle ) \right]^{-1} (\varrho_{\bar{s}}^{v,\prime}\delta^{\vec{\varrho}^v_{\bar{s}}})\right\}^j \right\rangle\ d\bar{s} \nonumber\\
& \hspace{3cm} + \left\langle \tilde{B}^{i}, -\sum_{v=1}^{\underline{n}} \int_0^1 \left[m(\langle B_i, \delta^{\vec{\varrho}^v_{\bar{s}}} \rangle ) \right]^{-1}(\varrho_{\bar{s}}^{v,\prime}\delta^{\vec{\varrho}^v_{\bar{s}}})\right\rangle \ d\bar{s}
\Bigg\}, \label{e.w.2}\\
\tilde{W}(\{A^k_{\alpha\beta}\})
&  := \prod_{u=1}^{\on} \Tr_{\rho_u}  \mathcal{T} \exp\Bigg[\left\langle A^k_{0i}, -q\int_0^1 ds\  \partial_0^{-1}\delta^{\vec{y}^u_s}\right\rangle y^{u,\prime}_{k,s}\otimes \hat{E}^{0i}_s \nonumber\\
& \hspace{3cm}+ \left\langle A^k_{\tau(j)}, -q\int_0^1 ds\  \partial_0^{-1}\delta^{\vec{y}^u_s}\right\rangle y^{u,\prime}_{k,s}\otimes \hat{E}^{\tau(j)}_s
\Bigg]. \label{e.w.3}
\end{align}
\end{defn}

\begin{lem}\label{l.x.1}
Refer to the parametrizations $\vec{y}^u$ and $\vec{\varrho}^v$ defined in Notation \ref{n.n.3}. After doing a change of variables given in Notation \ref{n.c.1}, the path integral in Expression \ref{e.a.3} becomes \beq
\frac{1}{\hat{Z}_{EH}}\int \tilde{V}(\{B^j_\mu\}) \tilde{W}(\{A^k_{\alpha\beta}\})\tilde{A}_S(B, \{\tilde{B}^i\})  e^{i\int_{\bR^4} A_0\cdot  B   -  A \cdot \tilde{B}}  D\Lambda \label{e.a.4} \eeq with \beq
\hat{Z}_{EH} = \int \exp\left[i\int_{\bR^4} A_0\cdot  B   -  A \cdot \tilde{B}\right]
D\Lambda, \nonumber \eeq
after applying Steps \ref{d.cs.r.1a} and \ref{d.cs.r.1b} in Definition \ref{d.cs.r.1}.

Here, $\tilde{V}$ and $\tilde{W}$ are defined by Equations (\ref{e.w.2}) and (\ref{e.w.3}) and \beq \tilde{A}_S(B, \{\tilde{B}^i\}) := \int_{I^2} \sqrt{\left|\left\langle B^1, \delta^{\vec{\sigma}(\hat{t})} \right\rangle \right|^2 - \sum_{j=1}^3\left|\left\langle \tilde{B}^j, \delta^{\vec{\sigma}(\hat{t})} \right\rangle \cdot e_1\right|^2 }\ J_{23}(\hat{t})\ d\hat{t}. \nonumber \eeq
\end{lem}

\begin{proof}
Firstly, we will write
\begin{align*}
V(\{\ul^v\}_{v=1}^{\un})(\{B^j_\mu\} )  =& \exp\left[ \sum_{v=1}^{\un}\int_0^1 d\bar{s}\ \left\langle \bar{E} \cdot B^j, \delta^{\vec{\varrho}^v_{\bar{s}}}\right\rangle \varrho_{j,\bar{s}}^{v,\prime} + \left\langle B_0^j, \delta^{\vec{\varrho}^v_{\bar{s}}}\right\rangle\varrho_{j,\bar{s}}^{v,\prime}\right], \\
W(q; \{\ol^u\}_{u=1}^{\on})(\{A^k_{\alpha\beta}\} ) =& \prod_{u=1}^{\on}\Tr_{\rho_u}  \mathcal{T}\exp\left[ q\int_0^1 ds\ \left\langle A_{0j}^k, \delta^{\vec{y}^u_s}\right\rangle y_{k,s}^{u,\prime}\otimes \hat{E}^{0j}_s +  \left\langle A_{\tau(j)}^k, \delta^{\vec{y}^u_s}\right\rangle y_{k,s}^{u,\prime}\otimes \hat{E}^{\tau(j)}_s\right].
\end{align*}
Here, $\hat{E}^{\alpha\beta}_s \equiv \hat{E}^{\alpha\beta}$ and $s$ keeps track of the ordering. This is Step \ref{d.cs.r.1a} in Definition \ref{d.cs.r.1}.

Refer to Notation \ref{n.c.1}. We want to replace $B \vec{\times}B \equiv m(B)B$ with $B$, $\partial_0A_0$ with $A_0$
and $\partial_0 A$ with $A$, so that the action becomes $\int_{\bR^4}A_0 \cdot B - A \cdot \tilde{B}$.

This means we need to replace $B_0$ with  $m(B_i)^{-1}\tilde{B}^i$, $B$ with $\bar{B}=m(B)^{-1}B$ and $A^k_{\alpha\beta}$ with $\partial_0^{-1}A^k_{\alpha\beta}$, hence
\begin{align*}
V( \{\ul^v\}_{v=1}^{\un})(\{B^j_\mu\}) \longmapsto& \exp\left[ \sum_{v=1}^{\un}\int_0^1 d\bar{s}\ \left\langle \bar{E} \cdot \bar{B}^j, \delta^{\vec{\varrho}^v_{\bar{s}}}\right\rangle \varrho_{j,\bar{s}}^{v,\prime} + \left\langle [m(B_i)^{-1}\tilde{B}^i]^j, \delta^{\vec{\varrho}^v_{\bar{s}}}\right\rangle\varrho_{j,\bar{s}}^{v,\prime}\right]\\
\equiv& \exp\left[ \sum_{v=1}^{\un}\int_0^1 d\bar{s}\ \left\langle \bar{E} \cdot \bar{B}^j, \delta^{\vec{\varrho}^v_{\bar{s}}}\right\rangle \varrho_{j,\bar{s}}^{v,\prime} + \left\langle [m(B_i(\vec{\varrho}^v_{\bar{s}}))^{-1}\tilde{B}^i]^j, \delta^{\vec{\varrho}^v_{\bar{s}}}\right\rangle\varrho_{j,\bar{s}}^{v,\prime}\right], \\
W(q; \{\ol^u\}_{u=1}^{\on})(\{A^k_{\alpha\beta}\} ) \longmapsto& W(q; \{\ol^u\}_{u=1}^{\on})(\{\partial_0^{-1}A^k_{\alpha\beta} \}).
\end{align*}
Note that \beq \left\langle \bar{E} \cdot \bar{B}^j, \delta^{\vec{\varrho}^v_{\bar{s}}}\right\rangle \equiv \left\langle \bar{E} \cdot [m(B(\vec{\varrho}^v_{\bar{s}}))^{-1}B]^j, \delta^{\vec{\varrho}^v_{\bar{s}}}\right\rangle. \nonumber \eeq

Because $m(B_i)$ and $m(B)$ are skew-symmetric operators, therefore $m(B_i)^{-1}$ and $m(B)^{-1}$ are skew-symmetric operators. Hence
\begin{align}
\exp&\left[ \sum_{v=1}^{\un}\int_0^1 d\bar{s}\ \left\langle \bar{E} \cdot \bar{B}^j, \delta^{\vec{\varrho}^v_{\bar{s}}}\right\rangle \varrho_{j,\bar{s}}^{v,\prime} + \left\langle \left\{\left[m(B_i(\vec{\varrho}^v_{\bar{s}}))\right]^{-1}\tilde{B}^i\right\}^j, \delta^{\vec{\varrho}^v_{\bar{s}}}\right\rangle\varrho_{j,\bar{s}}^{v,\prime}\right] \equiv \tilde{V}(\{B^j_\mu\})\nonumber\\
:=& \exp\Bigg\{ \sum_{j=1}^3\left\langle B^j, -\sum_{v=1}^{\underline{n}} \int_0^1 \left\{[m(B(\vec{\varrho}^v_{\bar{s}}) )]^{-1} (\varrho_{\bar{s}}^{v,\prime}\delta^{\vec{\varrho}^v_{\bar{s}}})\right\}^j \right\rangle\ d\bar{s} \nonumber\\
& \hspace{2cm} + \left\langle \tilde{B}^{i}, -\sum_{v=1}^{\underline{n}} \int_0^1 \left[m(B_i (\vec{\varrho}^v_{\bar{s}}) )\right]^{-1}(\varrho_{\bar{s}}^{v,\prime}\delta^{\vec{\varrho}^v_{\bar{s}}})\right\rangle \ d\bar{s}
\Bigg\}. \label{e.s.1}
\end{align}

Replace $B(\vec{\varrho}^v_{\bar{s}})$ and $B_i (\vec{\varrho}^v_{\bar{s}})$ with $\langle B, \delta^{\vec{\varrho}^v_{\bar{s}}} \rangle$ and $\langle B_i, \delta^{\vec{\varrho}^v_{\bar{s}}}\rangle$ and we will obtain the term $\tilde{V}$.

And making the same substitutions inside Equation (\ref{e.a.7}), we will have the area integrand computed as
\begin{align*}
\int_S \sqrt{|B^1 |^2 - \sum_{j=1}^3|\tilde{B}^j \cdot e_1|^2 }\ dA
=& \int_{I^2} \sqrt{\left|B^1(\vec{\sigma}(\hat{t})) \right|^2 - \sum_{j=1}^3|\tilde{B}^j(\vec{\sigma}(\hat{t})) \cdot e_1|^2 }\ J_{23}(\hat{t})\ d\hat{t} \\
=& \int_{I^2} \sqrt{\left|\left\langle B^1, \delta^{\vec{\sigma}(\hat{t})} \right\rangle \right|^2 - \sum_{j=1}^3\left|\left\langle \tilde{B}^j, \delta^{\vec{\sigma}(\hat{t})} \right\rangle \cdot e_1\right|^2 }\ J_{23}(\hat{t})\ d\hat{t}.
\end{align*}

Note that $\partial_0^{-1}$ is skew-symmetric because $\partial_0$ is skew-symmetric by doing a simple integration by parts, i.e. $\langle \partial_0^{-1}A^k_{\alpha\beta}, B^j_\mu \rangle = \langle A^k_{\alpha\beta}, -\partial_0^{-1}B^j_\mu \rangle$. Hence
\begin{align*}
W&(q; \{\ol^u\}_{u=1}^{\on})(\{\partial_0^{-1}A^k_{\alpha\beta}\} )\equiv \tilde{W}(\{A^k_{\alpha\beta}\}) \\
 :=&  \prod_{u=1}^{\on} \Tr_{\rho_u}  \mathcal{T} \exp\left[\left\langle A^k_{0j}, -q\int_0^1 ds\  \partial_0^{-1}\delta^{\vec{y}^u_s}\right\rangle y^{u,\prime}_{k,s}\otimes \hat{E}^{0j}_s +
\left\langle A^k_{\tau(j)}, -q\int_0^1 ds\  \partial_0^{-1}\delta^{\vec{y}^u_s}\right\rangle y^{u,\prime}_{k,s}\otimes \hat{E}^{\tau(j)}_s
\right].
\end{align*}

\end{proof}

Before we proceed to the rest of the steps in Definition \ref{d.cs.r.1}, we first point out that $m(B(\vec{x}))^{-1}$, $m(B_i(\vec{x}))^{-1}$ cannot be defined. So, we need to make an approximation to these operators, which is the next definition. The reader may wish to compare with the one given in Definition \ref{d.lo.1}.

\begin{defn}
We refer the reader to Section \ref{s.lo}. Also recall $x(s_i)$ defined in Definition \ref{d.lo.1}.

For $i = 1, 2 , 3$, define $\partial^\kappa_{i} :=\frac{1}{\kappa}\partial_i$, which maps $C^\infty(\bR^4) \otimes \Lambda^1(\bR^3) \rightarrow C^\infty(\bR^4) \otimes \Lambda^2(\bR^3)$ and define for a 9-vector $F \equiv (F^1, F^2, F^3)$, $F^i \in (C^\infty(\bR^4)\otimes \Lambda^1(\bR^3))^{\times^3}$,
\begin{align*}
m_\kappa(B)F :=& \left( \partial_1^\kappa +  B^1, \partial_2^\kappa +  B^2, \partial_3^\kappa +  B^3\right) \vec{\times} F \\
\equiv& ([m_\kappa(B)F]^1, [m_\kappa(B)F]^2, [m_\kappa(B)F]^3),
\end{align*}
whereby for $(i,j,k) \in C_3$,
\begin{align*}
[m_\kappa(B)F]^i =& \left(\partial_j^\kappa F^k + \frac{1}{2}dx_j\wedge B^j \times F^k  \right) +  \left(\partial_k^\kappa F^j + \frac{1}{2}dx_k\wedge B^k \times F^j \right).
\end{align*}

Here, it is understood that $[m_\kappa(B)F]^i \in (C^\infty(\bR^4))^{\times^3}\otimes \Lambda^2(\bR^3) \equiv (C^\infty(\bR^4)\otimes \Lambda^2(\bR^3))^{\times^3}$ and
\begin{align*}
\partial_i^\kappa F^j \equiv& (\partial_i^\kappa F_1^j, \partial_i^\kappa F_2^j , \partial_i^\kappa F_3^j),\ F_i^j \in C^\infty(\bR^4)\otimes \Lambda^1(\bR^3), \\
B^i \times F^j \equiv& (B^i_2 F_3^j - B^i_3 F_2^j, B^i_3 F_1^j - B^i_1 F_3^j, B^i_1 F_2^j - B^i_2 F_1^j ),\ F_i^j \in C^\infty(\bR^4)\otimes \Lambda^1(\bR^3).
\end{align*}
See Items \ref{d.lo.2a} and \ref{d.lo.2b} in Definition \ref{d.lo.2}.

Let $F = (F^1, F^2, F^3) \equiv (F^1 \otimes dx_1, F^2 \otimes dx_2, F^3 \otimes dx_3)$, $F^i \in C^\infty(\bR^4)^{\times^3}$.
Write $v_1 = \frac{1}{2}dx_2 \wedge dx_3$, $v_2 = \frac{1}{2}dx_3 \wedge dx_1$ and $v_3 = \frac{1}{2}dx_1 \wedge dx_2$. Observe that as $\kappa \rightarrow \infty$,
\begin{align*}
m_\kappa(B)&F \longrightarrow \\
&\left((B^2 \times F^3 - B^3 \times F^2)\otimes v_1, (B^3 \times F^1 - B^1 \times F^3)\otimes v_2, (B^1 \times F^2 - B^2 \times F^1)\otimes v_3\right),
\end{align*}
which can be identified with $m(B)F\equiv B \vec{\times}F$ using the Hodge star operator.

The inverse is given by \beq m_\kappa(B)^{-1}F = ([m_\kappa(B)^{-1}F]^1, [m_\kappa(B)^{-1}F]^2, [m_\kappa(B)^{-1}F]^3), \label{e.m.2a} \eeq whereby for $(i,j,k) \in C_3$,
\begin{align*}
[m_\kappa(B)^{-1}F]^i =& m_\kappa(B^j)^{-1} F^k\otimes  dx_k - m_{\kappa}(B^k)^{-1} F^j\otimes  dx_j,
\end{align*}
with \beq m_\kappa(B^i)^{-1} F^j(x) := \frac{\kappa}{2}\left[\int_{-\infty}^{x_i} - \int_{x_i}^\infty\right]e^{\kappa(s_i - x_i)[\frac{1}{2}B^i \times]} F^j(x(s_i)) ds_i. \label{e.m.2b} \eeq

Here, $[B^i \times]: C^\infty(\bR^4)^{\times^3} \rightarrow C^\infty(\bR^4)^{\times^3}$ is viewed as a linear operator acting on 3-vectors, hence represented by a $3 \times 3$ matrix. Note that $[m_\kappa(B)^{-1}F]^j \in [C^\infty(\bR^4)\otimes \Lambda^1(\bR^3)]^{\times^3}$. In other words, $m_\kappa(B)^{-1}F$ is a 9-vector, each component taking values in $C^\infty(\bR^4) \otimes \Lambda^1(\bR^3)$.

A direct computation will show that
\begin{align*}
m_\kappa&(B)m_\kappa(B)^{-1} F \\
=& \left(F^1\otimes (dx_1 \wedge dx_2 + dx_3 \wedge dx_1), F^2 \otimes (dx_2 \wedge dx_3 + dx_1 \wedge dx_2), F^3 \otimes (dx_3 \wedge dx_1 + dx_2 \wedge dx_3)\right).
\end{align*}
Each component in the RHS of the preceding equation is in $[C(\bR^4) \otimes \Lambda^2(\bR^3)]^{\times^3}$. Using $\ddag$ defined in Section \ref{s.lo} will identify the RHS of the preceding equation with $F$.

\begin{rem}\label{r.lo.2}
\begin{enumerate}
  \item Explicitly, if $B^i = (b_1, b_2, b_3)$, then the matrix representing $[B^i \times]$ is given by \beq
\left(
  \begin{array}{ccc}
    0 &\ -b_3 &\ b_2 \\
    b_3 &\ 0 &\ -b_1 \\
    -b_2 &\ b_1 &\ 0 \\
  \end{array}
\right). \nonumber \eeq
  \item Refer to Remark \ref{r.lo.1}. When $B \equiv 0$, we have for $(i,j,k) \in C_3$,
\begin{align*}
[m_\kappa(0)^{-1}F]^i =& \kappa[\partial_j^{-1} F^k\otimes  dx_k - \partial_k^{-1} F^j\otimes  dx_j].
\end{align*}
Here, it is understood that $F^i \in C^\infty(\bR^4)^{\times^3}$ and the integral operator $\partial_i^{-1}$ acts on the 3-vector componentwise.
\end{enumerate}
\end{rem}

Let $\nabla \equiv (\partial_1, \partial_2, \partial_3)$. For $i=1,2,3$, we will also define $m_\kappa(B_i)$, by \beq m_\kappa(B_i) F = \left[ \frac{1}{\kappa} \nabla + B_i \right]\times F, \label{e.m.1} \eeq with $F= (F^1, F^2, F^3)\in (C^\infty(\bR^4) \otimes \Lambda^1(\bR^3))^{\times^3}$.

The components take values in $C^\infty(\bR^4) \otimes \Lambda^2(\bR^3)$ and is given by
\begin{align*}
[m_\kappa(B_{\bar{i}})F]^i =& \left(\partial_j^\kappa F^k + dx_j\wedge B_{\bar{i}}^j F^k  \right) +  \left( \partial_k^\kappa F^j + dx_k\wedge B_{\bar{i}}^k F^j \right),\ \bar{i} = 1,2, 3,
\end{align*}
for $(i,j,k) \in C_3$ and $\partial_i^\kappa \equiv \frac{1}{\kappa}\partial_i$.

Let $F = (F^1, F^2, F^3) \equiv (F^1 \otimes dx_1, F^2 \otimes dx_2, F^3 \otimes dx_3)$, $F^i \in C^\infty(\bR^4)$.
Write $w_1 = dx_2 \wedge dx_3$, $w_2 = dx_3 \wedge dx_1$ and $w_3 = dx_1 \wedge dx_2$. Observe that as $\kappa \rightarrow \infty$,
\begin{align*}
m_\kappa(B_i)&F \longrightarrow \\
&\left((B_i^2 F^3 - B_i^3 F^2)\otimes w_1, (B_i^3  F^1 - B_i^1 F^3)\otimes w_2, (B_i^1  F^2 - B_i^2  F^1)\otimes w_3\right),
\end{align*}
which can be identified with $m(B_i)F\equiv B_i \times F$ using the Hodge star operator.

Its inverse is defined by, for $(i,j,k) \in C_3$,
\begin{align}
[m_\kappa(B_{\bar{j}})^{-1}F]^i =& m_\kappa(B_{\bar{j}}^j)^{-1} F^k\otimes dx_k - m_{\kappa}(B_{\bar{j}}^k)^{-1} F^j\otimes dx_j,\ \bar{j} = 1,2,3, \label{e.m.3}
\end{align}
with \beq m_\kappa(B_i^j)^{-1} G(x) := \frac{\kappa}{2}\left[\int_{-\infty}^{x_j} - \int_{x_j}^\infty\right]e^{\kappa(s_j- x_j )B_i^j } G(x(s_j)) ds_j. \label{e.m.3b} \eeq Note that each component of the 3-vector $m_\kappa(B_i)^{-1} F$ is in $C^\infty(\bR^4) \otimes \Lambda^1(\bR^3)$.

A direct computation will show that
\begin{align*}
m_\kappa&(B_i)m_\kappa(B_i)^{-1} F \\
=& (F^1\otimes (dx_1 \wedge dx_2 + dx_3 \wedge dx_1), F^2 \otimes (dx_2 \wedge dx_3 + dx_1 \wedge dx_2), F^3 \otimes (dx_3 \wedge dx_1 + dx_2 \wedge dx_3)).
\end{align*}
Using $\ddag$ will identify the RHS of the preceding equation with $F$.

\begin{rem}\label{r.m.1}
It is interesting to consider $B_i = 0$. Then from Equation (\ref{e.m.1}), we see that $m_\kappa(0) \equiv \frac{1}{\kappa}\nabla \times$. Thus, we have \beq \nabla \times (m_\kappa(0))^{-1}F = \kappa F, \nonumber \eeq after we make the identification between $\Lambda^1(\bR^3)$ and $\Lambda^2(\bR^3)$ using $\ddag$.

Furthermore, for $(i,j,k) \in C_3$,
\begin{align*}
[m_\kappa(0)^{-1}F]^i =& \kappa[\partial_j^{-1} F^k\otimes  dx_k - \partial_k^{-1} F^j\otimes  dx_j].
\end{align*}
\end{rem}
\end{defn}

Refer to Notation \ref{n.n.3}. To make sense of the path integral given by Expression \ref{e.a.4}, we will make the following approximations.
\begin{enumerate}
  \item We will approximate the Dirac-delta function with a Gaussian function $p_\kappa$, defined in Notation \ref{n.n.1}.
  \item We will approximate the operators $m(B)$ and $m(B_i)$ with the operators $m_\kappa(B)$ and $m_\kappa(B_i)$ respectively.
\end{enumerate}

This completes Step \ref{d.cs.r.1c} in Definition \ref{d.cs.r.1}.

\begin{defn}\label{d.he.1}
Define
\begin{align*}
\tilde{V}_\kappa(\{B^j_\mu\}) := \exp\Bigg\{ \sum_{j=1}^3& \left\langle B^j, -\bk\sum_{v=1}^{\underline{n}} \int_0^1  \left\{ \left[ m_\kappa(\langle B, p_\kappa^{\vec{\varrho}^v_{\bar{s}}}\rangle) \right]^{-1}(\varrho_{\bar{s}}^{v,\prime}  p_\kappa^{\vec{\varrho}^v_{\bar{s}}}) \right\}^j\ d\bar{s} \right\rangle\\
&\hspace{1cm}+ \left\langle \tilde{B}^{i}, -\bk\sum_{v=1}^{\underline{n}} \int_0^1 \left[m_\kappa(\langle B_i, p_\kappa^{\vec{\varrho}^v_{\bar{s}}} \rangle) \right]^{-1}(\varrho_{\bar{s}}^{v,\prime} p_\kappa^{\vec{\varrho}^v_{\bar{s}}})\ d\bar{s}\right\rangle
\Bigg\},
\end{align*}
and
\begin{align}
\tilde{W}_\kappa(\{A^k_{\alpha\beta}\})
=&\exp\Bigg[ q\left\langle A_{0j}^i, -\kappa\bk\sum_{u=1}^{\on}\int_0^1 ds\  \partial_0^{-1} p_\kappa^{\vec{y}^u_s}\right\rangle y^{u,\prime}_{i,s}\otimes \hat{E}^{0j}_s \nonumber\\
&\hspace{3.5cm}+ q\left\langle A^k_{\tau(j)}, -\kappa\bk\sum_{u=1}^{\on}\int_0^1 ds\ \partial_0^{-1}p_\kappa^{\vec{y}^u_s}\right\rangle y^{u,\prime}_{k,s} \otimes \hat{E}^{\tau(j)}_s\Bigg]. \label{e.w.1}
\end{align}
\end{defn}

\begin{rem}\label{r.ax.2}
\begin{enumerate}
\item Note that $\tilde{W}_\kappa(\{A^k_{\alpha\beta}\})$ will take values inside $\bigoplus_{n=0}^\infty (\mathfrak{su}(2)\times \mathfrak{su}(2))_{\bC}^{\otimes^n}$.
\item Write $\nu = (\nu^1, \nu^2, \nu^3)$, a 9-vector by \beq \nu^i = \varrho_{i,\bar{s}}^{v,\prime}  p_\kappa^{\vec{\varrho}^v_{\bar{s}}} \bar{E},\ \bar{E} = (1, 1, 1). \nonumber \eeq
Note that $\langle B, p_\kappa^{\vec{\varrho}^v_{\bar{s}}}\rangle $ is a 9-vector, so \beq \left[ m_\kappa(\langle B, p_\kappa^{\vec{\varrho}^v_{\bar{s}}}\rangle) \right]^{-1}(\varrho_{\bar{s}}^{v,\prime}  p_\kappa^{\vec{\varrho}^v_{\bar{s}}}) \equiv \left[ m_\kappa(\langle B, p_\kappa^{\vec{\varrho}^v_{\bar{s}}}\rangle) \right]^{-1}\nu \equiv (F^1, F^2, F^3), \nonumber \eeq and its component, a 3-vector written as \beq F^j \equiv \left\{ \left[ m_\kappa(\langle B, p_\kappa^{\vec{\varrho}^v_{\bar{s}}}\rangle) \right]^{-1}(\varrho_{\bar{s}}^{v,\prime}  p_\kappa^{\vec{\varrho}^v_{\bar{s}}}) \right\}^j. \nonumber \eeq
\end{enumerate}
\end{rem}

Recall in Step \ref{d.cs.r.1d} in Definition \ref{d.cs.r.1}, we need to do some scaling. The line integrals are respectively scaled by $\bk$, so note that in the exponent of $\tilde{V}_\kappa$, we add a factor $\bk$. However, for the exponent in $\tilde{W}_\kappa$, we include the factor $\kappa\bk$ instead. The extra factor $\kappa$ is necessary to obtain non-trivial limits later.

Because $S$ is 2-dimensional submanifold, we need to include a factor of $\bk^2$ to the surface integral. Hence we need to replace $\tilde{A}_S(B, \{\tilde{B}^i\})$ with
\begin{align*}
\tilde{A}_{\kappa,S}(B, \{\tilde{B}^i\}) :=
\bk^2\int_{I^2} \sqrt{\left|\left\langle B^1, p_\kappa^{\vec{\sigma}(\hat{t})} \right\rangle \right|^2 - \sum_{j=1}^3\left|\left\langle \tilde{B}^j, p_\kappa^{\vec{\sigma}(\hat{t})} \right\rangle \cdot e_1 \right|^2 }\ J_{23}(\hat{t})\ d\hat{t}.
\end{align*}

\begin{defn}($\widetilde{\Tr}$)\label{d.tr.1}\\
Define a linear functional $\widetilde{\Tr}$ as follows. Suppose a matrix $A$ is indexed by time $s$ and representation $\rho(i)$, $i=1,\ldots, r$. In other words, $A \equiv A(\rho(i),s)$. Let $\{A(\pi_1,s_1),\ldots, A(\pi_n,s_n)\}$ be a finite set of matrices. Let $S_i = \{j \in \{1,\ldots,n \}: \pi_j = \rho(i)\}$ and write $m_i := |S_i|$. For any $n \geq 1$, define a linear operator,
\begin{align}
\widetilde{\Tr}: & A(\pi_1,s_1) \otimes \cdots \otimes A(\pi_n,s_n) \nonumber \\
\mapsto&  \Tr_{\rho(1)}[A(\rho(1),s_{\beta_1(1)})\cdots \otimes A(\rho(1),s_{\beta_1(m_1)}) ] \cdots \Tr_{\rho(r)} [A(\rho(r),s_{\beta_r(1)}) \cdots A(\rho(r),s_{\beta_r(m_r)})], \nonumber
\end{align}
such that for each $i=1,\ldots, r$, $s_{\beta_i(1)} > s_{\beta_i(2)} > \ldots > s_{\beta_i(m_i)}$ and $\beta_i(j) \in S_i$ for $j=1,\ldots, m_i$.
\end{defn}

Together with the preceding approximation, we replace $\tilde{W}$ and $\tilde{V}$ with $\tilde{W}_\kappa$ and $\tilde{V}_\kappa$ respectively in Definition \ref{d.he.1}. Because $\tilde{V}_\kappa(\{B^j_\mu\})$ and $\tilde{A}_{\kappa,S}(B, \{\tilde{B}^i\})$ are scalars, we can bring them inside the time ordering operator $\mathcal{T}$ and trace, which is the linear functional $\widetilde{\Tr}$. Thus we approximate our path integral in Equation (\ref{e.a.4}) with
\beq
\frac{1}{\hat{Z}_{EH}}\widetilde{\Tr}\int \tilde{V}_\kappa(\{B^j_\mu\}) \tilde{W}_\kappa(\{A^k_{\alpha\beta}\}) \tilde{A}_{\kappa,S}(B, \{\tilde{B}^i\})  e^{i\int_{\bR^4} A_0\cdot  B   -  A \cdot \tilde{B}}  D\Lambda \label{e.a.5}. \eeq This completes Steps \ref{d.cs.r.1c} and \ref{d.cs.r.1d} in Definition \ref{d.cs.r.1}.

\begin{defn}\label{d.n.3}
Given a colored oriented hyperlink $\oL \equiv \{\ol^u, \rho_u\}_{u=1}^{\on}$ and another oriented hyperlink $\uL \equiv \{\ul^v \}_{v=1}^{\un}$ as in Notation \ref{n.h.1}, we obtain a new colored oriented hyperlink $\chi(\oL, \uL)$. Recall we parametrize $\ol^u$ using $\vec{y}^u$ and $\ul^v$ using $\vec{\varrho}^v$ respectively from Notation \ref{n.n.3}.

Define \beq \widetilde{E} = (\hat{E}^{01},  \hat{E}^{02}, \hat{E}^{03}),\ \rho_u^+(\widetilde{E}_s) = (\rho_u^+(\hat{E}_s^{01}),  \rho_u^+(\hat{E}_s^{02}), \rho_u^+(\hat{E}_s^{03})) \nonumber \eeq and \beq  \lambda_\kappa^v = (\lambda_{\kappa}^{v,1} , \lambda_{\kappa}^{v,2} , \lambda_{\kappa}^{v,3} ),\ \
\tilde{\lambda}_{\kappa}^v = (\tilde{\lambda}_{\kappa}^{v,1}, \tilde{\lambda}_{\kappa}^{v,2}, \tilde{\lambda}_{\kappa}^{v,3}),
 \nonumber \eeq which is a 9-vector and 3-vector respectively, with
\begin{align}
\lambda_{\kappa}^{v,j}(\bar{s}) =& \frac{-iq\kappa^2}{2\sqrt{4\pi}}\sum_{u=1}^{\on}\int_0^1 ds\ \left\langle p_\kappa^{\vec{\varrho}^{v}_{\bar{s}}}, \partial_0^{-1} p_\kappa^{\vec{y}^u_s}\right\rangle y^{u,\prime}_{j,s}\otimes \bar{E}\in \bC^{\times^3} ,\ \bar{E} = (1,1,1), \label{e.l.1}\\
\tilde{\lambda}_{\kappa}^{v,j}(\bar{s}) =& \frac{-iq\kappa^2}{2\sqrt{4\pi}}\sum_{u=1}^{\on}\int_0^1 ds\ \left\langle p_\kappa^{\vec{\varrho}^{v}_{\bar{s}}}, \partial_0^{-1} p_\kappa^{\vec{y}^u_s}\right\rangle y^{u,\prime}_{j,s} \in \bC, \label{e.l.2}
\end{align}
for $\bar{s} \in I = [0,1]$.
\end{defn}

\begin{rem}\label{r.x.1}
\begin{enumerate}
  \item We can also identify $\tilde{\lambda}_{\kappa}^{v}(\bar{s}) \otimes \bar{E}$ with $\lambda_\kappa^v(\bar{s})$.
  \item Recall from Section \ref{s.ss}, $\left\langle p_\kappa^{\vec{\varrho}^{v}_{\bar{s}}}, \partial_0^{-1} p_\kappa^{\vec{y}^u_s}\right\rangle$ means $\int_{\bR^4} p_\kappa^{\vec{\varrho}^{v}_{\bar{s}}} \cdot \left[ \partial_0^{-1} p_\kappa^{\vec{y}^u_s} \right]\ d\lambda$ and $\partial_0^{-1} p_\kappa^{\vec{y}^u_s}$ is defined using Equation (\ref{e.d.1}) in Definition \ref{d.lo.1}.
\end{enumerate}

\end{rem}

\begin{defn}\label{d.w.1}
Refer to Definition \ref{d.n.3}. Recall that we have 2 copies of $\mathfrak{su}(2)$ and $\rho_u \equiv (\rho_u^+, \rho_u^-)$, $u=1, \ldots, \on$, from Notation \ref{n.su.1}. Using Notation \ref{n.n.3}, we will write
\begin{align*}
\mathcal{W}^+_\kappa&(q; \ol^u, \uL) = \exp\Bigg\{ \frac{iq}{4}\frac{\kappa^3}{4\pi}\sum_{v=1}^{\underline{n}}\int_{I^2}\ d\hat{s}\  \left\langle  y^{u,\prime}_{j,s}\partial_0^{-1}p_\kappa^{\vec{y}^u_s} \otimes \rho_u^+(\widetilde{E}_s), \left\{ \left[m_\kappa(\lambda_\kappa^v(\bar{s}) )\right]^{-1} (\varrho_{\bar{s}}^{v,\prime} p_\kappa^{\vec{\varrho}^{v}_{\bar{s}}}) \right\}^j\right\rangle \Bigg\}, \\
\mathcal{W}^-_\kappa&(q; \ol^u, \uL) \\
=& \exp\Bigg\{-\frac{iq}{4}\frac{\kappa^3}{4\pi}
\sum_{v=1}^{\underline{n}}\int_{I^2}\ d\hat{s}\ \left\langle y^{u,\prime}_{k,s}\partial_0^{-1} p_\kappa^{\vec{y}^u_s}, \left\{ \left[m_\kappa(\tilde{\lambda}_\kappa^v(\bar{s}) )\right]^{-1}(\varrho_{\bar{s}}^{v,\prime} p_\kappa^{\vec{\varrho}^{v}_{\bar{s}}}) \right\}^k \right\rangle \otimes \sum_{j=1}^3\rho_u^-(\hat{E}_s^{\tau(j)})  \Bigg\}.
\end{align*}
Note that $\mathcal{W}^\pm \in \bigoplus_{n=0}^\infty (\mathfrak{su}(2))_{\bC}^{\otimes^n}$, but  $\mathcal{T}\mathcal{W}^\pm \equiv \mathcal{W}^\pm$. If we multiply the matrices after applying the time ordering operator, then $\mathcal{W}^\pm \in SU(2)$.

We will define \beq Z(\kappa, q; \chi(\oL, \uL)) := \prod_{u=1}^{\on}\left[ \Tr_{\rho_u^+}\  \mathcal{W}^+_\kappa(q; \ol^u, \uL) + \Tr_{\rho_u^-} \mathcal{W}^-_\kappa(q ; \ol^u, \uL) \right]. \label{e.z.1} \eeq See also Remark \ref{r.p.2}.
\end{defn}

\begin{rem}\label{r.p.2}
\begin{enumerate}
  \item Note that $[m_\kappa(\lambda_\kappa^v(\bar{s}))]^{-1}(\varrho_{\bar{s}}^{v,\prime} p_\kappa^{\vec{\varrho}^{v}_{\bar{s}}}) $ is a 9-vector, i.e.
\begin{align*}
[m_\kappa(\lambda_\kappa^v(\bar{s}))]^{-1}(\varrho_{\bar{s}}^{v,\prime} p_\kappa^{\vec{\varrho}^{v}_{\bar{s}}})   \equiv& m_\kappa(\lambda_\kappa^v(\bar{s}))^{-1}\left[  (\varrho_{1,\bar{s}}^{v,\prime}\bar{E}, \varrho_{2,\bar{s}}^{v,\prime}\bar{E}, \varrho_{3,\bar{s}}^{v,\prime}\bar{E})p_\kappa^{\vec{\varrho}^{v}_{\bar{s}}} \right],\ \bar{E} = (1,1,1),
\end{align*}
and its components given explicitly, for $(i,j,k) \in C_3$,
\begin{align*}
\left\{ \left[m_\kappa(\lambda_\kappa^v(\bar{s}) )\right]^{-1} (\varrho_{\bar{s}}^{v,\prime} p_\kappa^{\vec{\varrho}^{v}_{\bar{s}}}) \right\}^i
\equiv& m_\kappa(\lambda_\kappa^{v,j}(\bar{s}))^{-1} \left[ p_\kappa^{\vec{\varrho}^{v}_{\bar{s}}} \bar{E} \right]\varrho_{k,\bar{s}}^{v,\prime} - m_\kappa(\lambda_\kappa^{v,k}(\bar{s}))^{-1}\left[ p_\kappa^{\vec{\varrho}^{v}_{\bar{s}}}\bar{E}\right]\varrho_{j,\bar{s}}^{v,\prime},
\end{align*}
which are defined by Equations (\ref{e.m.2a}) and (\ref{e.m.2b}). Hence we notice that $m_\kappa(\lambda_\kappa^v(\bar{s}) )^{-1}(\varrho_{\bar{s}}^{v,\prime} p_\kappa^{\vec{\varrho}^{v}_{\bar{s}}}) \equiv m_\kappa(0)^{-1}(\varrho_{\bar{s}}^{v,\prime} p_\kappa^{\vec{\varrho}^{v}_{\bar{s}}})$.

\item Similarly, $\left[m_\kappa(\tilde{\lambda}_\kappa^v(\bar{s}) )\right]^{-1}(\varrho_{\bar{s}}^{v,\prime} p_\kappa^{\vec{\varrho}^{v}_{\bar{s}}})$ is a 3-vector and $ \left\{\left[m_\kappa(\tilde{\lambda}_\kappa^v(\bar{s}) )\right]^{-1}(\varrho_{\bar{s}}^{v,\prime} p_\kappa^{\vec{\varrho}^{v}_{\bar{s}}}) \right\}^j $ refers to its $j$-th component and explicitly given, for $(i,j,k) \in C_3$,
\begin{align*}
\left\{ \left[m_\kappa(\tilde{\lambda}_\kappa^v(\bar{s}) )\right]^{-1} (\varrho_{\bar{s}}^{v,\prime} p_\kappa^{\vec{\varrho}^{v}_{\bar{s}}}) \right\}^i
\equiv& m_\kappa(\tilde{\lambda}_\kappa^{v,j}(\bar{s}))^{-1} \left[ p_\kappa^{\vec{\varrho}^{v}_{\bar{s}}} \right]\varrho_{k,\bar{s}}^{v,\prime} - m_\kappa(\lambda_\kappa^{v,k}(\bar{s}))^{-1}\left[ p_\kappa^{\vec{\varrho}^{v}_{\bar{s}}}\right]\varrho_{j,\bar{s}}^{v,\prime},
\end{align*}
which are defined by Equations (\ref{e.m.3}) and (\ref{e.m.3b}).

\item For a given 3-vector $y = (y_1, y_2, y_3) \in \bC^{\times^3}$, the term $\langle \widetilde{E}, y \rangle$ means $y_j\otimes \hat{E}^{0j}$.
Thus if we write \beq \left\{ \left[m_\kappa(\lambda_\kappa^v(\bar{s}) )\right]^{-1} (\varrho_{\bar{s}}^{v,\prime} p_\kappa^{\vec{\varrho}^{v}_{\bar{s}}}) \right\}^j \equiv (x_1^j, x_2^j, x_3^j) \in (C^\infty(\bR^4))^{\times^3}, \nonumber \eeq then
\begin{align*}
\left\langle  y^{u,\prime}_{j,s}\partial_0^{-1}p_\kappa^{\vec{y}^u_s} \otimes \rho_u^+(\widetilde{E}_s), \left\{ \left[m_\kappa(\lambda_\kappa^v(\bar{s}) )\right]^{-1} (\varrho_{\bar{s}}^{v,\prime} p_\kappa^{\vec{\varrho}^{v}_{\bar{s}}}) \right\}^j\right\rangle
&\equiv y^{u,\prime}_{j,s} \langle \partial_0^{-1}p_\kappa^{\vec{y}^u_s}, x_i^j\rangle \otimes \rho_u^+(\hat{E}^{0i}_s).
\end{align*}
Using Remarks \ref{r.lo.2} and \ref{r.x.1}, we leave to the reader to show that it is equal to \beq -\epsilon^{ijk}\left\langle p_\kappa^{\vec{y}^u_s}, p_\kappa^{\vec{\varrho}^{v}_{\bar{s}}} \right\rangle_k y^{u,\prime}_{i,s}\varrho_{j,\bar{s}}^{v,\prime} \otimes \sum_{i=1}^3\rho_u^+(\hat{E}^{0i}_s). \nonumber \eeq Refer to Expression \ref{e.h.1}.
\end{enumerate}

\end{rem}

\begin{notation}\label{n.a.1}
Refer to Notation \ref{n.n.5}. Write
\begin{align*}
a_\kappa =& \bk^{3}\int_{I^2}\Bigg\{  -\sum_{i=1}^3\left[\left\langle q \sum_{u=1}^{\on} \kappa\int_0^1 y_{1,s}^{u,\prime}\partial_0^{-1}p_\kappa^{\vec{y}_s^u} ds, p_\kappa^{\vec{\sigma}(\hat{t})} \right\rangle \otimes \rho_u^+(\hat{E}_s^{0i})\right]^2
\Bigg\}^{1/2}\ J_{23}(\hat{t}) d\hat{t}, \\
b_\kappa =& \bk^{3}\int_{I^2}\Bigg\{  \sum_{i=1}^3\left[\left\langle q \sum_{u=1}^{\on} \kappa\int_0^1 y_{1,s}^{u,\prime}\partial_0^{-1}p_\kappa^{\vec{y}_s^u} ds, p_\kappa^{\vec{\sigma}(\hat{t})} \right\rangle \otimes \rho_u^-(\hat{E}_s^{\tau(i)}) \right]^2
\Bigg\}^{1/2}\  J_{23}(\hat{t}) d\hat{t}.
\end{align*}
And \beq \left\langle \int_0^1 y_{1,s}^{u,\prime}\partial_0^{-1}p_\kappa^{\vec{y}_s^u} ds, p_\kappa^{\vec{\sigma}(\hat{t})} \right\rangle
\equiv \int_0^1ds\ y_{1,s}^{u,\prime}\left\langle  \partial_0^{-1}p_\kappa^{\vec{y}_s^u} , p_\kappa^{\vec{\sigma}(\hat{t})} \right\rangle. \nonumber \eeq See also Section \ref{s.ss}.
\end{notation}

\begin{lem}\label{l.x.2}
Recall $\mathcal{W}_\kappa^\pm$ were defined in Definition \ref{d.w.1} and also refer to Notation \ref{n.a.1}. Apply Step \ref{d.cs.r.1e} in Definition \ref{d.cs.r.1}, the path integral in Expression \ref{e.a.5} is hence computed as
\begin{align}
\widetilde{\Tr}\left(
  \begin{array}{cc}
   a_\kappa \bigotimes_{u=1}^{\on} \mathcal{W}_\kappa^+(q, ; \ol^u, \uL)&\ 0 \\
    0 &\ b_\kappa\bigotimes_{u=1}^{\on} \mathcal{W}_\kappa^-(q; \ol^u, \uL)\\
  \end{array}
\right)
. \label{e.sa.1}
\end{align}
\end{lem}

\begin{proof}
From Equation (\ref{e.w.1}) and according to Step \ref{d.cs.r.1e} in Definition \ref{d.cs.r.1}, we replace $B^j$ and $\tilde{B}^i$ inside the function $\tilde{V}_\kappa(\{B^j_\mu\})$ with
\begin{align}
B^j \longmapsto& -iq\kappa\bk\sum_{u=1}^{\on}\int_0^1 ds\ y^{u,\prime}_{j, s}\partial_0^{-1}p_\kappa^{\vec{y}^u_s} \otimes (\rho_u^+(\hat{E}_s^{01}),  \rho_u^+(\hat{E}_s^{02}), \rho_u^+(\hat{E}_s^{03})), \label{e.b.1}\\
\tilde{B}^i \longmapsto& iq\kappa\bk\sum_{u=1}^{\on}\int_0^1 ds\ y^{u,\prime}_{s}\partial_0^{-1}p_\kappa^{\vec{y}^u_s} \otimes \rho_u^-(\hat{E}_s^{\tau(i)}). \label{e.b.2}
\end{align}

Expression \ref{e.b.1} will give us \beq
m_\kappa(\langle B, p_\kappa^{\vec{\varrho}^v_{\bar{s}}}\rangle)^{-1}\longmapsto m_\kappa(\lambda_\kappa^v(\bar{s}))^{-1} \nonumber \eeq
and it also means that \beq
B_i \longmapsto -iq\kappa\bk\sum_{u=1}^{\on}\int_0^1 ds\ y^{u,\prime}_{s}\partial_0^{-1}p_\kappa^{\vec{y}^u_s} \otimes \rho_u^+(\hat{E}_s^{0i}). \nonumber \eeq

Hence it means we replace
\begin{align*}
m_\kappa(\langle B_i, p_\kappa^{\vec{\varrho}^v_{\bar{s}}}\rangle)^{-1}\longmapsto m_\kappa(\tilde{\lambda}_\kappa^v(\bar{s}))^{-1}.
\end{align*}

Note that $\lambda_\kappa^v(\bar{s})$ and $\tilde{\lambda}_\kappa^v(\bar{s})$ are given by Equations (\ref{e.l.1}) and (\ref{e.l.2}) respectively. Substitute all these inside $\tilde{V}_\kappa$ as defined in Definition \ref{d.he.1} and we will obtain \beq
\left(
  \begin{array}{cc}
   \bigotimes_{u=1}^{\on} \mathcal{W}_\kappa^+(q, ; \ol^u, \uL) &\ 0 \\
    0 &\ \bigotimes_{u=1}^{\on}\mathcal{W}_\kappa^-(q, ; \ol^u, \uL)  \\
  \end{array}
\right). \nonumber \eeq

Making the substitution given by Equation (\ref{e.b.1}), we see that
\begin{align*}
\left|\left\langle B^1, p_\kappa^{\vec{\sigma}(\hat{t})} \right\rangle \right|^2 &\longmapsto -\sum_{i=1}^3\left| \left\langle q\kappa\bk\sum_{u=1}^{\on}\int_0^1 ds\ y^{u,\prime}_{1, s}\partial_0^{-1}p_\kappa^{\vec{y}^u_s}, p_\kappa^{\vec{\sigma}(\hat{t})}\right\rangle \otimes \rho_{u}^+(\hat{E}^{0i})\right|^2 \\
&= -q^2\kappa^2\bk^2\sum_{u,\bar{u}=1}^{\on}\int_{I^2} d\hat{s}
\left\langle y^{u,\prime}_{1, s}\partial_0^{-1}p_\kappa^{\vec{y}^u_s}, p_\kappa^{\vec{\sigma}(\hat{t})}\right\rangle\
\left\langle y^{\bar{u},\prime}_{1, \bar{s}}\partial_0^{-1}p_\kappa^{\vec{y}^{\bar{u}}_{\bar{s}}}, p_\kappa^{\vec{\sigma}(\hat{t})}\right\rangle\
\otimes  A(u,s; \bar{u}, \bar{s}),
\end{align*}
where $A(u,s; \bar{u}, \bar{s}) := \sum_{i=1}^3\rho_{u}^+(\hat{E}_{s}^{0i})\otimes \rho_{\bar{u}}^+(\hat{E}_{\bar{s}}^{0i})$.

Making the substitution given by Equation (\ref{e.b.2}), we see that
\begin{align*}
-\left|\left\langle \tilde{B}^j, p_\kappa^{\vec{\sigma}(\hat{t})} \right\rangle \cdot e_1 \right|^2 &\longmapsto \left| \left\langle q\kappa\bk\sum_{u=1}^{\on}\int_0^1 ds\ y^{u,\prime}_{1,s}\partial_0^{-1}p_\kappa^{\vec{y}^u_s}, p_\kappa^{\vec{\sigma}(\hat{t})} \right\rangle \otimes \rho_{u}^-(\hat{E}_s^{\tau(j)})\right|^2 \\
&= q^2\kappa^2\bk^2 \sum_{u,\bar{u}=1}^{\on}\int_{I^2} d\hat{s}
\left\langle y^{u,\prime}_{1, s}\partial_0^{-1}p_\kappa^{\vec{y}^u_s}, p_\kappa^{\vec{\sigma}(\hat{t})}\right\rangle\
\left\langle y^{\bar{u},\prime}_{1, \bar{s}}\partial_0^{-1}p_\kappa^{\vec{y}^{\bar{u}}_{\bar{s}}}, p_\kappa^{\vec{\sigma}(\hat{t})}\right\rangle\
\otimes  B^j(u,s; \bar{u}, \bar{s}),
\end{align*}
where $B^j(u,s; \bar{u}, \bar{s}) := \rho_{u}^+(\hat{E}_{s}^{\tau(j)})\otimes \rho_{\bar{u}}^+(\hat{E}_{\bar{s}}^{\tau(j)})$.


Suppose our representation is given by $\rho_u \equiv (\rho_u^+, 0)$.
Apply Step \ref{d.cs.r.1e} in Definition \ref{d.cs.1}, substitute inside Expression \ref{e.a.5} and with Notation \ref{n.a.1}, we will obtain
\begin{align}
\widetilde{\Tr}&\ a_\kappa\bigotimes_{u=1}^{n} \exp\Bigg\{ \frac{iq}{4}\frac{\kappa^3}{4\pi}\sum_{v=1}^{\underline{n}}\int_{I^2}\ d\hat{s}\  \left\langle  y^{u,\prime}_{j,s}\partial_0^{-1}p_\kappa^{\vec{y}^u_s} \otimes \rho_u^+(\widetilde{E}_s), \left\{ \left[m_\kappa(\lambda_\kappa^v(\bar{s}) )\right]^{-1} (\varrho_{\bar{s}}^{v,\prime} p_\kappa^{\vec{\varrho}^{v}_{\bar{s}}}) \right\}^j\right\rangle  \Bigg\} \nonumber \\
\equiv& \widetilde{\Tr}\ \left[a_\kappa \bigotimes_{u=1}^{n}\mathcal{W}_\kappa^+(q, ; \ol^u, \uL)\right], \label{e.a.6a}
\end{align}
which follows from Definition \ref{d.w.1}.

Similarly, suppose our representation is given by $\rho_u \equiv (0, \rho_u^-)$.
Apply Step \ref{d.cs.r.1e} in Definition \ref{d.cs.1}, substitute inside Expression \ref{e.a.5} and with Notation \ref{n.a.1}, we will obtain
\begin{align}
\widetilde{\Tr}&\ b_\kappa\bigotimes_{u=1}^{n} \exp\Bigg\{ -\frac{iq}{4}\frac{\kappa^3}{4\pi}\sum_{v=1}^{\underline{n}}\int_{I^2}\ d\hat{s}\ \left\langle y^{u,\prime}_{k,s}\partial_0^{-1} p_\kappa^{\vec{y}^u_s}, \left\{ \left[m_\kappa(\tilde{\lambda}_\kappa^v(\bar{s}) )\right]^{-1}(\varrho_{\bar{s}}^{v,\prime} p_\kappa^{\vec{\varrho}^{v}_{\bar{s}}}) \right\}^k \right\rangle  \otimes \sum_{j=1}^3\rho_u^-(\hat{E}_s^{\tau(j)}) \Bigg\} \nonumber \\
\equiv& \widetilde{\Tr}\ \left[ b_\kappa \bigotimes_{u=1}^{n} \mathcal{W}_\kappa^-(q, ; \ol^u, \uL)\right], \label{e.a.6b}
\end{align}
which follows from Definition \ref{d.w.1}.

In the general case $\rho_u \equiv (\rho_u^+, \rho_u^-)$, Expressions \ref{e.a.6a} and \ref{e.a.6b} will give us our desired result.


\end{proof}

\begin{rem}
\begin{enumerate}
  \item When $S = \emptyset$, then Expression \ref{e.sa.1} reduces to $Z(\kappa, q; \chi(\oL, \uL) )$ as defined in Equation (\ref{e.z.1}).
  \item The term \beq Z(q; \chi(\oL, \uL) ):= \lim_{\kappa \rightarrow \infty}Z(\kappa, q; \chi(\oL, \uL) )\nonumber \eeq will be referred to as the Wilson Loop observable for a colored hyperlink $\chi(\oL, \uL)$.
\end{enumerate}
\end{rem}

Unfortunately, $a_\kappa$ and $b_\kappa$ are not defined, if the components of the hyperlink are colored differently. The reason is that we do not know how to take the square root of $\sum_{i=1}^3 \rho_u^+(\hat{E}^{0i})\otimes \rho_{\bar{u}}^+(\hat{E}^{0i})$ and $\sum_{i=1}^3 \rho_u^+(\hat{E}^{\tau(i)})\otimes \rho_{\bar{u}}^+(\hat{E}^{\tau(i)})$, if $\rho_u^\pm \neq \rho_{\bar{u}}^\pm$.

What if all the representations are the same? Unfortunately, we will still have problems defining $a_\kappa$ and $b_\kappa$. The reason is the square root function is not analytic at 0, so we do not know how to apply the time ordering operator to the square root.

In a sequel \cite{EH-Lim03} to this article, we will show that the area path integral can be computed from the intersection points between the link $\pi_0(\oL)$ and the surface $S$. At any such intersection point, also termed as piercing, we see that it involves only one component $\ol^u$ inside $\oL$ and only one point in $\ol^u$, so the sum of the tensor products inside $a_\kappa$ and $b_\kappa$ reduce down to
\begin{align*}
\sum_{i=1}^3 \rho_u^+(E_s^{0i})\rho_u^+(E_s^{0i}) \equiv -\xi_{\rho_u^+}I_{\rho_u^+},\ {\rm and} \
\sum_{i=1}^3 \rho_u^-(E_s^{\tau(i)})\rho_u^-(E_s^{\tau(i)}) \equiv -\xi_{\rho_u^-}I_{\rho_u^-},
\end{align*}
$I_{\rho_u^\pm}$ are respectively identity matrices. It is now clear how to take the square root.

If the representations are all the same, i.e. $\rho_u^\pm \equiv \rho^\pm$, then we replace both $a_\kappa$ and $b_\kappa$ with
\begin{align*}
\bar{a}_\kappa =& \bk^{3}\int_{I^2}\Bigg\{  \left[\left\langle q \sum_{u=1}^{\on} \kappa\int_0^1 y_{1,s}^{u,\prime}\partial_0^{-1}p_\kappa^{\vec{y}_s^u} ds, p_\kappa^{\vec{\sigma}(\hat{t})} \right\rangle \right]^2 \otimes\xi_{\rho_u^+}I_{\rho_u^+}
\Bigg\}^{1/2}\ J_{23}(\hat{t}) d\hat{t}, \\
\bar{b}_\kappa =& \bk^{3}\int_{I^2}\Bigg\{  -\left[\left\langle q \sum_{u=1}^{\on} \kappa\int_0^1 y_{1,s}^{u,\prime}\partial_0^{-1}p_\kappa^{\vec{y}_s^u} ds, p_\kappa^{\vec{\sigma}(\hat{t})} \right\rangle \right]^2\otimes\xi_{\rho_u^-}I_{\rho_u^-}
\Bigg\}^{1/2}\  J_{23}(\hat{t}) d\hat{t}
\end{align*}
respectively, in Expression \ref{e.sa.1}.


Hence Expression \ref{e.sa.1}, upon further simplification, gives us
\begin{align}
\prod_{\bar{u}=1}^{\on}\Bigg\{& \Bigg[  \bk^{3}\int_{I^2}\Bigg|\left\langle q \sum_{u =1}^{\on} \kappa\int_0^1 y_{1,s}^{u,\prime} \partial_0^{-1}p_\kappa^{\vec{y}_s^u} ds, p_\kappa^{\vec{\sigma}(\hat{t})} \right\rangle\ \sqrt{\xi_{\rho_u^+}}
\Bigg|\ J_{23}(\hat{t}) d\hat{t} \Bigg]^{1/\on} \Tr_{\rho_{\bar{u}}^+}\ \mathcal{W}_\kappa^+(q; \ol^{\bar{u}}, \uL) \nonumber\\
+& \Bigg[
i\bk^{3}\int_{I^2}\Bigg|\left\langle q \sum_{u =1}^{\on} \kappa\int_0^1 y_{1,s}^{u,\prime} \partial_0^{-1}p_\kappa^{\vec{y}_s^u} ds, p_\kappa^{\vec{\sigma}(\hat{t})} \right\rangle\ \sqrt{\xi_{\rho_u^-}}
\Bigg|\ J_{23}(\hat{t}) d\hat{t} \Bigg]^{1/\on} \Tr_{\rho_{\bar{u}}^-}\ \mathcal{W}_\kappa^-(q; \ol^{\bar{u}}, \uL) \Bigg\}
. \label{e.a.1}
\end{align}

We will now define the path integral in Expression \ref{e.a.3} as the limit of Expression \ref{e.a.1} as $\kappa$ goes to infinity. We can further simplify this expression using the following lemma.
\begin{lem}\label{l.k.1}
We have $\kappa\lambda_{\kappa}^v(\bar{s}) \rightarrow 0$ and $\kappa\tilde{\lambda}_{\kappa}^v(\bar{s}) \rightarrow 0$ as $\kappa \rightarrow \infty$.
\end{lem}

\begin{proof}
Note that \beq \kappa^3\left\langle p_\kappa^{\vec{\varrho}^{v}_{\bar{s}}}, \partial_0^{-1} (p_\kappa^{\vec{y}^u_s})\right\rangle = \kappa^2\left\langle p_\kappa^{\varrho^v_{\bar{s}}}, p_\kappa^{y^u_s} \right\rangle \cdot \kappa\left\langle q_\kappa^{\varrho^v_{0,\bar{s}}}, \partial_0^{-1}q_\kappa^{y^u_{0,s}} \right\rangle. \nonumber \eeq The proof now follows directly from Lemma \ref{l.l.5}, the details to be left to the reader.
\end{proof}

\begin{cor}\label{c.x.1}
Refer to Definition \ref{d.lo.1}. Define $\hat{\mathcal{W}}_\kappa^\pm(q; \ol^u, \uL)$ as
\begin{align}
\hat{\mathcal{W}}_\kappa^\pm(q; \ol^u, \uL) := \exp\left[ \mp\frac{iq}{4}\frac{\kappa^3}{4\pi}\sum_{v=1}^{\underline{n}}\int_{I^2}\ d\hat{s}\ \epsilon^{ijk}\left\langle  p_\kappa^{\vec{y}^u_s}, p_\kappa^{\vec{\varrho}^v_{\bar{s}}}\right\rangle_k  y^{u,\prime}_{i,s}\varrho^{v,\prime}_{j,\bar{s}} \otimes \mathcal{F}^\pm\right], \label{e.h.5}
\end{align}
whereby $\mathcal{F}^\pm$ was defined in Notation \ref{n.su.1}. We thus have \beq \lim_{\kappa \rightarrow \infty}\Tr_{\rho_u^\pm}\ \mathcal{W}_\kappa^\pm(q; \ol^u, \uL) = \lim_{\kappa \rightarrow \infty}\Tr_{\rho_u^\pm}\ \hat{\mathcal{W}}_\kappa^\pm(q; \ol^u, \uL). \nonumber \eeq
\end{cor}

\begin{proof}
From Lemma \ref{l.k.1}, together with Remarks \ref{r.lo.2} and \ref{r.m.1}, we see that to compute the limit as $\kappa$ goes to infinity, for the exponent in $\mathcal{W}_\kappa^\pm(q; \ol^u, \uL)$, it suffices to compute the limit as $\kappa$ goes to infinity for
\begin{align}
\mp\frac{iq}{4}\frac{\kappa^3}{4\pi}\sum_{v=1}^{\underline{n}}\int_{I^2}\ d\hat{s}\ \Bigg\{ &\left\langle  y^{u,\prime}_{1,s}\partial_0^{-1}p_\kappa^{\vec{y}^u_s}, \kappa\left[\varrho^{v,\prime}_{3,\bar{s}}\partial_2^{-1} - \varrho^{v,\prime}_{2,\bar{s}}\partial_3^{-1}\right]p_\kappa^{\vec{\varrho}^{v}_{\bar{s}}}\right\rangle \nonumber\\
+& \left\langle y^{u,\prime}_{2,s} \partial_0^{-1}p_\kappa^{\vec{y}^u_s}, \kappa\left[\varrho^{v,\prime}_{1,\bar{s}}\partial_3^{-1} - \varrho^{v,\prime}_{3,\bar{s}}\partial_1^{-1}\right]p_\kappa^{\vec{\varrho}^{v}_{\bar{s}}}\right\rangle
\nonumber\\
+&\left\langle y^{u,\prime}_{3,s} \partial_0^{-1}p_\kappa^{\vec{y}^u_s}, \kappa\left[\varrho^{v,\prime}_{2,\bar{s}}\partial_1^{-1} - \varrho^{v,\prime}_{1,\bar{s}}\partial_2^{-1}\right]p_\kappa^{\vec{\varrho}^{v}_{\bar{s}}}\right\rangle \label{e.h.1}
\Bigg\} \otimes \mathcal{F}^\pm.
\end{align}

Using Notations \ref{n.n.3} and \ref{n.k.1} applied to Expression \ref{e.h.1}, after some simple manipulation, Expression \ref{e.h.1} can be written compactly as \beq \mp\frac{iq}{4}\frac{\kappa^3}{4\pi}\sum_{v=1}^{\underline{n}}\int_{I^2}\ d\hat{s}\ \epsilon^{ijk}\left\langle  p_\kappa^{\vec{y}^u_s}, p_\kappa^{\vec{\varrho}^v_{\bar{s}}}\right\rangle_k  y^{u,\prime}_{i,s}\varrho^{v,\prime}_{j,\bar{s}} \otimes \mathcal{F}^\pm. \nonumber \eeq This completes the proof.
\end{proof}

With this corollary, we can define the area path integral by replacing $\mathcal{W}_\kappa^\pm$ with $\hat{\mathcal{W}}_\kappa^\pm$ in Expression \ref{e.a.1}, and taking the limit of this new expression as $\kappa$ goes to infinity. This is for the case when the representations are all the same.

\begin{notation}\label{n.n.4}
Suppose we have a list of irreducible representations of $\mathfrak{su}(2)\times \mathfrak{su}(2)$, $\{\rho_1, \ldots, \rho_{\on}\}$ and there are $\bar{m}$ distinct representations, labeled as $\{\gamma_1, \ldots, \gamma_{\bar{m}} \}$, arranged in any order. For representation $\gamma_u$, let $\Gamma_u$ denote the set of integers in $\{1,2, \ldots, \on\}$ such that $\rho_v = \gamma_u$, $v \in \Gamma_u$.
\end{notation}

If the representations are not the same, we can define the area path integral in the following manner.

\begin{defn}(Area Path Integral)\label{d.api}\\
Write the surface $S$ into a disjoint union of smaller surfaces \beq S_1,\ S_2,\ \ldots,\ S_{\bar{m}}, \nonumber \eeq $\bar{m}$ as defined in Notation \ref{n.n.4}. In other words, $S_v$ will be the (possibly disconnected) surface whereby those hyperlinks colored with the same representation $\gamma_v$ pierce it.
Let $I^2 = \bigcup_{v=1}^{\bar{m}}I_v^2$ such that $\sigma: I_v^2 \rightarrow S_v$ is a parametrization of $S_v$.

Hence we can write the path integral in Expression \ref{e.a.5} as
\begin{align*}
\frac{1}{\hat{Z}_{EH}}&\widetilde{\Tr}\int \tilde{V}_\kappa(\{B^j_\mu\}) \tilde{W}_\kappa(\{A^k_{\alpha\beta}\})\tilde{A}_{\kappa,S}(B, \{\tilde{B}^i\}) \cdot  e^{i\int_{\bR^4} A_0\cdot  B   -  A \cdot \tilde{B}}  D\Lambda \\
=&
\sum_{v=1}^{\bar{m}}\frac{1}{\hat{Z}_{EH}}\widetilde{\Tr}\int \tilde{V}_\kappa(\{B^j_\mu\}) \tilde{W}_\kappa(\{A^k_{\alpha\beta}\})\tilde{A}_{\kappa,S}(B, \{\tilde{B}^i\}) \cdot  e^{i\int_{\bR^4} A_0\cdot  B   -  A \cdot \tilde{B}}  D\Lambda.
\end{align*}

The path integral in Expression \ref{e.a.3} is now define as the limit as $\kappa$ goes to infinity, of the expression
\begin{align}
\prod_{\bar{u}=1}^{\on}\Bigg\{& \Bigg[ \sum_{v=1}^{\bar{m}} \bk^{3}\int_{I_v^2}\Bigg|\left\langle q \sum_{u \in \Gamma_v} \kappa\int_0^1 y_{1,s}^{u,\prime} \partial_0^{-1}p_\kappa^{\vec{y}_s^u} ds, p_\kappa^{\vec{\sigma}(\hat{t})} \right\rangle\ \sqrt{\xi_{\rho_u^+}}
\Bigg|\ J_{23}(\hat{t}) d\hat{t} \Bigg]^{1/\on} \Tr_{\rho_{\bar{u}}^+}\ \hat{\mathcal{W}}_\kappa^+(q; \ol^{\bar{u}}, \uL) \nonumber\\
+& \Bigg[
i\sum_{v=1}^{\bar{m}}\bk^{3}\int_{I_v^2}\Bigg|\left\langle q \sum_{u \in \Gamma_v} \kappa\int_0^1 y_{1,s}^{u,\prime} \partial_0^{-1}p_\kappa^{\vec{y}_s^u} ds, p_\kappa^{\vec{\sigma}(\hat{t})} \right\rangle\ \sqrt{\xi_{\rho_u^-}}
\Bigg|\ J_{23}(\hat{t}) d\hat{t} \Bigg]^{1/\on} \Tr_{\rho_{\bar{u}}^-}\ \hat{\mathcal{W}}_\kappa^-(q; \ol^{\bar{u}}, \uL) \Bigg\}
. \nonumber
\end{align}
\end{defn}

\begin{rem}
The above expression is dependent on how we partition the surface $S$ into $\bigcup_{v=1}^{\bar{m}}S_v$. But its limit as $\kappa$ goes to infinity will be shown to be independent of this partition in a sequel to this article.
\end{rem}

\section{Volume Path Integral}\label{s.vo}

Fix a closed and bounded 3-manifold $R \subset \bR^3$, referred as a compact region from now on, possibly disconnected with a finite number of components. We identify it as $\{0\} \times R \subset \{0\} \times \bR^3$ inside $\bR \times \bR^3$. Furthermore, $\oL$ is disjoint from $R$.

\begin{notation}\label{n.r.1}
Refer to Notation \ref{n.n.3}. Let $\rho: I^3 \rightarrow \bR^3$ be any parametrization of $R$. Let $|J_\rho|(r)$ denote the determinant of the Jacobian of $\rho$, $r = (r_1, r_2, r_3)$. And write $dr = dr_1 dr_2 dr_3$. We will also write $\vec{\rho}(r) \equiv \vec{\rho}_r \equiv (0, \rho(r)) \in \bR^4$.
\end{notation}

Using the dynamical variables $\{B_\mu^i\}$ and the Minkowski metric $\eta^{ab}$, we see that the metric $g^{ab} \equiv B^a_\mu\eta^{\mu\gamma}B^b_\gamma$ and the volume $V_R$ is given by \beq V_R(\{B_\mu^i\}) = \int_R \sqrt{\epsilon_{ijk}\epsilon_{\bar{i}\bar{j}\bar{k}}g^{i\bar{i}}g^{j\bar{j}}g^{k\bar{k}}}. \nonumber \eeq

Refer to Notation \ref{n.n.6}. It is not difficult to see that the volume is indeed given by
\begin{align*}
V_R(\{B_\mu^i\})
= \int_R \Bigg\{ \left|\frac{B}{3} \cdot B \vec{\times} B  \right|^2 &- \sum_{(i,j) \in \Upsilon}\left|B_i \times B_0 \cdot B_j \right|^2 + \sum_{i=1}^3 \left|\frac{1}{3}B_i \times B_0 \cdot B_0 \right|^2\Bigg\}^{1/2}\ dV.
\end{align*}

\begin{rem}
Note that $B_i \times B_0 \cdot B_0 \equiv 0$. However, we include this term inside this formula, as when we do the substitution $B_i \times B_0 \mapsto \tilde{B}^i$, this term will give us a non-trivial contribution.
\end{rem}

Refer to Notation \ref{n.x.1}. We want to define a volume path integral
\begin{align}
\frac{1}{Z_{EH}}\int V( \{\ul^v\}_{v=1}^{\un})(\{B^i_\mu\})&W(q; \{\ol^u, \rho_u\}_{u=1}^{\on})(\{A^k_{\alpha\beta}\}) \nonumber\\
& \times   V_R(\{B_\mu^i\}) \exp\left[i\int_{\bR^4} \partial_0 A_0\cdot B \vec{\times} B   - \partial_0 A \cdot \tilde{B}\right] D\Lambda. \label{e.v.5}
\end{align}

\begin{rem}\label{r.z.1}
\begin{enumerate}
  \item When $R$ is the empty set, we define $V_\emptyset \equiv 1$, so we write Expression \ref{e.v.5} as $Z(q; \chi(\oL, \uL))$, which was termed as the Wilson Loop observable of the colored hyperlink $\chi(\oL, \uL))$ in Remark \ref{r.ax.1}.
  \item We will write Expression \ref{e.v.5} as $\hat{V}_R[Z(q; \chi(\oL, \uL))]$.
\end{enumerate}
\end{rem}

We will now make use of Definition \ref{d.cs.r.1} to make sense of the path integral in Expression \ref{e.v.5}.

\begin{lem}
Recall we defined $\tilde{V}(\{B^j_\mu\})$ and $\tilde{W}(\{A^k_{\alpha\beta}\})$ in Equations (\ref{e.w.2}) and (\ref{e.w.3}) respectively. Refer to the parametrizations $\vec{y}^u$ and $\vec{\varrho}^v$ defined in Notation \ref{n.n.3}. After doing a change of variables given in Notation \ref{n.c.1}, the path integral in Expression \ref{e.v.5} is defined as \beq
\frac{1}{\hat{Z}_{EH}}\int \tilde{V}(\{B^j_\mu\}) \tilde{W}(\{A^k_{\alpha\beta}\})\tilde{V}_R(\{B_\mu^i\})  e^{i\int_{\bR^4} A_0\cdot  B   -  A \cdot \tilde{B}}  D\Lambda \label{e.v.6} \eeq
with \beq
\hat{Z}_{EH} = \int \exp\left[i\int_{\bR^4} A_0\cdot  B   -  A \cdot \tilde{B}\right]
D\Lambda, \nonumber \eeq after applying Steps \ref{d.cs.r.1a} and \ref{d.cs.r.1b} in Definition \ref{d.cs.r.1}.

Here,
\begin{align*}
\tilde{V}_R(\{B_\mu^i\})
&= \int_{I^3} \Bigg\{ \left|\frac{\langle [m(\langle B, \delta^{\vec{\rho}(r)} \rangle)]^{-1}B, \delta^{\vec{\rho}(r)} \rangle}{3} \cdot \langle B, \delta^{\vec{\rho}(r)} \rangle  \right|^2 - \sum_{(i,j) \in \Upsilon} \left|\langle \tilde{B}^i, \delta^{\vec{\rho}(r)} \rangle  \cdot \langle \bar{B}_j , \delta^{\vec{\rho}(r)} \rangle \right|^2 \\
& \hspace{2cm}+ \sum_{i=1}^3 \left|\frac{1}{3}\langle \tilde{B}^i, \delta^{\vec{\rho}(r)} \rangle \cdot  \langle \tilde{B}^i, [m(\langle B_i, \delta^{\vec{\rho}(r)} \rangle)]^{-1}\delta^{\vec{\rho}(r)} \rangle \right|^2
\Bigg\}^{1/2} |J_\rho|(r)\ dr.
\end{align*}
Note that we can also write \beq \left\langle\bar{B}_j, \delta^{\vec{\rho}(r)} \right\rangle = \left\langle [m(\langle B, \delta^{\vec{\rho}(r)}\rangle)^{-1}B]_j, \delta^{\vec{\rho}(r)} \right\rangle. \nonumber \eeq
\end{lem}

\begin{proof}
We have shown in Lemma \ref{l.x.1} how to obtain $\tilde{V}$ and $\tilde{W}$ respectively. We will only show how to make the substitution inside the volume integrand $V_R$, which we will now give the details.

Apply Step \ref{d.cs.r.1b} in Definition \ref{d.cs.r.1} and using the notations and substitutions as discussed in Lemma \ref{l.x.1}, make the following change of variables,
\begin{align*}
B \vec{\times} B \longmapsto B,\ \
B \longmapsto \bar{B} \equiv m(B)^{-1}B, \ \
B_i \times B_0 \longmapsto& \tilde{B}^i,\ \
B_j \longmapsto \bar{B}_j,
\end{align*}
and \beq B_0 \longmapsto  m(B_i)^{-1}\tilde{B}^i, \ \ B_i \times B_0 \cdot B_0 \longmapsto \tilde{B}^i m(B_i)^{-1}\tilde{B}^i. \nonumber \eeq

Therefore,the volume integrand becomes
\begin{align*}
\tilde{V}_R(\{B_\mu^i\})
= \int_R \Bigg\{ \left|\frac{\bar{B}}{3} \cdot B  \right|^2 &- \sum_{(i,j) \in \Upsilon} \left|\tilde{B}^i  \cdot \bar{B}_j \right|^2  + \sum_{i=1}^3 \left| \frac{1}{3}\tilde{B}^i \cdot m(B_i)^{-1}\tilde{B}^i \right|^2
\Bigg\}^{1/2}\ dV.
\end{align*}

In terms of the parametrization $\rho$,
\begin{align*}
\tilde{V}_R(\{B_\mu^i\})
= \int_{I^3} &\Bigg\{ \left|\frac{\bar{B}}{3} \cdot B  \right|^2 - \sum_{(i,j) \in \Upsilon} \left|\tilde{B}^i  \cdot \bar{B}_j \right|^2
+ \sum_{i=1}^3 \left|\frac{1}{3}\tilde{B}^i \cdot m(B_i)^{-1}\tilde{B}^i \right|^2
\Bigg\}^{1/2}(\vec{\rho}(r)) |J_\rho|(r)\ dr.
\end{align*}
Note that $\bar{B}(\vec{x}) \equiv m(B(\vec{x}))^{-1} B(\vec{x})$ and $[m(B_i)^{-1}\tilde{B}^i](\vec{x}) \equiv m(B_i(\vec{x}))^{-1}\tilde{B}^i(\vec{x})$.

Apply Step \ref{d.cs.r.1a} in Definition \ref{d.cs.r.1} and write
\begin{align*}
\bar{B}(\vec{\rho}(r)) =& \langle [m(\langle B, \delta^{\vec{\rho}(r)} \rangle)]^{-1}B, \delta^{\vec{\rho}(r)} \rangle, \\
B(\vec{\rho}(r)) =& \langle B,  \delta^{\vec{\rho}(r)} \rangle,\ \tilde{B}^i(\vec{\rho}(r)) = \langle \tilde{B}^i, \delta^{\vec{\rho}(r)} \rangle,\ B_i(\vec{\rho}(r))= \langle B_i, \delta^{\vec{\rho}(r)}\rangle, \\
[m(B_i)^{-1}\tilde{B}^i](\vec{\rho}(r)) \equiv& m(B_i(\vec{\rho}(r)))^{-1}\tilde{B}^i(\vec{\rho}(r))
= \langle m(B_i(\vec{\rho}(r)))^{-1}\tilde{B}^i, \delta^{\vec{\rho}(r)}\rangle \\
=& - \langle \tilde{B}^i, m(B_i(\vec{\rho}(r)))^{-1}\delta^{\vec{\rho}(r)}\rangle
=- \langle \tilde{B}^i, [m(\langle B_i, \delta^{\vec{\rho}(r)} \rangle)]^{-1}\delta^{\vec{\rho}(r)}\rangle,
\end{align*}
because $m(B)^{-1}$ and $m(B_i)^{-1}$ are skew-symmetric.
Substitute into the above volume integrand and we will obtain our result.
\end{proof}

As stated in Section \ref{s.ao}, both $m(B(\vec{x}))^{-1}$, $m(B_i(\vec{x}))^{-1}$ cannot be defined. Recall in Section \ref{s.ao}, we approximate the Dirac-delta function with a Gaussian function $p_\kappa$, $m(B)$ with $m_\kappa(B)$ and $m(B_i)$ with $m_\kappa(B_i)$. And we need to add in factors of $\bk$. See Definition \ref{d.he.1}.

Because $R$ is 3-dimensional submanifold, we need to include a factor of $\bk^3$ to the volume integral. Hence we replace $\tilde{V}_R(\{\{B_\mu^i\}\})$ with
\begin{align}
\tilde{V}_{\kappa,R}(\{\{B_\mu^i\}\})
&= \bk^3\int_{I^3} \Bigg\{ \left|\frac{\langle [m_\kappa(\langle B, p_\kappa^{\vec{\rho}(r)} \rangle)]^{-1}B, p_\kappa^{\vec{\rho}(r)} \rangle}{3} \cdot \langle B, p_\kappa^{\vec{\rho}(r)} \rangle  \right|^2 \nonumber\\
& \hspace{1.6cm} - \sum_{(i,j) \in \Upsilon} \left|\langle \tilde{B}^i, p_\kappa^{\vec{\rho}(r)} \rangle  \cdot \langle \bar{B}_j^\kappa , p_\kappa^{\vec{\rho}(r)} \rangle \right|^2 \nonumber\\
& \hspace{1.6cm}+ \sum_{i=1}^3 \left|\frac{1}{3}\langle \tilde{B}^i, p_\kappa^{\vec{\rho}(r)} \rangle \cdot  \langle \tilde{B}^i, [m_\kappa(\langle B_i, p_\kappa^{\vec{\rho}(r)} \rangle)]^{-1}p_\kappa^{\vec{\rho}(r)} \rangle \right|^2
\Bigg\}^{1/2} |J_\rho|(r)\ dr. \label{e.v.10}
\end{align}
Note that we write \beq \left\langle\bar{B}_j, p_\kappa^{\vec{\rho}(r)} \right\rangle := \left\langle [m_\kappa(\langle B, p_\kappa^{\vec{\rho}(r)}\rangle)^{-1}B]_j, p_\kappa^{\vec{\rho}(r)} \right\rangle, \nonumber \eeq after making the approximations.

Together with the preceding approximation, we replace $\tilde{W}$ and $\tilde{V}$ with $\tilde{W}_\kappa$ and $\tilde{V}_\kappa$ respectively, defined earlier in Definition \ref{d.he.1}. Because $\tilde{V}_\kappa(\{B^j_\mu\})$ and $\tilde{V}_{\kappa,R}(\{\{B_\mu^i\}\})$ are scalars, we can bring them inside the time ordering operator $\mathcal{T}$ and trace, which is the linear functional $\widetilde{\Tr}$. Thus we approximate our path integral in Expression \ref{e.v.6} with
\beq
\frac{1}{\hat{Z}_{EH}}\widetilde{\Tr}\int \tilde{V}_\kappa(\{B^j_\mu\}) \tilde{W}_\kappa(\{A^k_{\alpha\beta}\})\tilde{V}_{\kappa,R}(\{\{B_\mu^i\}\})  e^{i\int_{\bR^4} A_0\cdot  B   -  A \cdot \tilde{B}}  D\Lambda \label{e.v.7}. \eeq This completes Steps \ref{d.cs.r.1c} and \ref{d.cs.r.1d} in Definition \ref{d.cs.r.1}.

\begin{defn}\label{d.vx.1}
Define the following 9-vector $\lambda_\kappa = (\lambda_{\kappa}^{1} , \lambda_{\kappa}^{2} , \lambda_{\kappa}^{3})$, whereby for $r \in I^3$,
\begin{align*}
\lambda_{\kappa}^{j}(r) =& \frac{-iq\kappa^2}{2\sqrt{4\pi}}\sum_{u=1}^{\on}\int_0^1 ds\ \left\langle p_\kappa^{\vec{\rho}_{r}}, \partial_0^{-1} p_\kappa^{\vec{y}^u_s}\right\rangle y^{u,\prime}_{j,s}\otimes \bar{E}\in \bC^{\times^3}.
\end{align*}

Also define a 3-vector $\tilde{\lambda}_\kappa(r)$ by
\begin{align*}
\tilde{\lambda}_{\kappa}(r) =& \frac{-iq\kappa^2}{2\sqrt{4\pi}}\sum_{u=1}^{\on}\int_0^1 ds\ \left\langle p_\kappa^{\vec{\rho}_{r}}, \partial_0^{-1} p_\kappa^{\vec{y}^u_s}\right\rangle y^{u,\prime}_{s} .
\end{align*}
\end{defn}

\begin{rem}
\begin{enumerate}
  \item We can also identify $\tilde{\lambda}_\kappa(r) \otimes \bar{E}$ with $\lambda_\kappa(r)$.
  \item Now, $\left\langle p_\kappa^{\vec{\rho}_{r}}, \partial_0^{-1} p_\kappa^{\vec{y}^u_s}\right\rangle$ means $\int_{\bR^4} p_\kappa^{\vec{\rho}_{r}} \cdot \left[ \partial_0^{-1} p_\kappa^{\vec{y}^u_s} \right]\ d\lambda$ and $\partial_0^{-1} p_\kappa^{\vec{y}^u_s}$ was defined using Equation (\ref{e.d.1}) in Definition \ref{d.lo.1}.
\end{enumerate}
\end{rem}

\begin{lem}
Recall $\mathcal{W}_\kappa^\pm$ were defined in Definition \ref{d.w.1}. Apply Step \ref{d.cs.r.1e} in Definition \ref{d.cs.r.1}, the path integral in Expression \ref{e.v.7} is hence computed as
\begin{align}
\widetilde{\Tr}\left(
  \begin{array}{cc}
    A_\kappa\bigotimes_{u=1}^{\on}\mathcal{W}_\kappa^+(q, ; \ol^u, \uL) &\ 0\\
    0 &\ B_\kappa\bigotimes_{u=1}^{\on}\mathcal{W}_\kappa^-(q; \ol^u, \uL) \\
  \end{array}
\right), \label{e.v.8}
\end{align}
whereby
\begin{align*}
A_\kappa :=& \bk^3\int_{I^3} \sqrt{\alpha_\kappa(\vec{\rho}(r))}|J_\rho|(r)\ dr, \ \
B_\kappa := \bk^3\int_{I^3} \sqrt{\beta_\kappa(\vec{\rho}(r))}|J_\rho|(r)\ dr,
\end{align*}
with $\alpha_\kappa$ and $\beta_\kappa$ both defined in Equations (\ref{ex.v.1}) and (\ref{ex.v.2}) respectively.
\end{lem}

\begin{proof}
In the proof of Lemma \ref{l.x.2}, we obtain the function $\tilde{V}_\kappa(\{B^j_\mu\})$ using the substitutions given by Expressions \ref{e.b.1} and \ref{e.b.2}.

Equation \ref{e.b.1} will give us \beq
m_\kappa(\langle B, p_\kappa^{\vec{\varrho}^v_{\bar{s}}}\rangle)^{-1}\longmapsto m_\kappa(\lambda_\kappa^v(\bar{s}))^{-1},\ m_\kappa(\langle B, p_\kappa^{\vec{\rho}_r} \rangle)^{-1} \longmapsto m_\kappa(\lambda_\kappa(r))^{-1}  \nonumber \eeq
and it also means that \beq
B_i \longmapsto -iq\kappa\bk\sum_{u=1}^{\on}\int_0^1 ds\ y^{u,\prime}_{s}\partial_0^{-1}p_\kappa^{\vec{y}^u_s} \otimes \rho_u^+(\hat{E}_s^{0i}). \nonumber \eeq

Hence it means we replace
\begin{align*}
m_\kappa(\langle B_i, p_\kappa^{\vec{\varrho}^v_{\bar{s}}}\rangle)^{-1}\longmapsto\ m_\kappa(\tilde{\lambda}_\kappa^v(\bar{s}))^{-1} \ {\rm and}\ [m_\kappa(\langle B_i, p_\kappa^{\vec{\rho}_r} \rangle)]^{-1} \longmapsto\ m_\kappa(\tilde{\lambda}_\kappa(r))^{-1}.
\end{align*}

Equation \ref{e.b.1} will also imply that
\begin{align*}
[m_\kappa(\langle B, &p_\kappa^{\vec{\rho}_r} \rangle)]^{-1}B \\
&\longmapsto
m_\kappa(\lambda_\kappa(r))^{-1}\left\{ -iq\kappa\bk\sum_{u=1}^{\on}\int_0^1 ds\ y^{u,\prime}_{ s}\partial_0^{-1}p_\kappa^{\vec{y}^u_s} \otimes (\rho_u^+(\hat{E}_s^{01}),  \rho_u^+(\hat{E}_s^{02}), \rho_u^+(\hat{E}_s^{03})) \right\} \\
&\equiv -iq\kappa\bk\sum_{u=1}^{\on}\int_0^1 ds\ m_\kappa(\tilde{\lambda}_\kappa(r))^{-1}[y^{u,\prime}_{ s}\partial_0^{-1}p_\kappa^{\vec{y}^u_s}] \otimes (\rho_u^+(\hat{E}_s^{01}),  \rho_u^+(\hat{E}_s^{02}), \rho_u^+(\hat{E}_s^{03})).
\end{align*}

Equation \ref{e.b.2} will lead us to replace \beq m_\kappa(\langle \tilde{B}^i, p_\kappa^{\vec{\rho}_r} \rangle)^{-1} \longmapsto
m_\kappa(\tilde{\lambda}_\kappa(r))^{-1}. \nonumber \eeq

Note that $\lambda_\kappa^v(\bar{s})$ and $\tilde{\lambda}_\kappa^v(\bar{s})$ were given by Equations (\ref{e.l.1})  and (\ref{e.l.2}) respectively; $\lambda_\kappa(r)$ and $\tilde{\lambda}_\kappa(r)$ were defined in Definition \ref{d.vx.1}.

Suppose $\rho_u \equiv (\rho_u^+, 0)$. Now the above substitutions, when applied to $\tilde{V}_\kappa$ as defined in Definition \ref{d.he.1} will yield $\mathcal{W}_\kappa^+$ and
when applied to the term
\beq \left|\frac{\langle [m_\kappa(\langle B, p_\kappa^{\vec{\rho}(r)} \rangle)]^{-1}B, p_\kappa^{\vec{\rho}(r)} \rangle}{3} \cdot \langle B, p_\kappa^{\vec{\rho}(r)} \rangle  \right|^2, \nonumber \eeq
will yield (after some simplification)
\begin{align}
\alpha_\kappa(\vec{\rho}(r)) := q^4&\kappa^4 \Bigg|\sum_{u, \bar{u}=1}^{\on} \bk\int_0^1 ds\ \left\langle m_\kappa(\tilde{\lambda}_\kappa(r))^{-1}[y_s^{u,\prime}p_\kappa^{\vec{y}_s^u}] ,\partial_0^{-1}p_\kappa^{\vec{\rho}(r)}, \right\rangle \cdot  \bk\int_0^1 d\bar{s}\  \left\langle \partial_0^{-1}p_\kappa^{\vec{y}_{\bar{s}}^{\bar{u}}}, p_\kappa^{\vec{\rho}(r)} \right\rangle y_{\bar{s}}^{\bar{u},\prime} \nonumber \\
& \hspace{1cm} \otimes \frac{1}{3}\sum_{i=1}^3 \rho_u^+(\hat{E}_s^{0i}) \otimes \rho_{\bar{u}}^+(\hat{E}_{\bar{s}}^{0i})\Bigg|^2
. \label{ex.v.1}
\end{align}
Note that we made use of $\langle \partial_0^{-1}f, g \rangle = -\langle f, \partial_0^{-1}g \rangle$ to obtain the above formula. Refer to Remark \ref{r.v.1} for the meaning of the term $\left\langle \partial_0^{-1}p_\kappa^{\vec{\rho}(r)}, m_\kappa(\tilde{\lambda}_\kappa(r))^{-1}[y_s^{u,\prime}p_\kappa^{\vec{y}_s^u}] \right\rangle$.



Apply Step \ref{d.cs.r.1e} in Definition \ref{d.cs.1}, the path integral in Expression \ref{e.v.7} is now define as
\begin{align}
\widetilde{\Tr}\ A_\kappa &\bigotimes_{u=1}^{n}\exp\Bigg\{ \frac{iq}{4}\frac{\kappa^3}{4\pi}\sum_{v=1}^{\underline{n}}\int_{I^2}\ d\hat{s}\  \left\langle  y^{u,\prime}_{j,s}\partial_0^{-1}p_\kappa^{\vec{y}^u_s} \otimes \rho_u^+(\widetilde{E}_s), \left\{ \left[m_\kappa(\lambda_\kappa^v(\bar{s}) )\right]^{-1} (\varrho_{\bar{s}}^{v,\prime} p_\kappa^{\vec{\varrho}^{v}_{\bar{s}}}) \right\}^j\right\rangle\Bigg\} \nonumber \\
\equiv& \widetilde{\Tr}\ \left[\bk^3\int_{I^3} \sqrt{a_\kappa(\vec{\rho}(r))}|J_\rho|(r)\ dr\bigotimes_{u=1}^{n}\mathcal{W}_\kappa^+(q, ; \ol^u, \uL)\right], \label{e.v.9a}
\end{align}
which follows from Definition \ref{d.w.1}.

Suppose $\rho_u \equiv (0, \rho_u^-)$. Now the above substitutions, when applied to $\tilde{V}_\kappa$ as defined in Definition \ref{d.he.1} will yield $\mathcal{W}_\kappa^-$.

From Notation \ref{n.su.1}, note that $\rho^+(\hat{E}^{0i}) \otimes \rho^-(\hat{E}^{\tau(j)}) \equiv 0$ for any $i$ and $j$. Hence, after the substitution, the term $\sum_{(i,j) \in \Upsilon}  |\tilde{B}^i  \cdot \bar{B}_j^\kappa|^2$ will not contribute to anything, but \beq \sum_{i=1}^3  \left|\frac{1}{3}\langle \tilde{B}^i, p_\kappa^{\vec{\rho}(r)} \rangle \cdot  \langle \tilde{B}^i, [m_\kappa(\langle B_i, p_\kappa^{\vec{\rho}(r)} \rangle)]^{-1}p_\kappa^{\vec{\rho}(r)} \rangle \right|^2 \nonumber \eeq will give us the term
\begin{align}
\beta_\kappa(\vec{\rho}(r)) := q^4&\kappa^4\Bigg|\sum_{u, \bar{u}=1}^{\on}
\bk\int_0^1 d\bar{s}\  \left\langle \partial_0^{-1}p_\kappa^{\vec{y}_{\bar{s}}^{\bar{u}}} , p_\kappa^{\vec{\rho}(r)} \right\rangle y_{\bar{s}}^{\bar{u},\prime}
\cdot
\bk\int_0^1 ds\ \left\langle \partial_0^{-1}p_\kappa^{\vec{y}_s^u}, m_\kappa(\tilde{\lambda}_\kappa(r))^{-1}[y_s^{u,\prime}p_\kappa^{\vec{\rho}(r)}] \right\rangle   \nonumber \\
& \hspace{1cm} \otimes \frac{1}{3}\sum_{i=1}^3 \rho_{\bar{u}}^-(\hat{E}_{\bar{s}}^{\tau(i)}) \otimes \rho_u^-(\hat{E}_s^{\tau(i)})
\Bigg|^2. \label{ex.v.2}
\end{align}


Apply Step \ref{d.cs.r.1e} in Definition \ref{d.cs.1}, the path integral in Expression \ref{e.v.7} is now define as
\begin{align}
\widetilde{\Tr}\ B_\kappa &\bigotimes_{u=1}^{n}\exp\Bigg\{
\frac{-iq}{4}\frac{\kappa^3}{4\pi}\sum_{v=1}^{\underline{n}}\int_{I^2}\ d\hat{s}\ \left\langle y^{u,\prime}_{k,s}\partial_0^{-1} p_\kappa^{\vec{y}^u_s}, \left\{ \left[m_\kappa(\tilde{\lambda}_\kappa^v(\bar{s}) )\right]^{-1}(\varrho_{\bar{s}}^{v,\prime} p_\kappa^{\vec{\varrho}^{v}_{\bar{s}}}) \right\}^k \right\rangle  \otimes \sum_{j=1}^3\rho_u^-(\hat{E}_s^{\tau(j)}) \Bigg\}\nonumber \\
\equiv& \widetilde{\Tr}\ \left[\bk^3\int_{I^3} \sqrt{b_\kappa(\vec{\rho}(r))}|J_\rho|(r)\ dr\bigotimes_{u=1}^{n} \mathcal{W}_\kappa^-(q, ; \ol^u, \uL)\right], \label{e.v.9b}
\end{align}
which follows from Definition \ref{d.w.1}.

In the general case $\rho_u \equiv (\rho_u^+, \rho_u^-)$, Expressions \ref{e.v.9a} and \ref{e.v.9b} will give us our desired result. This completes Step \ref{d.cs.r.1e} in Definition \ref{d.cs.r.1}.
\end{proof}

\begin{rem}\label{r.v.1}
\begin{enumerate}
  \item Note that $m_\kappa(\tilde{\lambda}_\kappa(r))^{-1}[y_s^{u,\prime}p_\kappa^{\vec{y}_s^u}] \equiv (f_1, f_2, f_3)$ is a 3-vector, each $f_i$ is a function on $\bR^4$. And $\left\langle \partial_0^{-1}p_\kappa^{\vec{\rho}(r)}, m_\kappa(\tilde{\lambda}_\kappa(r))^{-1}[y_s^{u,\prime}p_\kappa^{\vec{y}_s^u}] \right\rangle$ means  \begin{align*}
\left\langle \partial_0^{-1}p_\kappa^{\vec{\rho}(r)}, (f_1, f_2, f_3) \right\rangle
=  \left( \left\langle \partial_0^{-1}p_\kappa^{\vec{\rho}(r)}, f_1 \right\rangle, \left\langle \partial_0^{-1}p_\kappa^{\vec{\rho}(r)}, f_2 \right\rangle, \left\langle \partial_0^{-1}p_\kappa^{\vec{\rho}(r)}, f_3 \right\rangle \right)
.
\end{align*}

  \item When $V = \emptyset$, then Expression \ref{e.v.8} reduces to $Z(\kappa, q; \chi(\oL, \uL) )$ as defined in Equation (\ref{e.z.1}).
\end{enumerate}
\end{rem}

There is a problem with Expression \ref{e.v.8}, as the square root of Expressions \ref{ex.v.1} and \ref{ex.v.2} do not make sense.  The problem lies with
\begin{align*}
\sum_{i=1}^3 & \rho_u^+(\hat{E}_s^{0i}) \otimes \rho_{\bar{u}}^+(\hat{E}_{\bar{s}}^{0i}) \ {\rm and}\
\sum_{i=1}^3  \rho_u^-(\hat{E})_s^{\tau(i)}) \otimes \rho_{\bar{u}}^-(\hat{E}_{\bar{s}}^{\tau(i)}),
\end{align*}
which lies in ${\rm End}(V_{\bar{u}}^+)$ and ${\rm End}(V_{\bar{u}}^-)$ respectively.

Of course, the above sum of tensor products is not the problem. The problem will come later, when we have to take the absolute value. We do not know how to take the modulus of $\sum_{i=1}^3 \rho_u^+(\hat{E}^{0i}) \otimes \rho_{\bar{u}}^+(\hat{E}^{0i})$ and $\sum_{i=1}^3 \rho_u^-(\hat{E}^{\tau(i)}) \otimes \rho_{\bar{u}}^-(\hat{E}^{\tau(i)})$, when $\rho_u^\pm \neq \rho_{\bar{u}}^\pm$.

Now, the reader may think that if we make the representation to be the same throughout, i.e. all the component in the matter hyperlink have the same color, then we will resolve the problem. Unfortunately, Expression \ref{e.v.8} is still not defined. The problem lies with the square root and the time ordering operator. The square root function or the modulus function is not analytic at 0, so it is not clear how to apply the time ordering operator.

In a sequel \cite{EH-Lim04} to this article, we will show that when we take the limit, the path integral is computed using the crossings of a link diagram. In fact, there is another problem that is lurking, which is the self-linking problem (for links) first pointed out by Witten in \cite{MR990772}. We will not give the details here; we refer the reader to \cite{CS-Lim02} on an explanation of the self-linking problem. A quick answer is that we need to consider a ribbon or a framed link instead, not a link. A ribbon is a link, equipped with a non-tangential normal, also known as a frame, to the link. This normal, will give rise to half-twists on the link diagram.


Project the matter hyperlink in $\bR^3$ and we will refer to it as a link. Let $l^1, \ldots l^{\on}$ be the components in the link, with each $l^u$ being a knot in $\bR^3$. For each knot $l^u$, let $\mathcal{N}(l^u)$ be a tubular neighborhood of $l^u$, homeomorphic to the torus, containing $l^u$ as a longitude curve in it, such that $\mathcal{N}(l^u) \cap \mathcal{N}(l^{\bar{u}}) = \emptyset$ if $u \neq \bar{u}$.

Now, we can always write the compact region $R$ as a disjoint union $\bigcup_{v=1}^{\bar{m}} R_v$, such that either
\begin{description}
  \item[i] \noindent $R_v \subseteq \mathcal{N}(l^u)$ for some $u$ and $R_v$ contains an arc $\tilde{l} \subset l^u$, such that for every $k= 1,2, 3$, $\tilde{l}$ projects onto the plane $\Sigma_k$ to form a link diagram as defined in \cite{CS-Lim02} with no crossing;
  \item[ii] \noindent or $R_v \cap \mathcal{N}(l^u) = \emptyset$ for every $u = 1, \ldots, \on$.
\end{description}

So, it suffices to consider the path integral over a compact region $R \subset \bR^3$ which either is inside the tubular neighborhood for some knot $l^u$ or it does not intersect any tubular neighborhood at all. By considering such a compact region $R$, we will see in a sequel to this article, the crossings in a link diagram no longer contribute to the path integral.

From \cite{CS-Lim02}, we see that each half-twist on a link diagram of a knot $l^u$ gives rise to an operator $\sum_{i=1}^3 \rho_u^+(\hat{E}^{0i}) \rho_{u}^+(\hat{E}^{0i})$ and $\sum_{i=1}^3 \rho_u^-(\hat{E}^{\tau(i)}) \rho_{u}^-(\hat{E}^{\tau(i)})$, which are $-\xi_{\rho_{u}^+}I_{\rho_u^+}$ and $-\xi_{\rho_{u}^-}I_{\rho_u^-}$ respectively. Then, it is clear how to take the absolute value, i.e. \beq \left| -\xi_{\rho_{u}^+}I_{\rho_u^+} \right| = \xi_{\rho_{u}^+}I_{\rho_u^+},\ \ \left| -\xi_{\rho_{u}^-}I_{\rho_u^-} \right| = \xi_{\rho_{u}^-}I_{\rho_u^-}. \nonumber \eeq Do note that the volume functional commutes with the matrices, so indeed, we really need to get some scalar multiple of the identity.

\begin{notation}\label{n.r.2}
Let $L = \{\ol^1, \ldots, \ol^{\on}\}$ be a matter hyperlink. Project each $\ol^u$ into $\bR^3$ to form a knot $l^u$ and let $\mathcal{N}(l^u)$ be a tubular neighborhood of $l^u$. Given a compact region $R \subset \bR^3$, write $R = \bigcup_{v=1}^{\bar{m}}R_v$ as a disjoint union, such that either $R_v \subseteq \mathcal{N}(l^u)$ for some $u$ or $R_v \cap \mathcal{N}(l^u) = \emptyset$ for every $u = 1, \ldots, \on$. Let $I_v^3 \subset I^3$ such that $\rho: I_v^3 \rightarrow R_v \subset R$ be a parametrization of $R_v$.
\end{notation}

\begin{defn}
From Notation \ref{n.r.2}, we write $R = \bigcup_{v=1}^{\bar{m}}R_v$ as a disjoint union of regions. Hence we can write the path integral in Expression \ref{e.v.7} as
\begin{align*}
\frac{1}{\hat{Z}_{EH}}&\widetilde{\Tr}\int \tilde{V}_\kappa(\{B^j_\mu\}) \tilde{W}_\kappa(\{A^k_{\alpha\beta}\})\tilde{V}_{\kappa,R}(\{B_\mu^j\}) \cdot  e^{i\int_{\bR^4} A_0\cdot  B   -  A \cdot \tilde{B}}  D\Lambda \\
=&
\sum_{v=1}^{\bar{m}}\frac{1}{\hat{Z}_{EH}}\widetilde{\Tr}\int \tilde{V}_\kappa(\{B^j_\mu\}) \tilde{W}_\kappa(\{A^k_{\alpha\beta}\})\tilde{V}_{\kappa,R_v}(\{B_\mu^j\}) \cdot  e^{i\int_{\bR^4} A_0\cdot  B   -  A \cdot \tilde{B}}  D\Lambda.
\end{align*}

We now define the path integral \beq
\frac{1}{\hat{Z}_{EH}}\int \tilde{V}_\kappa(\{B^j_\mu\}) \tilde{W}_\kappa(\{A^k_{\alpha\beta}\})\tilde{V}_{\kappa,R}(\{B_\mu^j\}) \cdot  e^{i\int_{\bR^4} A_0\cdot  B   -  A \cdot \tilde{B}}  D\Lambda, \nonumber \eeq as
\begin{align}
q^2\prod_{\bar{u}=1}^{\on} \Bigg\{&  \Bigg[ \sum_{v=1}^{\bar{m}}\kappa^2\bk^5\sum_{u=1}^{\on} \int_{r\in I_v^3}dr |J_\rho|(r)\Bigg|\int_{I^2} d\hat{s}\  \left\langle m_\kappa(\tilde{\lambda}_\kappa(r))^{-1}[y_s^{u,\prime}p_\kappa^{\vec{y}_s^u}], \partial_0^{-1}p_\kappa^{\vec{\rho}(r)} \right\rangle  \nonumber \\
&\hspace{4.2cm} \cdot\  \left\langle \partial_0^{-1}p_\kappa^{\vec{y}_{\bar{s}}^{u}}, p_\kappa^{\vec{\rho}(r)} \right\rangle y_{\bar{s}}^{u,\prime}
\frac{\xi_{\rho_{u}^+}}{3}
\Bigg| \Bigg]^{1/\on} \Tr_{\rho_{\bar{u}}^+}\mathcal{W}^+_\kappa(q; \ol^{\bar{u}}, \uL) \nonumber\\
+&\Bigg[\sum_{v=1}^{\bar{m}} \kappa^2\bk^5 \sum_{u=1}^{\on} \int_{r\in I_v^3}dr |J_\rho|(r)\Bigg|\int_{I^2} d\hat{s}\  \left\langle \partial_0^{-1}p_\kappa^{\vec{y}_s^u}, m_\kappa(\tilde{\lambda}_\kappa(r))^{-1}y_s^{u,\prime}p_\kappa^{\vec{\rho}(r)} \right\rangle \nonumber \\
&\hspace{4.2cm} \cdot\  \left\langle \partial_0^{-1}p_\kappa^{\vec{y}_{\bar{s}}^{u}}, p_\kappa^{\vec{\rho}(r)} \right\rangle y_{\bar{s}}^{u,\prime}
\frac{\xi_{\rho_{u}^-}}{3}
\Bigg| \Bigg]^{1/\on} \Tr_{\rho_{\bar{u}}^-}\mathcal{W}^-_\kappa(q; \ol^{\bar{u}}, \uL)
\Bigg\}, \label{e.v.4}
\end{align}
from Expression \ref{e.v.8}. Notice that it is no longer necessary to keep track of the matrices, so we remove the time ordering operator.
\end{defn}

\begin{lem}
Refer to Notation \ref{n.k.1} and Equation (\ref{e.h.5}).
The limit of Expression \ref{e.v.4} as $\kappa$ goes to infinity, is given by computing the limit of
\begin{align}
 q^2\prod_{\bar{u}=1}^{\on}\Bigg\{&  \Bigg[ \sum_{v=1}^{\bar{m}}\tilde{\kappa}\sum_{u=1}^{\on} \int_{r\in I_v^3}dr |J_\rho|(r)\Bigg|\int_{I^2} d\hat{s}\  \epsilon^{ijk}\left\langle p_\kappa^{\vec{y}_{s}^u}, p_\kappa^{\vec{\rho}(r)} \right\rangle_k y_{i,s}^{u,\prime}y_{j,\bar{s}}^{u,\prime} \nonumber \\
&\hspace{2cm}\times e^{-\kappa^2|y_{\bar{s}}^{u} - \rho(r)|^2/8} \left\langle \partial_0^{-1}q_\kappa^{y_{0,\bar{s}}^{u}}, q_\kappa^{0} \right\rangle
\xi_{\rho_{u}^+}
\Bigg| \Bigg]^{1/\on} \Tr_{\rho_{\bar{u}}^+}\hat{\mathcal{W}}^+_\kappa(q; \ol^{\bar{u}}, \uL) \nonumber\\
+&\Bigg[ \sum_{v=1}^{\bar{m}} \tilde{\kappa}\sum_{u=1}^{\on} \int_{r\in I_v^3}dr |J_\rho|(r)\Bigg|\int_{I^2} d\hat{s}\  \epsilon^{ijk}\left\langle p_\kappa^{\vec{y}_{s}^u}, p_\kappa^{\vec{\rho}(r)} \right\rangle_k y_{i,s}^{u,\prime}y_{j,\bar{s}}^{u,\prime} \nonumber \\
&\hspace{2cm}\times e^{-\kappa^2|y_{\bar{s}}^{u} - \rho(r)|^2/8} \left\langle \partial_0^{-1}q_\kappa^{y_{0,\bar{s}}^{u}}, q_\kappa^{0} \right\rangle
\xi_{\rho_{u}^-}
\Bigg| \Bigg]^{1/\on} \Tr_{\rho_{\bar{u}}^-}\hat{\mathcal{W}}^-_\kappa(q; \ol^{\bar{u}}, \uL)
\Bigg\}, \label{ex.v.3}
\end{align}
as $\kappa$ goes to infinity. Here,
\beq \tilde{\kappa} = \frac{\sqrt\pi}{2}\frac{\kappa}{4}\left(\frac{\kappa}{\sqrt{2\pi}}\right)^2
\left(\frac{\kappa^{2} }{8\pi}\right)^2. \nonumber \eeq
\end{lem}

\begin{proof}
From Lemma \ref{l.k.1}, we have seen that $\kappa\lambda_{\kappa}^v(\bar{s}) \rightarrow 0$ and $\kappa\tilde{\lambda}_{\kappa}^v(\bar{s}) \rightarrow 0$ as $\kappa \rightarrow \infty$ .

Now, \beq \kappa^3\left\langle p_\kappa^{\vec{\varrho}^{v}_{\bar{s}}}, \partial_0^{-1} (p_\kappa^{\vec{y}^u_s})\right\rangle = \kappa^2\left\langle p_\kappa^{\varrho^v_{\bar{s}}}, p_\kappa^{y^u_s} \right\rangle \cdot \kappa\left\langle q_\kappa^{\varrho^v_{0,\bar{s}}}, \partial_0^{-1}q_\kappa^{y^u_{0,s}} \right\rangle \nonumber \eeq
and \beq \kappa^3\left\langle \partial_0^{-1}p_\kappa^{\vec{y}_s^u}, p_\kappa^{\vec{\rho}(r)} \right\rangle = \kappa^2\left\langle p_\kappa^{y_s^u}, p_\kappa^{\rho(r)}\right\rangle \cdot \kappa\left\langle \partial_0^{-1}q_\kappa^{y_{0,s}^u}, q_\kappa^0 \right\rangle = \kappa^2 e^{-\kappa^2|y_{\bar{s}}^{u} - \rho(r)|^2/8} \cdot \kappa\left\langle \partial_0^{-1}q_\kappa^{y_{0,s}^u}, q_\kappa^0 \right\rangle, \nonumber \eeq and both converge to 0, using Lemma \ref{l.l.5}.

Refer to Definition \ref{d.lo.1}. From Item \ref{i.l.7} in Lemma \ref{l.l.5}, we see that
\begin{align*}
\left\langle \partial_0^{-1}p_\kappa^{\vec{y}_{\bar{s}}^{u}}, p_\kappa^{\vec{\rho}(r)} \right\rangle =& \left\langle \partial_0^{-1}q_\kappa^{y_{0,\bar{s}}^u}, q_\kappa^{0} \right\rangle \cdot \left\langle p_\kappa^{y_{\bar{s}}^u}, p_\kappa^{\rho(r)} \right\rangle = e^{-\kappa^2|y_{\bar{s}}^{u} - \rho(r)|^2/8}  \left\langle \partial_0^{-1}q_\kappa^{y_{0,\bar{s}}^{u}}, q_\kappa^{0} \right\rangle.
\end{align*}

From Remark \ref{r.m.1}, we see that
\begin{align*}
m_\kappa(0)^{-1}[y_s^{u,\prime}p_\kappa^{\vec{y}_s^u}] \cdot y_{\bar{s}}^u
=& \kappa\left[\partial_2^{-1}p_\kappa^{\vec{y}_s^u}y_{3,s}^{u,\prime} - \partial_3^{-1}p_\kappa^{\vec{y}_{2,s}^u}y_{2,s}^{u,\prime}\right]y_{1,\bar{s}}^{u,\prime} + \kappa\left[\partial_3^{-1}p_\kappa^{\vec{y}_s^u}y_{1,s}^{u,\prime} - \partial_1^{-1}p_\kappa^{\vec{y}_{3,s}^u}y_{3,s}^{u,\prime}\right]y_{2,\bar{s}}^{u,\prime}\\
&+ \kappa\left[\partial_1^{-1}p_\kappa^{\vec{y}_s^u}y_{2,s}^{u,\prime} - \partial_2^{-1}p_\kappa^{\vec{y}_{1,s}^u}y_{1,s}^{u,\prime}\right]y_{3,\bar{s}}^{u,\prime}.
\end{align*}
Thus,
\begin{align*}
\left\langle m_\kappa(0)^{-1}[y_s^{u,\prime}p_\kappa^{\vec{y}_s^u}], \partial_0^{-1}p_\kappa^{\vec{\rho}(r)} \right\rangle \cdot y_{\bar{s}}^u =& \left\langle p_\kappa^{\vec{\rho}(r)}, p_\kappa^{\vec{y}_{s}^u}  \right\rangle_k [y_s^{u,\prime}\times   y_{\bar{s}}^{u,\prime}]^k
\equiv \left\langle p_\kappa^{\vec{y}_{s}^u}, p_\kappa^{\vec{\rho}(r)}  \right\rangle_k [y_s^{u,\prime}\times   y_{\bar{s}}^{u,\prime}]^k,
\end{align*}
using Notation \ref{n.k.1}.

Since $\kappa\lambda_\kappa,\ \kappa\tilde{\lambda}_\kappa \rightarrow 0$, the limit of Expressions \ref{ex.v.1} and \ref{ex.v.2} are equivalent to compute the limits of
\begin{align*}
&q^4\kappa^4\bk^4\left|\sum_{u=1}^{\on} \int_{I^2} d\hat{s}\ \  \left\langle p_\kappa^{\vec{y}_{s}^u}, p_\kappa^{\vec{\rho}(r)} \right\rangle_k [y_s^{u,\prime}\times   y_{\bar{s}}^{u,\prime}]^k e^{-\kappa^2|y_{\bar{s}}^{u} - \rho(r)|^2/8}  \left\langle \partial_0^{-1}q_\kappa^{y_{0,\bar{s}}^{u}}, q_\kappa^{0} \right\rangle\otimes \frac{\xi_{\rho_{u}^\pm}}{3}\right|^2  \\
&=q^4\kappa^4\bk^4\left|\sum_{u=1}^{\on} \int_{I^2} d\hat{s}\  \epsilon^{ijk}\left\langle p_\kappa^{\vec{y}_{s}^u}, p_\kappa^{\vec{\rho}(r)} \right\rangle_k y_{i,s}^{u,\prime}y_{j,\bar{s}}^{u,\prime} e^{-\kappa^2|y_{\bar{s}}^{u} - \rho(r)|^2/8}  \left\langle \partial_0^{-1}q_\kappa^{y_{0,\bar{s}}^{u}}, q_\kappa^{0} \right\rangle\otimes \frac{\xi_{\rho_{u}^\pm}}{3}\right|^2
\end{align*}
respectively as $\kappa$ goes to infinity.

Replace Expressions \ref{ex.v.1} and \ref{ex.v.2} in Expression \ref{e.v.4} with the above expressions and using the proof in Corollary \ref{c.x.1}, we will obtain our result.
\end{proof}

\begin{defn}(Volume Path Integral)\label{d.vpi}\\
We define Expression \ref{e.v.5} as the limit as $\kappa$ goes to infinity, of Expression \ref{ex.v.3}.
\end{defn}

\begin{rem}
The path integral in Expression \ref{e.v.4} will of course depend on the choice of partition $\{R_v\}_{v=1}^{\bar{m}}$. But its limit as $\kappa$ goes to infinity will be shown to be independent of this partition in a sequel.
\end{rem}

\section{Curvature Path Integral}\label{s.co}

Because curvature is a two form, we need to choose a surface $S$ and integrate curvature over it. Now, $S$ should be orientable, closed and bounded, with or without boundary, and disjoint from $\uL$. We allow $S$ to be disconnected, with finite number of components. Furthermore, we insist the link $\pi_a(\uL)$ intersects $\pi_a(S)$ at most finitely many number of points inside $\pi_a(S)$.

\begin{notation}
Let \beq \partial_0A_{\alpha\beta}^i \equiv \frac{\partial A_{\alpha\beta}^i}{\partial x_0},\ \ \partial_iA_{\alpha\beta}^j \equiv \frac{\partial A_{\alpha \beta}^j}{\partial x_i}. \nonumber \eeq
\end{notation}

In terms of $\{A_{\alpha\beta}^i\}$, we define
\begin{align}
F_S(\{A^k_{\alpha\beta} \})
:=& \frac{1}{2}\int_S \partial_0A_{\alpha\beta}^i \otimes dx_0 \wedge dx_i \otimes \hat{E}^{\alpha\beta} + \frac{1}{2}\sum_{i=1}^3\int_S \partial_iA_{\alpha\beta}^j \otimes dx_i \wedge dx_j \otimes \hat{E}^{\alpha \beta} \nonumber \\
&+ \frac{1}{2}\int_S A_{\alpha\beta}^iA_{\gamma\mu}^j\otimes dx_i \wedge dx_j \otimes [\hat{E}^{\alpha\beta}, \hat{E}^{\gamma\mu}] \nonumber\\
:=& F_S^1(\{A^k_{\alpha\beta}\}) + F_S^2(\{A^k_{\alpha\beta}\}) + F_S^3(\{A^k_{\alpha\beta}\}). \nonumber
\end{align}
Note that we do not specify any representation for $\hat{E}^{\alpha\beta}$.

Refer to Notation \ref{n.x.1}. We are going to define the following curvature path integral, given by
\beq
\frac{1}{Z_{EH}}\int F_S(\{A^k_{\alpha\beta}\}) V( \{\ul^v\}_{v=1}^{\un})(\{B^i_\mu\})W(q; \{\ol^u, \rho_u\}_{u=1}^{\on})(\{A^k_{\alpha\beta}\}) e^{i\int_{\bR^4} A_0\cdot  B \vec{\times} B   -  A \cdot \tilde{B}}  D\Lambda, \label{e.f.1} \eeq where \beq Z_{EH} = \int \exp\left[i\int_{\bR^4} \partial_0 A_0\cdot B \vec{\times} B   - \partial_0 A \cdot \tilde{B}\right] D\Lambda. \nonumber \eeq

We will now make use of Chern-Simon rules given in Definition \ref{d.cs.r.1} to make sense of the path integral in Expression \ref{e.f.1}, as before

\begin{rem}
\begin{enumerate}
  \item When $S$ is the empty set, we define $F_\emptyset \equiv 1$, so we write Expression \ref{e.f.1} as $Z(q; \chi(\oL, \uL))$, which was termed as the Wilson Loop observable of the colored hyperlink $\chi(\oL, \uL)$ in Remark \ref{r.ax.1}.
  \item The path integral will take values in $\mathfrak{su}(2) \times \mathfrak{su}(2)$.
\end{enumerate}
\end{rem}

\subsection{$F_S^1$ Path Integral}\label{ss.fpi.1a}

We will first define the path integral \beq
\frac{1}{Z_{EH}}\int F_S^1(\{A^k_{\alpha\beta}\})  V( \{\ul^v\}_{v=1}^{\un})(\{B^i_\mu\})W(q; \{\ol^u, \rho_u\}_{u=1}^{\on})(\{A^k_{\alpha\beta}\}) e^{i\int_{\bR^4} \partial_0 A_0 \cdot B \vec{\times} B   -  A \cdot \tilde{B}}  D\Lambda, \label{e.f.1a} \eeq which we will denote it as $\hat{F}_S^1[Z(q; \chi(\oL, \uL))]$.

Recall we defined $\tilde{V}(\{B^j_\mu\})$ and $\tilde{W}(\{A^k_{\alpha\beta}\})$ respectively in Equations (\ref{e.w.2}) and (\ref{e.w.3}).

\begin{lem}\label{l.f.1}
After doing a change of variables given in Notation \ref{n.c.1}, the path integral in Expression (\ref{e.f.1a}) is defined as \beq
\frac{1}{\hat{Z}_{EH}}\int \bar{F}_S^1(\{A^k_{\alpha\beta}\})\tilde{V}(\{B^j_\mu\}) \tilde{W}(\{A^k_{\alpha\beta}\})  e^{i\int_{\bR^4} A_0\cdot  B   -  A \cdot \tilde{B}}  D\Lambda \label{e.f.2} \eeq with \beq
\hat{Z}_{EH} = \int \exp\left[i\int_{\bR^4} A_0\cdot  B   -  A \cdot \tilde{B}\right]
D\Lambda, \nonumber \eeq after applying Steps \ref{d.cs.r.1a} and \ref{d.cs.r.1b} in Definition \ref{d.cs.r.1}. Here, $\bar{F}_S^1(\{A^k_{\alpha\beta}\})$ is defined in Equation (\ref{e.f.1b}).

\end{lem}

\begin{proof}
In Lemma \ref{l.x.1}, we replace $\partial_0A_0$ with $A_0$ and $\partial_0 A$ with $A$ respectively. Recall we define $J_{\alpha\beta}$ in Notation \ref{n.s.3}. Thus define
\begin{align}
F_S^1(\{A^k_{\alpha\beta}\}) \longmapsto& \int_{I^2}\ d\hat{t} \left\langle A_{0j}^i, \delta^{\vec{\sigma}(\hat{t})} \right\rangle J_{0i}(\hat{t}) \otimes \hat{E}^{0j} \nonumber \\
&\hspace{2cm}+ \left\langle A_{\tau(j)}^i, \delta^{\vec{\sigma}(\hat{t})} \right\rangle  J_{0i}(\hat{t})) \otimes \hat{E}^{\tau(j)} =: \bar{F}_S^1(\{A^k_{\alpha\beta}\}). \label{e.f.1b}
\end{align}

The rest of the proof is exactly the same as the proof in Lemma \ref{l.x.1}, which will give us Expression \ref{e.f.2}.
\end{proof}

As explained in the earlier sections, $m(B(\vec{x}))^{-1}$, $m(B_i(\vec{x}))^{-1}$ cannot be defined. To make sense of the path integral given by Expression \ref{e.f.2}, we will make the following approximations as done in Sections \ref{s.ao} and \ref{s.vo}. We approximate the Dirac-delta function with $p_\kappa$, $m(B)$ with $m_\kappa(B)$ and $m(B_i)$ with $m_\kappa(B_i)$.

With the above approximations, the path integral that was obtained previously is of the form \beq
\frac{1}{\hat{Z}_{EH}}\int F(\{B_i^j\}, \{\tilde{B}_{i}^j\})e^{\langle A_{0i}^j, \alpha_j^i \rangle}e^{\langle A_{\tau(i)}^j, \beta_j^i \rangle} e^{i\int_{\bR^4} A_0\cdot  B   -  A \cdot \tilde{B}}  D\Lambda, \label{e.x.2} \eeq whereby $F$ is some continuous function dependent on the variables $B_i^j$ and $\tilde{B}_i^j$, and $\alpha_i^j$, $\beta_i^j$ are in $C^\infty(\bR^4)$. 

According to the rules of the Chern-Simons integral given in Definition \ref{d.cs.r.1}, we substitute $B_i^j$ with $\alpha_i^j$ and $\tilde{B}_i^j$ with $\beta_i^j$ respectively inside the function $F$, and we will obtain the formula for the path integral. More importantly, in Sections \ref{s.ao} and \ref{s.vo}, we replace
\begin{align}
\left[m_\kappa(\langle B, \delta^{\vec{\varrho}^v_{\bar{s}}} \rangle )\right]^{-1} \longmapsto & m_\kappa(\lambda_\kappa^v(\bar{s}))^{-1}, \label{e.x.3a}\\
\left[m_\kappa(\langle B_i, \delta^{\vec{\varrho}^v_{\bar{s}}} \rangle )\right]^{-1} \longmapsto & m_\kappa(\tilde{\lambda}_\kappa^v(\bar{s}))^{-1}. \label{e.x.3b}
\end{align}

However, if one look at Expression \ref{e.f.2}, it is not in the form given by Expression \ref{e.x.2}, even after we replace $m(B(\vec{\varrho}_{\bar{s}}^v))^{-1}$ and $m(B_i(\vec{\varrho}_{\bar{s}}^v))^{-1}$ with $m_\kappa(B(\vec{\varrho}_{\bar{s}}^v))^{-1}$ and $m_\kappa(B_i(\vec{\varrho}_{\bar{s}}^v))^{-1}$ respectively.

Refer to Notation \ref{n.n.3}. To define the above path integral, we make an approximation to $m(B(\vec{\varrho}^v_{\bar{s}}) )$ and $m(B_i (\vec{\varrho}^v_{\bar{s}}) )$. However, unlike the approximation we made in Sections \ref{s.ao} and \ref{s.vo}, to simplify the integral, we approximate it with $m_\kappa(\lambda_\kappa^v(\bar{s}))$ and $m_\kappa(\tilde{\lambda}_\kappa^v(\bar{s}))$ respectively. Specifically, we will replace $m(B(\vec{\varrho}_{\bar{s}}^v))^{-1}$ and $m(B_i(\vec{\varrho}_{\bar{s}}^v))^{-1}$ with the substitutions given by Equations (\ref{e.x.3a}) and (\ref{e.x.3b}) respectively. As mentioned earlier, the reason is Expression \ref{e.f.2} is not a Chern-Simons integral. To facilitate computations, we will first make this approximation.

With these approximations, we replace $\tilde{W}$ and $\tilde{V}$ with $\bar{W}_\kappa$ and $\bar{V}_\kappa$ respectively in Definition \ref{d.f.1}.

\begin{defn}\label{d.f.1}
Define
\begin{align}
\bar{V}_\kappa(\{B^j_\mu\})
&:= \exp\Bigg\{ \sum_{j=1}^3\left\langle B^j, -\bk\sum_{v=1}^{\underline{n}} \int_0^1 \ d\bar{s}\ \left\{m_\kappa(\lambda_\kappa^v(\bar{s}))^{-1} [\varrho_{\bar{s}}^{v,\prime}p_\kappa^{\vec{\varrho}^v_{\bar{s}}}]\right\}^j \right\rangle \nonumber\\
& \hspace{2cm} + \sum_{i=1}^3\left\langle \tilde{B}^{i}, -\bk\sum_{v=1}^{\underline{n}} \int_0^1\ d\bar{s}\ m_\kappa(\tilde{\lambda}_\kappa^v(\bar{s}))^{-1}\left[\varrho_{\bar{s}}^{v,\prime}p_\kappa^{\vec{\varrho}^v_{\bar{s}}} \right]\right\rangle
\Bigg\}, \label{e.b.3}\\
\bar{W}_\kappa(\{A^k_{\alpha\beta}\})
&  := \prod_{u=1}^{\on} \Tr_{\rho_u}  \mathcal{T} \exp\Bigg[\left\langle A^k_{0i}, -q\bk\kappa\int_0^1 ds\  \partial_0^{-1} p_\kappa^{\vec{y}^u_s} y^{u,\prime}_{k,s}\right\rangle \otimes \hat{E}^{0i}_s \nonumber\\
& \hspace{3cm}+ \left\langle A^k_{\tau(j)}, -q\bk\kappa\int_0^1 ds\  \partial_0^{-1}p_\kappa^{\vec{y}^u_s}\right\rangle y^{u,\prime}_{k,s}\otimes \hat{E}^{\tau(j)}_s
\Bigg]. \nonumber
\end{align}
\end{defn}

\begin{rem}
\begin{enumerate}
  \item We add in factors of $\bk$ as before.
  \item Also refer to Remark \ref{r.ax.2}.
\end{enumerate}
\end{rem}

Recall from Step \ref{d.cs.r.1d} in Definition \ref{d.cs.r.1}, we need to scale the surface integral with a factor of $\bk^2$. Thus we replace $\bar{F}_S^1(\{A^k_{\alpha\beta}\})$ with
\begin{align*}
\bar{F}_{\kappa,S}^1(\{A^k_{\alpha\beta}\}):=&
\bk^2\int_{I^2} \langle A_{0j}^i, p_\kappa^{\vec{\sigma}(\hat{t})} \rangle  J_{0i}(\hat{t})d\hat{t}\otimes \hat{E}^{0j} + \langle A_{\tau(j)}^i, p_\kappa^{\vec{\sigma}(\hat{t})} \rangle  J_{0i}(\hat{t}) d\hat{t}\otimes \hat{E}^{\tau(j)} \\
\equiv& \left\langle A_0^i, \int_{I^2}\bk^2 p_\kappa^{\vec{\sigma}(\hat{t})}J_{0i}(\hat{t})d\hat{t}\otimes \widetilde{E} \right\rangle
+ \left\langle A_{\tau(j)}^i, \int_{I^2}\bk^2 p_\kappa^{\vec{\sigma}(\hat{t})}J_{0i}(\hat{t})d\hat{t}\otimes \hat{E}^{\tau(j)} \right\rangle.
\end{align*}

Hence we approximate our path integral in Expression \ref{e.f.2} with
\beq
\frac{1}{\hat{Z}_{EH}}\int \bar{F}_{\kappa,S}^1(\{A^k_{\alpha\beta}\})\bar{V}_\kappa(\{B^j_\mu\}) \bar{W}_\kappa(\{A^k_{\alpha\beta}\})  e^{i\int_{\bR^4} A_0\cdot  B   -  A \cdot \tilde{B}}  D\Lambda \label{e.f.3}. \eeq This completes Steps \ref{d.cs.r.1c} and \ref{d.cs.r.1d} in Definition \ref{d.cs.r.1}.

\begin{lem}
Recall $\bar{E} = (1,1,1)$ and $\widetilde{E}$ was defined in Definition \ref{d.n.3}. Refer to Definition \ref{d.w.1} where $\mathcal{W}_\kappa^\pm(q; \oL, \ul^v)$ was defined. Apply Step \ref{d.cs.r.1e} in Definition \ref{d.cs.r.1}, the path integral in Expression \ref{e.f.3} is hence computed as
\begin{align}
&\Bigg\{-i \bk\sum_{v=1}^{\un}\left\langle \int_I\ d\bar{s} \left\{ \left[m_\kappa(\lambda_\kappa^v(\bar{s}) )\right]^{-1} (\varrho_{\bar{s}}^{v,\prime} p_\kappa^{\vec{\varrho}^{v}_{\bar{s}}}) \right\}^i , \int_{I^2} \bk^2p_\kappa^{\vec{\sigma}(\hat{t})} J_{0i}(\hat{t})d\hat{t}\otimes \widetilde{E}\right\rangle \nonumber \\
&+ i\bk\sum_{v=1}^{\un}\left\langle \int_I\ d\bar{s}\left\{ \left[m_\kappa(\tilde{\lambda}_\kappa^v(\bar{s}) )\right]^{-1} (\varrho_{\bar{s}}^{v,\prime} p_\kappa^{\vec{\varrho}^{v}_{\bar{s}}}) \right\}^k , \int_{I^2} \bk^2p_\kappa^{\vec{\sigma}(\hat{t})} J_{0k}(\hat{t})d\hat{t} \right\rangle \otimes \sum_{j=1}^3\hat{E}^{\tau(j)} \Bigg\}Z(\kappa, q; \chi(\oL, \uL)) .\label{e.z.2}
\end{align}
\end{lem}

\begin{proof}
From Equation (\ref{e.b.3}), according to the rules of Definition \ref{d.cs.r.1} of the Chern-Simons path integral, we now replace
\begin{align}
\begin{split}
(A_{01}^i,  A_{02}^i,  A_{03}^i)  \longmapsto& -i\bk\sum_{v=1}^{\un}\int_I\ d\bar{s}\left\{\left[m_\kappa(\lambda_\kappa^v(\bar{s}) )\right]^{-1}( \varrho_{\bar{s}}^{v,\prime} p_\kappa^{\vec{\varrho}^{v}_{\bar{s}}} ) \right\}^i, \\
A_{\tau(j)}^i  \longmapsto& i\bk\sum_{v=1}^{\un}\int_I\ d\bar{s}\left\{\left[m_\kappa(\tilde{\lambda}_\kappa^v(\bar{s}) )\right]^{-1} (\varrho_{\bar{s}}^{v,\prime} p_\kappa^{\vec{\varrho}^{v}_{\bar{s}}}) \right\}^i,
\end{split}
\label{e.x.5}
\end{align}
inside the functionals $\bar{W}_\kappa$ and $\bar{F}_{\kappa,S}^1$. Note that $\int_I\ ds\left\{\left[m_\kappa(\lambda_\kappa^v(s) )\right]^{-1} (\varrho_{\bar{s}}^{v,\prime} p_\kappa^{\vec{\varrho}^{v}_s}) \right\}^i$ is a 3-vector.

It can be shown that the substitution inside $\bar{W}_\kappa$ will yield $Z(\kappa, q; \chi(\oL, \uL))$, by using a similar argument in Lemma \ref{l.x.2}. Notice that it is now no longer necessary to have the time ordering operator. So, we will focus on making the substitution inside $\bar{F}_{\kappa,S}^1$. After the substitution, we obtain the term
\begin{align}
-&i \bk\left\langle \sum_{v=1}^{\un}\int_I\ d\bar{s}\left\{\left[m_\kappa(\lambda_\kappa^v(\bar{s}) )\right]^{-1}( \varrho_{\bar{s}}^{v,\prime} p_\kappa^{\vec{\varrho}^{v}_{\bar{s}}} ) \right\}^i , \int_{I^2} \bk^2p_\kappa^{\vec{\sigma}(\hat{t})} J_{0i}(\hat{t})d\hat{t}\otimes \widetilde{E}\right\rangle \nonumber \\
&+ i\bk\left\langle \sum_{v=1}^{\un}\int_I\ d\bar{s}\left\{\left[m_\kappa(\tilde{\lambda}_\kappa^v(\bar{s}) )\right]^{-1} (\varrho_{\bar{s}}^{v,\prime} p_\kappa^{\vec{\varrho}^{v}_{\bar{s}}}) \right\}^k , \int_{I^2} \bk^2p_\kappa^{\vec{\sigma}(\hat{t})} J_{0k}(\hat{t})d\hat{t} \right\rangle \otimes \sum_{j=1}^3\hat{E}^{\tau(j)}. \label{e.z.3}
\end{align}
This completes the proof.
\end{proof}

\begin{rem}
Refer to Remark \ref{r.p.2} for an explanation of Expression \ref{e.z.3}.
\end{rem}

Using Lemma \ref{l.k.1}, we see that to compute the limit of the first and second term in Expression \ref{e.z.3}, it suffices to compute the limit of
\begin{align}
\bar{A}_\kappa^\pm := \mp\frac{i\kappa^3}{32\pi\sqrt{4\pi}}\sum_{v=1}^{\un}\int_{I^3}\ \Bigg\{ &\left\langle  p_\kappa^{\vec{\sigma}(\hat{t})}, \kappa\left[\varrho^{v,\prime}_{3,s}\partial_2^{-1} - \varrho^{v,\prime}_{2,s}\partial_3^{-1}\right]p_\kappa^{\vec{\varrho}^{v}_s}\right\rangle J_{01}(\hat{t})\nonumber \\
+& \left\langle p_\kappa^{\vec{\sigma}(\hat{t})}, \kappa\left[\varrho^{v,\prime}_{1,s}\partial_3^{-1} - \varrho^{v,\prime}_{3,s}\partial_1^{-1}\right]p_\kappa^{\vec{\varrho}^{v}_s}\right\rangle J_{02}(\hat{t})\nonumber\\
+&\left\langle p_\kappa^{\vec{\sigma}(\hat{t})}, \kappa\left[\varrho^{v,\prime}_{2,s}\partial_1^{-1} - \varrho^{v,\prime}_{1,s}\partial_2^{-1}\right]p_\kappa^{\vec{\varrho}^{v}_s}\right\rangle J_{03}(\hat{t})
\Bigg\}\ ds d\hat{t} \otimes \mathcal{F}^{\pm}\label{e.c.4a}
\end{align}
respectively, as $\kappa$ goes to infinity.

From the proof of Lemma \ref{c.x.1}, we see that \beq \lim_{\kappa \rightarrow \infty }Z(\kappa, q; \chi(\oL, \uL)) = \lim_{\kappa \rightarrow \infty}\prod_{u=1}^{\on}\left[ \Tr_{\rho_u^+}\ \hat{\mathcal{W}}_\kappa^+(q; \ol^u, \uL) + \Tr_{\rho_u^-}\ \hat{\mathcal{W}}_\kappa^-(q; \ol^u, \uL) \right]. \label{e.x.4} \eeq

\begin{defn}\label{d.x.1}
We define Expression \ref{e.f.1a} as \beq \lim_{\kappa \rightarrow \infty}\left[ \bar{A}_\kappa^+ + \bar{A}_\kappa^-\right]\prod_{u=1}^{\on}\left[ \Tr_{\rho_u^+}\ \hat{\mathcal{W}}_\kappa^+(q; \ol^u, \uL) + \Tr_{\rho_u^-}\ \hat{\mathcal{W}}_\kappa^-(q; \ol^u, \uL) \right]. \nonumber \eeq Here, $A_\kappa^\pm$ were defined by Expression \ref{e.c.4a}.
\end{defn}

\subsection{$F_S^2$ Path Integral}\label{ss.fpi.2a}

We will now define the path integral \beq
\frac{1}{Z_{EH}}\int F_S^2(\{A^k_{\alpha\beta}\}) V( \{\ul^v\}_{v=1}^{\un})(\{B^i_\mu\})W(q; \{\ol^u, \rho_u\}_{u=1}^{\on})(\{A^k_{\alpha\beta}\}) e^{i\varsigma\int_{\bR^4} \partial_0 A_0 \cdot B \vec{\times} B   -  A \cdot \tilde{B}}  D\Lambda, \label{e.f.2a} \eeq which we will denote it as $\hat{F}_S^2[Z(q; \chi(\oL, \uL))]$.

Refer to Notation \ref{n.s.3}. Consider the term $\frac{1}{2}\sum_{i=1}^3\int_S \partial_i A_{\alpha \beta}^j \otimes dx_i \wedge dx_j  \otimes \hat{E}^{\alpha \beta}$. Observe that we can write this term as \beq \frac{1}{2}\sum_{i=1}^3\int_S \partial_i A_{\alpha \beta}^j \otimes dx_i \wedge dx_j \otimes \hat{E}^{\alpha \beta} = \frac{1}{2}\int_{I^2} \nabla \times (A_{\alpha\beta}^1, A_{\alpha\beta}^2, A_{\alpha\beta}^3)(\sigma(\hat{t})) \cdot (J_{23}, J_{31}, J_{12})(\hat{t})\ d\hat{t}\ \otimes \hat{E}^{\alpha\beta}. \nonumber \eeq

\begin{rem}
Note that $\nabla \times (A_{\alpha \beta}^1, A_{\alpha \beta}^2, A_{\alpha \beta}^3)$ means take the curl of the 3-vector $(A_{\alpha \beta}^1, A_{\alpha \beta}^2, A_{\alpha \beta}^3)$.
\end{rem}

\begin{lem}
Recall $\bar{E} = (1, 1, 1)$. Let $\delta^{\vec{x}}$ be the Dirac-delta function, i.e. for any function $f$, $\langle f, \delta^{\vec{x}} \rangle = f(\vec{x})$. Refer to the parametrizations $\vec{y}^u$ and $\vec{\varrho}^v$ defined in Notation \ref{n.n.3}. After doing a change of variables given in Notation \ref{n.c.1}, the path integral in Expression (\ref{e.f.2a}) is defined as \beq
\frac{1}{\hat{Z}_{EH}}\int \bar{F}_S^2(\{A^k_{\alpha\beta}\})\tilde{V}(\{B^j_\mu\}) \tilde{W}(\{A^k_{\alpha\beta}\})  e^{i\int_{\bR^4} A_0\cdot  B   -  A \cdot \tilde{B}}  D\Lambda \label{e.f.4} \eeq with \beq
\hat{Z}_{EH} = \int \exp\left[i\int_{\bR^4} A_0\cdot  B   -  A \cdot \tilde{B}\right]
D\Lambda, \nonumber \eeq after applying Steps \ref{d.cs.r.1a} and \ref{d.cs.r.1b} in Definition \ref{d.cs.r.1}.

Here, $\tilde{V}(\{B^j_\mu\})$ and $\tilde{W}(\{A^k_{\alpha\beta}\})$ were defined in Equations (\ref{e.w.2}) and (\ref{e.w.3}) respectively and
\begin{align*}
\bar{F}_S^2&(\{A^k_{\alpha\beta}\}) \\
=& \int_{I^2}d\hat{t}\left\langle A_{0j},  \partial_0^{-1}\nabla \times (\delta^{\vec{\sigma}(\hat{t})}J_\sigma(\hat{t})) \right\rangle \otimes \hat{E}^{0j} + \int_{I^2}d\hat{t}\left\langle A_{\tau(j)}, \partial_0^{-1} \nabla \times(\delta^{\vec{\sigma}(\hat{t})}J_\sigma(\hat{t})) \right\rangle  \otimes \hat{E}^{\tau(j)}.
\end{align*}
\end{lem}

\begin{proof}
By writing $A_{\alpha\beta} = (A_{\alpha\beta}^1, A_{\alpha\beta}^2, A_{\alpha\beta}^3)$ and $J_\sigma = (J_{23}, J_{31}, J_{12})$, we see that
\begin{align*}
(\nabla \times A_{\alpha\beta}(\vec{x})) \cdot J_\sigma(\hat{t}) =& \left\langle \nabla \times A_{\alpha\beta}^{i}, \delta^{\vec{x}} J_{\tau(i)}(\hat{t}) \right\rangle
\equiv \left\langle \nabla \times A_{\alpha\beta},  \delta^{\vec{x}} J_\sigma(\hat{t})\right\rangle \\
=&  \left\langle A_{\alpha\beta},  -\nabla \times\delta^{\vec{x}} J_\sigma(\hat{t})\right\rangle.
\end{align*}
In the last equality, we make use of the fact that $\nabla \times$ is an skew-symmetric operator.

Note that $A_{\alpha\beta} \mapsto \partial_0^{-1}A_{\alpha\beta}$ and $\langle \partial_0^{-1}f, g \rangle = -\langle f, \partial_0^{-1}g \rangle$. The rest of the proof is similar to the proof in Lemma \ref{l.f.1}, hence omitted.
\end{proof}

As explained in Subsection \ref{ss.fpi.1a}, Expression \ref{e.f.4} is not a Chern-Simons integral, therefore we need to make some approximations, i.e. approximate $m(B)$ and $m(B_i)$ with $m_\kappa(B)$ and $m_\kappa(B_i)$, and $\delta^{\vec{x}}$ with $p_\kappa^{\vec{x}}$. After making these approximations, it is still not a Chern-Simons integral. So, we need to make use of the substitutions given by Equations (\ref{e.x.3a}) and (\ref{e.x.3b}). Finally we need to add in factors of $\bk$.

Thus we approximate our path integral in Expression \ref{e.f.4} with
\beq
\frac{1}{\hat{Z}_{EH}}\int \bar{F}_{\kappa,S}^2(\{A^k_{\alpha\beta}\}) \bar{V}_\kappa(\{B^j_\mu\}) \bar{W}_\kappa(\{A^k_{\alpha\beta}\})  e^{i\int_{\bR^4} A_0\cdot  B   -  A \cdot \tilde{B}}  D\Lambda \label{e.f.5}. \eeq

Here, $\bar{V}_\kappa(\{B^j_\mu\})$ and $\bar{W}_\kappa(\{A^k_{\alpha\beta}\})$ were defined in Definition \ref{d.f.1} and
\begin{align}
\bar{F}_{\kappa,S}^2&(\{A^k_{\alpha\beta}\}) \nonumber\\
:=&  \bk^2\int_{I^2}d\hat{t}\left\langle A_{0j},   \partial_0^{-1}\nabla \times (p_\kappa^{\vec{\sigma}(\hat{t})}J_\sigma(\hat{t})) \right\rangle \otimes \hat{E}^{0j} + \bk^2\int_{I^2}d\hat{t}\left\langle A_{\tau(j)}, \partial_0^{-1} \nabla \times(p_\kappa^{\vec{\sigma}(\hat{t})}J_\sigma(\hat{t})) \right\rangle  \otimes \hat{E}^{\tau(j)} \nonumber\\
\equiv&\left\langle A_{0j}, \bk^2\int_{I^2}d\hat{t}\  \partial_0^{-1}\nabla \times (p_\kappa^{\vec{\sigma}(\hat{t})}J_\sigma(\hat{t})) \right\rangle \otimes \hat{E}^{0j} + \left\langle A_{\tau(j)}, \bk^2\int_{I^2}d\hat{t}\ \partial_0^{-1} \nabla \times(p_\kappa^{\vec{\sigma}(\hat{t})}J_\sigma(\hat{t})) \right\rangle  \otimes \hat{E}^{\tau(j)}.\label{e.f.4a}
\end{align}

This completes Steps \ref{d.cs.r.1c} and \ref{d.cs.r.1d} in Definition \ref{d.cs.r.1}.

\begin{notation}\label{n.j.1}
Write $\hat{J}(\hat{t}) = (\hat{J}^1(\hat{t}), \hat{J}^2(\hat{t}), \hat{J}^3(\hat{t}))$, whereby \beq \hat{J}^i = J_{\tau(i)}\otimes(\hat{E}^{01}, \hat{E}^{02}, \hat{E}^{03}). \nonumber \eeq
\end{notation}

\begin{lem}
Refer to Definition \ref{d.w.1} where $\mathcal{W}_\kappa^\pm(q; \oL, \ul^v)$ was defined. Apply Step \ref{d.cs.r.1e} in Definition \ref{d.cs.r.1}, the path integral in Expression \ref{e.f.5} is hence computed as $\bar{A}_\kappa Z(\kappa, q; \chi(\oL, \uL))$, whereby
\begin{align}
\bar{A}_\kappa& := -i \bk\sum_{i=1}^3\left\langle \sum_{v=1}^{\un}\int_I\ d\bar{s}\ \left\{\left[m_\kappa(\lambda_\kappa^v(\bar{s}) )\right]^{-1}( \varrho_{\bar{s}}^{v,\prime} p_\kappa^{\vec{\varrho}^{v}_{\bar{s}}} ) \right\}^i , \left\{\bk^2 \int_{I^2} \partial_0^{-1}\nabla \vec{\times}(p_\kappa^{\vec{\sigma}(\hat{t})} \otimes \hat{J}(\hat{t})) \right\}^id\hat{t} \right\rangle  \nonumber\\
&+ i\bk\sum_{k=1}^3\left\langle \sum_{v=1}^{\un}\int_I\ d\bar{s}\ \left\{\left[m_\kappa(\tilde{\lambda}_\kappa^v(\bar{s}) )\right]^{-1} (\varrho_{\bar{s}}^{v,\prime} p_\kappa^{\vec{\varrho}^{v}_{\bar{s}}}) \right\}^k , \left\{ \bk^2\int_{I^2} \partial_0^{-1}\nabla \times (p_\kappa^{\vec{\sigma}(\hat{t})} J_\sigma(\hat{t})) \right\}^k d\hat{t} \right\rangle \otimes \mathcal{F}^-. \label{e.c.3}
\end{align}
Here, both $\lambda_\kappa^v$ and $\tilde{\lambda}_\kappa^v$ were defined in Equations (\ref{e.l.1}) and (\ref{e.l.2}) respectively. And note that $Z(\kappa, q; \chi(\oL, \uL))$ was defined in Equation (\ref{e.z.1}).
\end{lem}

\begin{proof}
According to the rules of Definition \ref{d.cs.r.1} of the Chern-Simons integral and from Equation (\ref{e.b.3}), we use the Substitution given by Equations (\ref{e.x.5})
inside the functionals $\bar{W}_\kappa$ and $\bar{F}_S^2$ given by Equations (\ref{e.b.3}) and (\ref{e.f.4a}) respectively.

One can show that the substitution inside $\bar{W}_\kappa$ will yield $Z(\kappa, q; \chi(\oL, \uL))$. Now refer to Notation \ref{n.j.1}. Making the above substitution inside $\bar{F}_{\kappa,S}^2$ will yield the term $\bar{A}_\kappa$.
\end{proof}

\begin{rem}
\begin{enumerate}
  \item See Remark \ref{r.p.2} for an explanation of RHS of Equation (\ref{e.c.3}).
  \item The term \beq \partial_0^{-1}\nabla \vec{\times}(p_\kappa^{\vec{\sigma}(\hat{t})} \otimes \hat{J}(\hat{t})) \equiv \partial_0^{-1}\nabla \times \left( p_\kappa^{\vec{\sigma}(\hat{t})}\otimes J_\sigma(\hat{t}) \right)\otimes (\hat{E}^{01}, \hat{E}^{02}, \hat{E}^{03})
       \nonumber \eeq can be written as $(a^1, a^2, a^3)\otimes (\hat{E}^{01}, \hat{E}^{02}, \hat{E}^{03})$, with each $a^i$ given by for $(i,j,k) \in C_3$,
      \begin{align*}
      a^i =& \partial_0^{-1}(\partial_j p_\kappa^{\vec{\sigma}(\hat{t})})J_{\tau(k)}(\hat{t}) - \partial_0^{-1}(\partial_k p_\kappa^{\vec{\sigma}(\hat{t})})J_{\tau(j)}(\hat{t}).
      \end{align*}
      Note that $\partial_0^{-1}(\partial_i p_\kappa^{\vec{\sigma}(\hat{t})})$ is defined using Equation (\ref{e.d.1}).
\end{enumerate}

\end{rem}

Now $\langle \nabla \times f, g \rangle = \langle f, -\nabla \times g \rangle$ and from Remark \ref{r.m.1},
\beq -\nabla \times m_\kappa(0)^{-1} = -\kappa. \nonumber \eeq Hence $m_\kappa(0)^{-1} = \kappa (\nabla \times)^{-1}$.
Since $\kappa\lambda_\kappa \rightarrow 0$ and $\kappa\tilde{\lambda}_\kappa \rightarrow 0$ as $\kappa \rightarrow \infty$ by Lemma \ref{l.k.1}, therefore to compute the limit of Expression \ref{e.c.3} as $\kappa$ goes to infinity is equivalent to compute the limit of
\begin{align*}
\bar{B}_\kappa := \frac{i\kappa^3}{32\pi\sqrt{4\pi}}&\int_{I^3}\  \kappa\left\langle \partial_0^{-1} p_\kappa^{\vec{\sigma}(\hat{t})}, p_\kappa^{\vec{\varrho}^{v}_s}\right\rangle \varrho_s^{v,\prime}\cdot J_\sigma(\hat{t}) \otimes \sum_{i=1}^3\hat{E}^{0i}\ ds d\hat{t}, \\
-\frac{i\kappa^3}{32\pi\sqrt{4\pi}}&\int_{I^3}\ \kappa\left\langle \partial_0^{-1} p_\kappa^{\vec{\sigma}(\hat{t})}, p_\kappa^{\vec{\varrho}^{v}_s}\right\rangle \varrho_s^{v,\prime}\cdot J_\sigma(\hat{t}) \otimes \sum_{j=1}^3\hat{E}^{\tau(j)}\ ds d\hat{t},
\end{align*}
as $\kappa$ goes to infinity.

\begin{defn}\label{d.x.2}
Refer to Equation \ref{e.x.4}. We define Expression \ref{e.f.2a} by
\beq \lim_{\kappa \rightarrow \infty}\bar{B}_\kappa \prod_{u=1}^{\on}\left[ \Tr_{\rho_u^+}\ \hat{\mathcal{W}}_\kappa^+(q; \ol^u, \uL) + \Tr_{\rho_u^-}\ \hat{\mathcal{W}}_\kappa^-(q; \ol^u, \uL) \right], \nonumber \eeq whereby
$\bar{B}_\kappa$ was defined in the preceding paragraph.
\end{defn}

\subsection{$F_S^3$ Path Integral}\label{ss.fpi.3a}

Refer to Notations \ref{n.n.5} and \ref{n.s.3}. We have finally come to the last term \beq F_S^3(\{A^k_{\alpha\beta}\}) = \frac{1}{2}\int_S A_{\alpha\beta}^iA_{\gamma\mu}^j\otimes dx_i \wedge dx_j \otimes[\hat{E}^{\alpha\beta}, \hat{E}^{\gamma\mu}]. \nonumber \eeq

We will now define the path integral \beq
\frac{1}{Z_{EH}}\int F_S^3(\{A^k_{\alpha\beta}\})  V( \{\ul^v\}_{v=1}^{\un})(\{B^i_\mu\})W(q; \{\ol^u, \rho_u\}_{u=1}^{\on})(\{A^k_{\alpha\beta}\}) e^{i\varsigma\int_{\bR^4} \partial_0 A_0 \cdot B \vec{\times} B   -  A \cdot \tilde{B}}  D\Lambda, \label{e.f.6a} \eeq which we will denote it as $\hat{F}_S^3[Z(q; \chi(\oL, \uL))]$.

\begin{lem}
Recall $\bar{E} = (1, 1, 1)$. Let $\delta^{\vec{x}}$ be the Dirac-delta function, i.e. for any function $f$, $\langle f, \delta^{\vec{x}} \rangle = f(\vec{x})$. Refer to the parametrizations $\vec{y}^u$ and $\vec{\varrho}^v$ defined in Notation \ref{n.n.3}. After doing a change of variables given in Notation \ref{n.c.1}, the path integral in Expression \ref{e.f.6a} is defined as \beq
\frac{1}{\hat{Z}_{EH}}\int \bar{F}_S^3(\{A^k_{\alpha\beta}\})\tilde{V}(\{B^j_\mu\}) \tilde{W}(\{A^k_{\alpha\beta}\})  e^{i\int_{\bR^4} A_0\cdot  B   -  A \cdot \tilde{B}}  D\Lambda \label{e.f.7} \eeq with \beq
\hat{Z}_{EH} = \int \exp\left[i\int_{\bR^4} A_0\cdot  B   -  A \cdot \tilde{B}\right]
D\Lambda, \nonumber \eeq after applying Steps \ref{d.cs.r.1a} and \ref{d.cs.r.1b} in Definition \ref{d.cs.r.1}.

Here, $\tilde{V}(\{B^j_\mu\})$ and $\tilde{W}(\{A^k_{\alpha\beta}\})$ were defined in Equations (\ref{e.w.2}) and (\ref{e.w.3}) respectively and $\bar{F}_S^3(\{A^k_{\alpha\beta}\})$ is defined by Equation (\ref{e.x.6}).
See Equations (\ref{e.f.9a}) and (\ref{e.f.9b}).
\end{lem}

\begin{proof}
Now we can write
\begin{align*}
A_{\alpha\beta}^i(\vec{x})A_{\gamma\mu}^j(\vec{x}) \equiv& \left\langle A_{\alpha\beta}^i, \delta^{\vec{x}} \right\rangle  \left\langle A_{\gamma\mu}^j, \delta^{\vec{x}} \right\rangle
=: \left\langle A_{\alpha\beta}^i \otimes A_{\gamma\mu}^j, \delta^{\vec{x}} \otimes \delta^{\vec{x}}\right\rangle,
\end{align*}
using the tensor inner product.

Refer to Notations \ref{n.n.5}, \ref{n.s.2} and \ref{n.s.3}. Then we can write
\begin{align*}
F_S^3&(\{A^k_{\alpha\beta}\}) \\
=& 8\sum_{(i,j) \in \Upsilon}\sum_{(\bar{i},\bar{j}) \in \Upsilon}\int_{I^2} d\hat{t}\ A_{0i}^{\bar{i}}(\vec{\sigma}(\hat{t}))A_{0j}^{\bar{j}}(\vec{\sigma}(\hat{t})) J_{\bar{i}\bar{j}}(\hat{t}) \otimes[\hat{E}^{0i}, \hat{E}^{0j}] \\
&+ 8\sum_{(i,j) \in \Upsilon}\sum_{(\bar{i},\bar{j}) \in \Upsilon}\int_{I^2} d\hat{t}\ A_{\tau(i)}^{\bar{i}}(\vec{\sigma}(\hat{t}))A_{\tau(j)}^{\bar{j}}(\vec{\sigma}(\hat{t})) J_{\bar{i}\bar{j}}(\hat{t}) \otimes[\hat{E}^{\tau(i)}, \hat{E}^{\tau(j)}] \\
=& 8\sum_{(i,j) \in \Upsilon}\sum_{(\bar{i},\bar{j}) \in \Upsilon}\int_{I^2} d\hat{t} \left\langle A_{0i}^{\bar{i}}\otimes A_{0j}^{\bar{j}}, \delta^{\vec{\sigma}(\hat{t})} \otimes \delta^{\vec{\sigma}(\hat{t})} \right\rangle\ J_{\bar{i}\bar{j}}(\hat{t}) \otimes[\hat{E}^{0i}, \hat{E}^{0j}] \\
&+ 8\sum_{(i,j) \in \Upsilon}\sum_{(\bar{i},\bar{j}) \in \Upsilon}\int_{I^2} d\hat{t} \left\langle A_{\tau(i)}^{\bar{i}}\otimes A_{\tau(j)}^{\bar{j}}, \delta^{\vec{\sigma}(\hat{t})} \otimes \delta^{\vec{\sigma}(\hat{t})} \right\rangle\ J_{\bar{i}\bar{j}}(\hat{t}) \otimes[\hat{E}^{\tau(i)}, \hat{E}^{\tau(j)}].
\end{align*}

Note that $A_{\alpha\beta} \mapsto \partial_0^{-1}A_{\alpha\beta}$ and $\langle \partial_0^{-1}f, g \rangle = -\langle f, \partial_0^{-1}g \rangle$. So we define
\begin{align}
\bar{F}_S^3&(\{A^k_{\alpha\beta}\}) \nonumber \\
:=& 8\sum_{(i,j) \in \Upsilon}\sum_{(\bar{i},\bar{j}) \in \Upsilon}\int_{I^2} d\hat{t} \left\langle A_{0i}^{\bar{i}}\otimes A_{0j}^{\bar{j}}, \partial_0^{-1}\delta^{\vec{\sigma}(\hat{t})} \otimes \partial_0^{-1}\delta^{\vec{\sigma}(\hat{t})} \right\rangle\ J_{\bar{i}\bar{j}}(\hat{t}) \otimes[\hat{E}^{0i}, \hat{E}^{0j}] \nonumber \\
&+ 8\sum_{(i,j) \in \Upsilon}\sum_{(\bar{i},\bar{j}) \in \Upsilon}\int_{I^2} d\hat{t} \left\langle A_{0i}^{\bar{i}}\otimes A_{0j}^{\bar{j}}, \partial_0^{-1}\delta^{\vec{\sigma}(\hat{t})} \otimes \partial_0^{-1}\delta^{\vec{\sigma}(\hat{t})} \right\rangle\ J_{\bar{i}\bar{j}}(\hat{t}) \otimes[\hat{E}^{\tau(i)}, \hat{E}^{\tau(j)}] \nonumber\\
\equiv& 8\sum_{(i,j) \in \Upsilon}\int_{I^2}d\hat{t}\left\langle A_{0i}\otimes A_{0j},  \partial_0^{-1}\delta^{\vec{\sigma}(t)} \otimes \partial_0^{-1}\delta^{\vec{\sigma}(\bar{t})} J_\sigma(\hat{t}) \right\rangle \otimes [\hat{E}^{0i}, \hat{E}^{0j}] \nonumber \\
&+ 8\sum_{(i,j) \in \Upsilon}\int_{I^2}d\hat{t}\left\langle A_{\tau(i)}\otimes A_{\tau(j)}, \partial_0^{-1}\delta^{\vec{\sigma}(t)} \otimes \partial_0^{-1}\delta^{\vec{\sigma}(\hat{t})}J_\sigma(\hat{t}) \right\rangle  \otimes [\hat{E}^{\tau(i)}, \hat{E}^{\tau(j)}]. \label{e.x.6}
\end{align}
Here,
\begin{align}
A_{0i}\otimes A_{0j} \equiv& \left(A_{0i}^2\otimes A_{0j}^3, A_{0i}^3\otimes A_{0j}^1, A_{0i}^1\otimes A_{0j}^2 \right), \label{e.f.9a}\\
A_{\tau(i)}\otimes A_{\tau(j)} \equiv& \left(A_{\tau(i)}^2\otimes A_{\tau(j)}^3, A_{\tau(i)}^3\otimes A_{\tau(j)}^1, A_{\tau(i)}^1\otimes A_{\tau(j)}^2 \right), \label{e.f.9b}
\end{align}
are respectively 3-vectors, whose components are in $C^\infty(\bR) \otimes C^\infty(\bR)$. The rest of the proof is similar to the proof in Lemma \ref{l.f.1}, hence omitted.

\end{proof}

As explained in Subsection \ref{ss.fpi.1a}, Expression \ref{e.f.7} is not a Chern-Simons integral. To make it into a Chern-Simons integral, we make use of the substitutions given by Equations (\ref{e.x.3a}) and (\ref{e.x.3b}). We also approximate $\delta^{\vec{x}}$ with $p_\kappa^{\vec{x}}$. And we need to add in factors of $\bk$.

Thus we approximate our path integral in Expression \ref{e.f.4} with
\beq
\frac{1}{\hat{Z}_{EH}}\int \bar{F}_{\kappa,S}^3(\{A^k_{\alpha\beta}\}) \bar{V}_\kappa(\{B^j_\mu\}) \bar{W}_\kappa(\{A^k_{\alpha\beta}\})  e^{i\int_{\bR^4} A_0\cdot  B   -  A \cdot \tilde{B}}  D\Lambda \label{e.f.8}. \eeq

Here, $\bar{V}_\kappa(\{B^j_\mu\})$ and $\bar{W}_\kappa(\{A^k_{\alpha\beta}\})$ were defined in Definition \ref{d.f.1} and
\begin{align}
\bar{F}_{\kappa,S}^3&(\{A^k_{\alpha\beta}\}) \nonumber\\
:=&  8\bk^2\sum_{(i,j) \in \Upsilon}\int_{I^2}d\hat{t}\left\langle A_{0i}\otimes A_{0j},  \partial_0^{-1}p_\kappa^{\vec{\sigma}(t)} \otimes \partial_0^{-1}p_\kappa^{\vec{\sigma}(\bar{t})} J_\sigma(\hat{t}) \right\rangle \otimes [\hat{E}^{0i}, \hat{E}^{0j}] \nonumber \\
&+ 8\bk^2\sum_{(i,j) \in \Upsilon}\int_{I^2}d\hat{t}\left\langle A_{\tau(i)}\otimes A_{\tau(j)}, \partial_0^{-1}p_\kappa^{\vec{\sigma}(t)} \otimes \partial_0^{-1}p_\kappa^{\vec{\sigma}(\hat{t})}J_\sigma(\hat{t}) \right\rangle  \otimes [\hat{E}^{\tau(i)}, \hat{E}^{\tau(j)}] \nonumber \\
\equiv& \sum_{(i,j) \in \Upsilon}\left\langle A_{0i}\otimes A_{0j},\ 8\bk^2\int_{I^2}d\hat{t}\ \partial_0^{-1}p_\kappa^{\vec{\sigma}(t)} \otimes \partial_0^{-1}p_\kappa^{\vec{\sigma}(\bar{t})} J_\sigma(\hat{t}) \right\rangle \otimes [\hat{E}^{0i}, \hat{E}^{0j}] \nonumber \\
&+ \sum_{(i,j) \in \Upsilon}\left\langle A_{\tau(i)}\otimes A_{\tau(j)},\ 8\bk^2\int_{I^2}d\hat{t}\ \partial_0^{-1}p_\kappa^{\vec{\sigma}(t)} \otimes \partial_0^{-1}p_\kappa^{\vec{\sigma}(\hat{t})}J_\sigma(\hat{t}) \right\rangle  \otimes [\hat{E}^{\tau(i)}, \hat{E}^{\tau(j)}].\label{e.f.6b}
\end{align}

This completes Steps \ref{d.cs.r.1c} and \ref{d.cs.r.1d} in Definition \ref{d.cs.r.1}.

\begin{notation}\label{n.e.1}
Refer to Notation \ref{n.s.2}. Observe that \beq \sum_{(i,j) \in \Upsilon}[\hat{E}^{0i}, \hat{E}^{0j}] = (\mathcal{E}^+,0) = \mathcal{F}^+, \ \ \sum_{(i,j) \in \Upsilon}[\hat{E}^{\tau(i)}, \hat{E}^{\tau(j)}] = (0,\mathcal{E}^-) = \mathcal{F}^-.
\nonumber \eeq
\end{notation}

\begin{lem}
Refer to Definition \ref{d.w.1} where $\mathcal{W}_\kappa^\pm(q; \oL, \ul^v)$ was defined. Apply Step \ref{d.cs.r.1e} in Definition \ref{d.cs.r.1}, the path integral in Expression \ref{e.f.8} is hence computed as $(A_\kappa + B_\kappa)Z(\kappa, q; \chi(\oL, \uL))$, whereby $A_\kappa$ is given by Expression \ref{e.b.6a} and $B_\kappa$ is given by Expression \ref{e.b.6b}.
\end{lem}

\begin{proof}
According to the rules of Definition \ref{d.cs.r.1} of the Chern-Simons integral and from Equation (\ref{e.b.3}), we apply the substitution given in Equation (\ref{e.x.5})
inside the functionals $\bar{W}_\kappa$ and $\bar{F}_S^3$ given by Equations (\ref{e.b.3}) and (\ref{e.f.6b}) respectively. Here, both $\lambda_\kappa^v$ and $\tilde{\lambda}_\kappa^v$ were defined by Equations (\ref{e.l.1}) and (\ref{e.l.2}) respectively.

The 3-vector \beq \sum_{v=1}^{\un}\int_I\ d\bar{s}\ \left\{\left[m_\kappa(\lambda_\kappa^v(\bar{s}) )\right]^{-1}( \varrho_{\bar{s}}^{v,\prime} p_\kappa^{\vec{\varrho}^{v}_{\bar{s}}} ) \right\}^i \equiv (C_1^i, C_2^i, C_3^i), \nonumber \eeq will have each component \beq C_j^i := \sum_{v=1}^{\un}\int_I\ d\bar{s}\ \left\{\left[m_\kappa(\lambda_\kappa^v(\bar{s}) )\right]^{-1}( \varrho_{\bar{s}}^{v,\prime} p_\kappa^{\vec{\varrho}^{v}_{\bar{s}}} ) \right\}^i_j. \nonumber \eeq

It was shown that the substitution inside $\bar{W}_\kappa$ will yield $Z(\kappa, q; \chi(\oL, \uL))$. Thus after making the above substitution into $\bar{F}_{\kappa,S}^3(\{A^k_{\alpha\beta}\})$, should give us the sum of 2 terms, $A_\kappa + B_\kappa$, the first term $A_\kappa$ being (after some simplification)
\begin{align}
8\left(i\bk\right)^2\sum_{(\bar{i},\bar{j}) \in \Upsilon}\sum_{(i,j) \in \Upsilon}\int_{I^2}d\hat{t}&\ J_{ij}(\hat{t})
\Bigg\{ \left\langle \bk\sum_{v=1}^{\un}\int_I\ ds\ \left\{ \left[m_\kappa(\lambda_\kappa^v(s) )\right]^{-1} \varrho_s^{v,\prime} p_\kappa^{\vec{\varrho}^{v}_s} \right\}_{\bar{i}}^i, \partial_0^{-1} p_\kappa^{\vec{\sigma}(\hat{t})} \right\rangle \nonumber\\
\times& \left\langle \bk\sum_{\bar{v}=1}^{\un}\int_I\ d\bar{s}\ \left\{ \left[m_\kappa(\lambda_\kappa^{\bar{v}}(\bar{s}) )\right]^{-1} \varrho_{\bar{s}}^{\bar{v},\prime} p_\kappa^{\vec{\varrho}^{\bar{v}}_{\bar{s}}} \right\}_{\bar{j}}^j, \partial_0^{-1} p_\kappa^{\vec{\sigma}(\hat{t})} \right\rangle
\Bigg\}  \otimes [\hat{E}^{0\bar{i}}, \hat{E}^{0\bar{j}}], \label{e.b.6a}
\end{align}
and the second term $B_\kappa$ being
\begin{align}
8\left(i\bk\right)^2\sum_{(i,j) \in \Upsilon}\int_{I^2} \Bigg\{& \left\langle \bk\sum_{v=1}^{\un}\int_I\ ds\ \left\{ \left[m_\kappa(\tilde{\lambda}_\kappa^v(s) )\right]^{-1} \varrho_s^{v,\prime} p_\kappa^{\vec{\varrho}^{v}_s} \right\}^i, \partial_0^{-1} p_\kappa^{\vec{\sigma}(\hat{t})} \right\rangle \nonumber\\
\times& \left\langle \bk\sum_{\bar{v}=1}^{\un}\int_I\ d\bar{s}\ \left\{ \left[m_\kappa(\tilde{\lambda}_\kappa^{\bar{v}}(\bar{s}) )\right]^{-1} \varrho_{\bar{s}}^{\bar{v},\prime} p_\kappa^{\vec{\varrho}^{\bar{v}}_{\bar{s}}} \right\}^j,\partial_0^{-1}  p_\kappa^{\vec{\sigma}(\hat{t})} \right\rangle \Bigg\}J_{ij}(\hat{t})\ d\hat{t}\otimes (0,\mathcal{E}^-). \label{e.b.6b}
\end{align}

\end{proof}

\begin{rem}
Refer to Remark \ref{r.p.2} for a detailed explanation of Expressions \ref{e.b.6a} and \ref{e.b.6b}.
\end{rem}

Refer to Notation \ref{n.s.2}. Since $\kappa\lambda_\kappa^v \rightarrow 0$ and $\kappa\tilde{\lambda}_\kappa^v \rightarrow 0$ as $\kappa$ goes to infinity in Lemma \ref{l.k.1}, to compute the limit for the terms $A_\kappa$ and $B_\kappa$ as $\kappa$ goes to infinity, it suffices to compute the limit of
\begin{align*}
\bar{C}_\kappa^\pm := \frac{-\kappa^4}{32\pi^2}\sum_{(i,j,k)\in C_3}\sum_{v,\bar{v}=1}^{\un}\int_{I^4}\ \Bigg\{ &\left\langle  \partial_0^{-1}p_\kappa^{\vec{\sigma}(\hat{t})}, \kappa\left[\varrho^{v,\prime}_{k,s}\partial_j^{-1} - \varrho^{v,\prime}_{j,s}\partial_k^{-1}\right]p_\kappa^{\vec{\varrho}^{v}_s}\right\rangle \\
\times& \left\langle  \partial_0^{-1}p_\kappa^{\vec{\sigma}(\hat{t})}, \kappa\left[\varrho^{\bar{v},\prime}_{i,\bar{s}}\partial_k^{-1} - \varrho^{\bar{v},\prime}_{k,\bar{s}}\partial_i^{-1}\right]p_\kappa^{\vec{\varrho}^{\bar{v}}_{\bar{s}}}
\right\rangle
\Bigg\}\ J_{ij}(\hat{t})\ d\hat{s}d\hat{t} \otimes \mathcal{F}^\pm,
\end{align*}
respectively as $\kappa$ goes to infinity. Note that $\hat{s} = (s, \bar{s})$ and $\hat{t} = (t, \bar{t})$.

\begin{defn}\label{d.x.3}
Refer to Equation \ref{e.x.4}. We define Expression \ref{e.f.6a} by
\beq \lim_{\kappa \rightarrow \infty}\left[\bar{C}_\kappa^+ + \bar{C}_\kappa^-\right]\prod_{u=1}^{\on}\left[ \Tr_{\rho_u^+}\ \hat{\mathcal{W}}_\kappa^+(q; \ol^u, \uL) + \Tr_{\rho_u^-}\ \hat{\mathcal{W}}_\kappa^-(q; \ol^u, \uL) \right], \nonumber \eeq whereby
$\bar{C}_\kappa^\pm$ were defined in the preceding paragraph.
\end{defn}

Putting all together, we can now make the final definition.

\begin{defn}(Curvature Path Integral)\label{d.cpi}\\
Refer to Definitions \ref{d.x.1}, \ref{d.x.2} and \ref{d.x.3}. We define the curvature path integral given by Expression \ref{e.f.1} as \beq \lim_{\kappa \rightarrow \infty}\left[(\bar{A}_\kappa^+ + \bar{A}_\kappa^-) + \bar{B}_\kappa + (\bar{C}_\kappa^+ + \bar{C}_\kappa^-)\right]\prod_{u=1}^{\on}\left[ \Tr_{\rho_u^+}\ \hat{\mathcal{W}}_\kappa^+(q; \ol^u, \uL) + \Tr_{\rho_u^-}\ \hat{\mathcal{W}}_\kappa^-(q; \ol^u, \uL) \right]. \nonumber \eeq
\end{defn}

\appendix

\section{Important Lemmas}

\begin{lem}\label{l.l.5}
Refer to Notation \ref{n.n.1}.
\begin{enumerate}
  \item\label{i.l.6} Let $s, t \in \bR$, $s \neq t$. Then \beq \lim_{\kappa \rightarrow \infty}\frac{\kappa}{\sqrt{2\pi}}\langle q_\kappa^{s}, \partial_0^{-1}q_\kappa^t \rangle = \left\{
      \begin{array}{ll}
      1, & \hbox{$s > t$;} \\
      -1, & \hbox{$s < t$.}
      \end{array}
      \right. \nonumber \eeq
\item\label{i.l.5} Let $\ba, \mathbf{b} \in \bR^2$. Then, \beq \langle \pk{\bR^2}^{\ba}, \pk{\bR^2}^{\mathbf{b}} \rangle = e^{-\kappa^2|\ba - \mathbf{b}|^2/8}. \nonumber \eeq
\item\label{i.l.7} Let $a, b \in \bR^3$. Then, \beq \langle \pk{\bR^2}^{a}, \pk{\bR^2}^{b} \rangle = e^{-\kappa^2|a - b|^2/8}. \nonumber \eeq
\end{enumerate}
\end{lem}

\begin{proof}
The proof of Item (\ref{i.l.6}) can be found in Lemma 4.5 in \cite{CS-Lim01}. We reproduce it here for the convenience of the reader.

By definition of $\partial_0^{-1}$, we have
\begin{align*}
&\frac{\kappa}{\sqrt{2\pi}}\langle q_\kappa^{s}, \partial_0^{-1}q_\kappa^t \rangle \\
&= \frac{\kappa}{\sqrt{2\pi}}\int_\bR \frac{\sqrt \kappa}{(2\pi)^{1/4}}e^{-\frac{\kappa^2}{4}(x - s)^2} \cdot \frac{1}{2}\Bigg[ \int_{-\infty}^x \frac{\sqrt\kappa}{(2\pi)^{1/4}}e^{-\frac{\kappa^2}{4}(y-t)^2}\ dy
- \int_{x}^\infty \frac{\sqrt\kappa}{(2\pi)^{1/4}}e^{-\kappa^2(y-t)^2/4}\ dy \Bigg]\ dx \\
&= \int_\bR \frac{\kappa}{(4\pi)^{1/2}}e^{-\kappa^2 r^2/4} \cdot \Bigg[ \int_{-\infty}^{r + s-t} \frac{\kappa}{(4\pi)^{1/2}}e^{-\kappa^2\tau^2/4}\ d\tau
- \int_{r+s-t}^\infty \frac{\kappa}{(4\pi)^{1/2}}e^{-\kappa^2\tau^2/4}\ d\tau \Bigg]\ dr \\
&\longrightarrow
\left\{
  \begin{array}{ll}
    1, & \hbox{$s > t$;} \\
    -1, & \hbox{$s < t$,}
  \end{array}
\right.
\end{align*}
as $\kappa \rightarrow \infty$.

Let $\mathbf{x} = (x_+, x_-)$. Item (\ref{i.l.5}) follows from direct integration,
\begin{align*}
\int_{\bR^2}\frac{\kappa}{\sqrt{2\pi}} e^{-\kappa^2|\mathbf{x} - \ba|^2/4}\cdot
\frac{\kappa}{\sqrt{2\pi}} e^{-\kappa^2|\mathbf{x} - \mathbf{b}|^2/4}\ dx_+\ dx_-
=& e^{-\kappa^2|\ba - \mathbf{b}|^2/8}.
\end{align*}

Note that $\langle p_\kappa^a, p_\kappa^b \rangle $ means integrate the product over $\bR^3$ using Lebesgue measure. The proof of Item \ref{i.l.7} is similar to the above item, so omitted.
\end{proof}

\nocite{*}



\end{document}